\documentclass[aap,preprint]{imsart}
\usepackage{amsmath, graphicx}

\RequirePackage[OT1]{fontenc}
\RequirePackage{amsthm,amsmath,amssymb}
\RequirePackage[numbers]{natbib}
\RequirePackage[colorlinks,citecolor=blue,urlcolor=blue]{hyperref}

\startlocaldefs
\numberwithin{equation}{section}
\theoremstyle{plain}
\newtheorem{theorem}{Theorem}[section]
\endlocaldefs
\startlocaldefs
\numberwithin{equation}{section}
\theoremstyle{plain}
\newtheorem{lemma}{Lemma}[section]
\endlocaldefs
\startlocaldefs
\numberwithin{equation}{section}
\theoremstyle{plain}
\newtheorem{corollary}{Corollary}[section]
\endlocaldefs
\startlocaldefs
\numberwithin{equation}{section}
\theoremstyle{plain}
\newtheorem{proposition}{Proposition}[section]
\endlocaldefs
\startlocaldefs
\numberwithin{equation}{section}
\theoremstyle{plain}
\newtheorem{definition}{Definition}[section]
\endlocaldefs

\begin{document}

\begin{frontmatter}
\title{$\varepsilon$-Strong Simulation for Multidimensional Stochastic Differential
Equations via Rough Path Analysis} 
\runtitle{$\varepsilon$%
-Strong Simulation for SDEs}

\begin{aug}
\author{\fnms{Jose} \snm{%
Blanchet}\thanksref{m1}\ead[label=e1]{jose.blanchet@columbia.edu}},
\author{%
\fnms{Xinyun} \snm{Chen}\thanksref{m2}\ead[label=e2]{%
xinyun.chen@whu.edu.cn}}
\and
\author{\fnms{Jing} \snm{Dong}\thanksref{%
m3}\ead[label=e3]{jing.dong@northwestern.edu}}

%
\runauthor{J Blanchet, X. Chen and J. Dong}

\affiliation{%
Columbia University\thanksmark{m1} and Wuhan University\thanksmark{%
m2} and Northwestern University \thanksmark{m3}}%

\address{500 West 120th Street, RM 313\\
New York, NY
10027\\
\printead{e1}}

\address{%
Economics and Management School of Wuhan University\\
Luojia Hill, Wuhan, China
\printead{%
e2}}

\address{2145 Sheridan Road, M239\\
Evanston, IL 60208\\
\printead{e3}}
\end{aug}

\begin{abstract}
Consider a multidimensional diffusion process $X=\{X\left(t\right):t\in\lbrack0,1]\}$.
Let $\varepsilon>0$ be a \textit{deterministic}, user defined, tolerance error parameter.
Under standard regularity conditions on the drift and diffusion coefficients of $X$,
we construct a probability space, supporting both $X$ and an explicit, piecewise
constant, fully simulatable process $X_{\varepsilon}$ such that
\[
\sup_{0\leq t\leq1}\left\Vert X_{\varepsilon}\left(t\right)-X\left(t\right)  \right\Vert_{\infty}<\varepsilon
\]
\textit{with probability one}. Moreover, the user can adaptively choose
$\varepsilon^{\prime}\in\left(0,\varepsilon\right)$ so that
$X_{\varepsilon^{\prime}}$ (also piecewise constant and fully simulatable)
can be constructed conditional on $X_{\varepsilon}$ to ensure an error smaller than
$\varepsilon^{\prime}$ with probability one. Our construction requires a detailed study of continuity
estimates of the It\^o map using Lyons' theory of rough paths. We approximate
the underlying Brownian motion, jointly with the L\'{e}vy areas with a deterministic
$\varepsilon$ error in the underlying rough path metric.
\end{abstract}

\begin{keyword}[class=MSC]
\kwd[Primary ]{%
34K50}
\kwd{%
65C05}
\kwd{82B80}
\kwd[; secondary ]{97K60}
\end{%
keyword}

\begin{%
keyword}
\kwd{stochastic
differential equation}
\kwd{Monte Carlo method}
\kwd{Brownian Motion}
\kwd{L\'{e}vy area}
\kwd{Rough path}
\end{keyword%
}

\end{frontmatter}

\section{Introduction}

Consider the It\^o Stochastic Differential Equation (SDE)

\begin{equation}
dX(t)=\mu (X(t))dt+\sigma (X(t))dZ(t)\mbox{ , }X(0)=x(0)  \label{eq:main}
\end{equation}%
where $Z\left( \cdot \right) $ is a $d^{\prime }$-dimensional Brownian
motion, and $\mu \left( \cdot \right) :R^{d}\rightarrow R^{d}$ and $\sigma
\left( \cdot \right) :R^{d}\rightarrow R^{d\times d^{\prime }}$ satisfy
suitable regularity conditions. We shall assume, in particular, that both $%
\mu \left( \cdot \right) $ and $\sigma \left( \cdot \right) $ are Lipschitz
continuous so that a strong solution to the SDE is guaranteed to exist.
Additional assumptions on the first and second order derivatives of $\mu
\left( \cdot \right) $ and $\sigma \left( \cdot \right) $, which are
standard in the theory of rough paths, will be discussed in the sequel.

Our contribution in this paper is the joint construction of $X=\{X\left(
t\right) :t\in \lbrack 0,1]\}$ and a family of processes $X_{\varepsilon
}=\{X_{\varepsilon }\left( t\right) :t\in \lbrack 0,1]\}$, for each $%
\varepsilon \in \left( 0,1\right) $, supported on a probability space $%
\left( \Omega ,\mathcal{F},P\right) $, and such that the following
properties hold:

\begin{enumerate}
\item[(T1)] The process $X_{\varepsilon}$ is piecewise constant, with
finitely many discontinuities in $[0,1]$.

\item[(T2)] The process $X_{\varepsilon}$ can be simulated exactly and,
since it takes only finitely many values, its path can be fully stored.

\item[(T3)] We have that with $P$-probability \textit{one}
\begin{equation}
\sup_{t\in\lbrack0,1]}\left\vert \left\vert X_{\varepsilon}\left( t\right)
-X\left( t\right) \right\vert \right\vert _{\infty}<\varepsilon.  \label{TES}
\end{equation}

\item[(T4)] For any $m>1$ and $0<\varepsilon_{m}<...<\varepsilon_{1}<1$ we
can simulate $X_{\varepsilon_{m}}$ conditional on $X_{\varepsilon_{1}}$,...,$%
X_{\varepsilon_{m-1}}$.
\end{enumerate}

We refer to the class of procedures which achieve the construction of such
family $\{X_{\varepsilon }:\varepsilon \in \left( 0,1\right) \}$ as \textit{%
Tolerance-Enforced Simulation} (TES) or $\varepsilon $-strong simulation
methods. Throughout the paper we use $||\cdot||_{\infty}$ to denote the max-norm on $\mathbb{R}^{d}$.

This paper provides the first construction of a Tolerance-Enforced
Simulation procedure for multidimensional SDEs in substantial generality.
All other TES or $\varepsilon $-strong simulation procedures up to now are
applicable to one dimensional processes or multidimensional processes with
constant diffusion matrix (i.e. $\sigma \left( x\right) =\sigma $).

Let us discuss some considerations that motivate our study. We first discuss
how this paper relates to the current literature on $\varepsilon $-strong
simulation of stochastic processes, which is a recent area of research. The
paper of \cite{Chen2013} provides the construction of $X_{\varepsilon }$
satisfying only (T1) to (T3), in one dimension. In particular, bound (\ref%
{TES}) is satisfied for a given fixed $\varepsilon _{0}=\varepsilon >0$ and
it is not clear how to jointly simulate $\left\{ X_{\varepsilon
_{m}}\right\} _{m\geq 1}$ as $\varepsilon _{m}\searrow 0$ \ applying the
technique in \cite{Chen2013}. The motivation of constructing $X_{\varepsilon
_{0}}$ for \cite{Chen2013} came from the desire to produce exact samples
from a one dimensional diffusion $X\left( \cdot \right) $ satisfying (\ref%
{eq:main}), and also assuming $\sigma \left( \cdot \right) $ constant.

The authors in \cite{Chen2013} were interested in extending the
applicability of an algorithm introduced by Beskos and Roberts, see \cite%
{BR05}. The procedure of Beskos and Roberts, applicable to one dimensional
diffusions, imposed strong boundedness assumptions on the drift coefficient
and its derivative. The technique in \cite{Chen2013} enabled an extension
which is free of such boundedness assumptions by using a localization
technique that allowed to apply the ideas behind the algorithm in \cite{BR05}%
; see also \cite{BPR_06} for another approach which eliminates boundedness
assumptions. All of these developments are in the one-dimensional case.

The assumption of a constant diffusion coefficient comes at basically no
cost in generality when considering one dimensional diffusions because one
can always apply Lamperti's (one-to-one) transformation. Such transformation
allows to recast the simulation problem to one involving a diffusion with
constant $\sigma \left( \cdot \right) $. Lamperti's transformation cannot be
generally applied in higher dimensions.

The paper of \cite{BPR_12} extends the work of \cite{Chen2013} in that their
algorithms satisfy (T1) to (T4), but also in the context of one dimensional
processes. The paper \cite{PJR_2014}\ not only provides an additional
extension which allows to deal with one dimensional SDEs with jumps, but
also contains a comprehensive discussion on exact and $\varepsilon $-strong
simulation for SDEs. Property (T4) in the definition of TES is desirable
because it provides another approach at constructing unbiased estimators for
expectations of the form $Ef\left( X\right) $, where $f\left( \cdot \right) $
is, say, a continuous function of the sample path $X$. In order to see this,
let us assume for simplicity that $f\left( \cdot \right) $ is positive and
Lipschitz continuous in the uniform norm with Lipschitz constant $K$. Then,
let $T$ be any positive random variable with a strictly positive density $%
g\left( \cdot \right) $ on $[0,\infty )$ and define%
\begin{equation}
Z:=I\left( f\left( X\right) >T\right) /g\left( T\right) .
\label{Eq:Unbiased}
\end{equation}%
Observe that
\begin{equation*}
E[Z]=E[E\left[ Z|X\right] ]=E\left[\int_{0}^{\infty }I\left( f\left( X\right)
>t\right) \frac{g\left( t\right) }{g\left( t\right) }dt\right]=E[f\left( X\right)] ,
\end{equation*}%
so $Z$ is an unbiased estimator for $Ef\left( X\right) $. Therefore, if Properties
T(1) to T(4) hold, it is possible to simulate $Z$ by noting that $f\left(
X_{\varepsilon }\right) >T+K\varepsilon $ implies $f\left( X\right) >T$ and
if $f\left( X_{\varepsilon }\right) <T-K\varepsilon $, then $f\left(
X\right) \leq T$. Since (T4) allows to keep simulating as $\varepsilon $
becomes smaller and $T$ is independent of $X_{\varepsilon }$ with a positive
density $g\left( \cdot \right) $, then one eventually is able to simulate $Z$
exactly.

The major obstacle involved in developing exact sampling algorithms for
multidimensional diffusions is the fact that $\sigma \left( \cdot \right) $
cannot be assumed to be constant. Moreover, even in the case of
multidimensional diffusions with constant $\sigma \left( \cdot \right) $,
the one dimensional algorithms developed so far can only be extended to the
case in which the drift coefficient $\mu \left( \cdot \right) $ is the
gradient of some function, that is, if $\mu \left( x\right) =\nabla v\left(
x\right) $ for some $v\left( \cdot \right) $. The reason is that in this
case one can represent the likelihood ratio $L\left( t\right) $, between the
solution to (\ref{eq:main}) and Brownian motion (assuming $\sigma =I$ for
simplicity) involving a Riemann integral as follows%
\begin{align}
L\left( t\right) & =\exp \left( \int_{0}^{t}\mu \left( X\left( s\right)
\right) dX\left( s\right) -\frac{1}{2}\int_{0}^{t}\left\Vert \mu \left(
X\left( s\right) \right) \right\Vert _{2}^{2}ds\right)  \notag \\
& =\frac{\exp \left( v\left( X\left( t\right) \right) \right) }{\exp \left(
v\left( X\left( 0\right) \right) \right) }\exp \left( -\frac{1}{2}%
\int_{0}^{t}\lambda \left( X\left( s\right) \right) ds\right) ,
\label{LR_Rie}
\end{align}%
for $\lambda \left( x\right) =\Delta v\left( x\right) +\left\vert \left\vert
\nabla v\left( x\right) \right\vert \right\vert _{2}^{2}$. The fact that the
stochastic integral can be transformed into a Riemann integral facilitates
the execution of acceptance-rejection because one can interpret (up to a
constant and using localization as in \cite{Chen2013}) the exponential of
the integral of $\lambda \left( \cdot \right) $ as the probability that no
arrivals occur in a Poisson process with a stochastic intensity. Such event
(i.e. no arrivals) can be simulated by thinning.

So, our motivation in this paper is to investigate a novel approach that
allows to study $\varepsilon $-strong simulation for multidimensional
diffusions in substantial generality, without imposing the assumption that $%
\sigma \left( \cdot \right) $ is constant or that a Lamperti-type
transformation can be applied. Given the previous discussion on the
connections between exact sampling and $\varepsilon $-strong simulation, and
the limitations of the current techniques, we believe that our results here
provide an important step in the development of exact sampling algorithms
for general multidimensional diffusions. For example, in contrast to
existing techniques, which demand $L\left( t\right) $ to be expressed in
terms of a Riemann integral as indicated in (\ref{LR_Rie}), our results here
allow to approximate directly $L\left( t\right) $ in terms of the stochastic
integral representation (and thus one does not need to assume that $\mu
\left( x\right) =$ $\nabla v\left( x\right) $). We plan to report on these
implications in future papers.

Our results already allow to obtain unbiased estimator of expectations
of sample path functionals via \eqref{Eq:Unbiased}.
However, it is noted in \cite{BPR_12} that the expected
number of random variables required to simulate $Z$ is typically infinite.
The recent paper \cite{PJR_2014} discusses via numerical examples the
practical limitations of these types of estimators. The work of \cite{RG2012}%
, also proposes unbiased estimators for the expectation of Lipschitz
continuous functions of $X(1)$ using randomized multilevel Monte Carlo.
Nevertheless, their algorithm also exhibits infinite expected termination
time, except when one can simulate the L\'{e}vy areas exactly, which
currently can be done only in the context of two dimensional SDEs using the
results in \cite{Gaines_Lyons_1994}.

The authors in \cite{BFRS_2013} also use rough path analysis for Monte Carlo
estimation, but their focus is on connections to multilevel techniques and
not on $\varepsilon $-strong simulation.

In this paper we concentrate only on what \textit{is possible to do} in
terms of $\varepsilon $-strong simulation procedures and how to enable the
use of rough path theory for $\varepsilon $-strong simulation. We shall
study efficient implementations of the algorithms proposed in a separate
paper. Other research avenues that we plan to investigate, and which
leverage off our development in this paper, involve quantification of model
uncertainty using the fact that our $\varepsilon $-strong simulation
algorithms in the end are uniform for cases with a large class of drift and diffusion
coefficients.

Finally, we note that in order to build our Tolerance-Enforced Simulation
procedure we had to obtain new tools for the analysis of L\'{e}vy areas and
associated conditional large deviations results for L\'{e}vy areas given the
increments of Brownian motion. We believe that these technical results might
be of independent interest.

The rest of the paper is organized as follows. In Section \ref%
{Section_Main_Results} we describe the two main results of the paper. The
first of them, Theorem \ref{th:main}, provides an error bound between the
solution to the SDE described in (\ref{eq:main}) and a suitable piecewise
constant approximation. The second result, Theorem \ref{th:main2}, refers to
the procedures that are involved in simulating the bounds, jointly with the
piecewise constant approximation, thereby yielding (\ref{TES}). Section \ref%
{sec:main_alg} is divided into two subsections and it builds the
elements behind the proof of Theorem \ref{th:main2}. As it turns out, one
needs to simulate bounds on the so-called H\"{o}lder norms of the underlying
Brownian motion and the corresponding L\'{e}vy areas. Section \ref{SubS_TES_1} lays out the details of
the simulation of the Brownian motion and an upper bound of its $\alpha$-H\"{o}lder norm and Section \ref{sec:sim_levy}
lays out the details of the simulation of the L\'{e}vy areas and an upper bound of its $2\alpha$-H\"{o}lder norm.
Section \ref{Sect_RDE} is also
divided in several parts, corresponding to the elements of rough path theory
required to analyze the SDE described in (\ref{eq:main}) as a continuous map
of Brownian motion under a suitable metric (described in Section \ref%
{Section_Main_Results}). While the final form of the estimates in Section %
\ref{Sect_RDE} might be somewhat different than those obtained in the
literature on rough path analysis, the techniques that we use here are
certainly standard in that literature. We have chosen to present the details
because the techniques might not be well known to the Monte Carlo simulation
community and also because our emphasis is in finding explicit constants
(i.e. bounds) that are amenable to simulation.

\section{Main Results\label{Section_Main_Results}}

Our approach consists in studying the process $X$ as a transformation of the
underlying Brownian motion $Z$. Such transformation is known as the
It\^o-Lyons map and its continuity properties are studied in the theory of
rough paths, pioneered by T. Lyons, in \cite{Lyons_1998}.
A rough path is an effective way to summarize an irregular path information.
The theory of rough paths allows to
define the solution to an SDE such as (\ref{eq:main}) in a path-by-path
basis (free of probability) by imposing constraints on the regularity of the
iterated integrals of the underlying process $Z$. Namely, integrals of the
form
\begin{equation}
A_{i,j}\left( s,t\right) =\int_{s}^{t}\left( Z_{i}\left( u\right)
-Z_{i}\left( s\right) \right) dZ_{j}\left( u\right) .  \label{A_IJ_B}
\end{equation}

The theory results in different interpretations of the solution to (\ref%
{eq:main}) depending on how the iterated integrals of $Z$ are interpreted.
In this paper, we interpret the integral in (\ref{A_IJ_B}) in the sense of
It\^o.

It turns out that the It\^o-Lyons map is continuous under a suitable $\alpha $%
-H\"{o}lder metric defined in the space of rough paths. In particular, such
metric can be expressed as the maximum of the following two quantities:
\begin{align}
||Z||_{\alpha }& :=\sup_{0\leq s<t\leq 1}\frac{||Z(t)-Z(s)||_{\infty }}{%
|t-s|^{\alpha }},  \label{alpha_H_N_Z} \\
||A||_{2\alpha }& :=\sup_{0\leq s<t\leq 1}\max_{1\leq i,j\leq d^{\prime }}%
\frac{|A_{i,j}(s,t)|}{|t-s|^{2\alpha }}.  \label{2alpha_HNA}
\end{align}%
As we shall discuss, continuity estimates of the It\^o-Lyons map can be given
explicitly in terms of these two quantities.

In the case of Brownian motion, as we consider here, we have that $\alpha
\in \left( 1/3,1/2\right) $. It is shown in \cite{Davie_2007}, that under
suitable regularity conditions on $\mu \left( \cdot \right) $ and $\sigma
\left( \cdot \right) $, which we shall discuss momentarily, the Euler scheme
provides an almost sure approximation in uniform norm to the solution to the
SDE (\ref{eq:main}). Our first result provides an explicit characterization
of all of the (path-dependent) quantities that are involved in the final
error analysis (such as $\left\vert \left\vert Z\right\vert \right\vert
_{\alpha }$ and $\left\vert \left\vert A\right\vert \right\vert _{2\alpha }$%
), the difference between our analysis and what has been done in previous
developments is that ultimately we must be able to implement the Euler
scheme jointly with the path-dependent quantities that are involved in the
error analysis. So, it is not sufficient to argue that there exists a
path-dependent constant that serves as a bound of some sort, we actually
must provide a suitable representation that can be simulated in finite time.

In order to provide our first result, we introduce some notations. Let $%
D_{n} $ denote the dyadic discretization of order $n$ and $\Delta_{n}$
denote the mesh of the discretization. Specifically, $D_{n}:=%
\{t_{0}^{n},t_{1}^{n},\dots,t_{2^{n}}^{n}\}$ where $t_{k}^{n}=k/2^{n}$ for $%
k=0,1,2,\dots,2^{n}$ and $\Delta_{n}=1/2^{n}$. 

Given $\hat{X}^{n}(0)=x(0)$, define $\{\hat{X}^{n}(t):t\in D_{n}\}$ by the
following recursion:
\begin{eqnarray}
\hat{X}_{i}^{n}(t_{k+1}^{n})&=&\hat{X}_{i}^{n}(t_{k}^{n})+\mu _{i}(\hat{X}%
^{n}(t_{k}^{n}))\Delta _{n}+\sum_{j=1}^{d^{\prime}}\sigma _{i,j}(\hat{X}%
^{n}(t_{k}^{n}))(Z_j(t_{k+1}^{n})-Z_j(t_{k}^{n})) \nonumber\\
&&+\sum_{j=1}^{d^{\prime}}\sum_{l=1}^{d}\sum_{m=1}^{d^{\prime}}\partial_l\sigma_{i,j}(\hat X^n(t_k^n))\sigma_{l,m}(\hat X^n(t_k^n))\tilde A_{m,j}^n(t_k^n,t_{k+1}^n),  \label{eq:recursion1}
\end{eqnarray}
where $\tilde A_{i,i}^n(t_k^n,t_{k+1}^n)=A_{i,i}(t_k^n,t_{k+1}^n)=(Z_i(t_{k+1}^n)-Z_i(t_k^n))^2/2-\Delta_n/2$, and $\tilde A_{i,j}^n(t_k^n,t_{k+1}^n)=0$ for $i\neq j$.
We let $\hat{X}^{n}(t)=\hat{X}^{n}(\lfloor t\rfloor )$ where $\lfloor
t\rfloor =\max \{t_{k}^{n}:t_{k}^{n}\leq t\}$ for $t\in \lbrack 0,1]$. We
denote
\begin{equation*}
R_{i,j}^{n}(t_{l}^{n},t_{m}^{n}):=
\sum_{k=l+1}^{m}\left\{A_{i,j}(t_{k-1}^{n},t_{k}^{n})-\tilde A_{i,i}^n(t_k^n,t_{k+1}^n)\right\}.
\end{equation*}%
and for fixed $\beta \in (1-\alpha ,2\alpha )$, write
\begin{equation*}
\Gamma _{R}:=\sup_{n}\sup_{0\leq s<t\leq 1,s,t\in D_{n}}\max_{1\leq i,j\leq
d^{\prime }}\frac{|R_{i,j}^{n}(s,t)|}{|t-s|^{\beta }\Delta _{n}^{2\alpha
-\beta }}.
\end{equation*}%
We notice that when $i=j$, $R_{i,i}^{n}(t_{l}^{n},t_{m}^{n})=0$; when $i\neq j$, $R_{i,j}^{n}(t_{l}^{n},t_{m}^{n})=
\sum_{k=l+1}^{m}A_{i,j}(t_{k-1}^{n},t_{k}^{n})$.
We also redefine $||Z||_{\alpha }$ and $||A||_{2\alpha }$ as
\begin{align*}
||Z||_{\alpha }& :=\sup_{n}\sup_{0\leq s<t\leq 1,s,t\in D_{n}}\frac{%
||Z(t)-Z(s)||_{\infty }}{|t-s|^{\alpha }}, \\
||A||_{2\alpha }& :=\sup_{n}\sup_{0\leq s<t\leq 1,s,t\in D_{n}}\max_{1\leq
i,j\leq d^{\prime }}\frac{|A_{i,j}(s,t)|}{|t-s|^{2\alpha }}.
\end{align*}%
The new definitions are equivalent to (\ref{alpha_H_N_Z}) and (\ref%
{2alpha_HNA}) since both $Z$ and $A$ are continuous processes. It is well
known that a solution to $X$ can be constructed path-by-path (see \cite%
{Davie_2007} and Section \ref{Sect_RDE}). The next result characterizes an
explicit bound for the error obtained by approximating $X$ using $\hat{X}%
^{n} $.

\begin{theorem}
\label{th:main} Suppose that there exists a constant $M$ such that $||\mu
||_{\infty }\leq M$, $||\nabla\mu||_{\infty }\leq M$ and $||\sigma
^{(i)}||_{\infty }\leq M$ for $i=0,1,2,3$, where $\sigma
^{(i)}$ denotes the $i$-th derivative of $\sigma$. If $\left\vert \left\vert
Z\right\vert \right\vert _{\alpha }\leq K_{\alpha }<\infty $, $\left\vert
\left\vert A\right\vert \right\vert _{2\alpha }\leq K_{2\alpha }<\infty $,
and $\Gamma _{R}<K_{R}$, we can compute $G$ explicitly in terms of $M$, $%
K_{\alpha }$, $K_{2\alpha }$ and $K_{R}$, such that
\begin{equation*}
\sup_{t\in \lbrack 0,1]}||\hat{X}^{n}(t)-X(t)||_{\infty }\leq G\Delta
_{n}^{2\alpha -\beta }.
\end{equation*}
\end{theorem}

\textbf{Remark: } A recipe that explains step-by-step how to compute $G$ in
terms of algebraic expressions involving $M,K_{\alpha },K_{2\alpha }$ and $%
K_{R}$ is given in Procedure A in the appendix to this section.

\bigskip

Using Theorem \ref{th:main}, we can proceed to state the main contribution
of this paper.

\begin{theorem}
\label{th:main2} In the context of Theorem \ref{th:main}, there is an
explicit Monte Carlo procedure that allows us to simulate random variables $%
K_{\alpha }$, $K_{2\alpha }$, and $K_{R}$ jointly with $\left\{ Z(t):t\in
D_{n}\right\} $ for any $n\geq 1$. Consequently, given any deterministic $%
\varepsilon >0$ we can select $n\left( \varepsilon \right) $ such that $%
G\Delta _{n\left( \varepsilon \right) }^{2\alpha -\beta }\leq \varepsilon $
and then set $X_{\varepsilon }\left( t\right) =\hat{X}^{n}(t)$ so that
\begin{equation}
\sup_{t\in \lbrack 0,1]}||X_{\varepsilon }(t)-X(t)||_{\infty }\leq
\varepsilon ,  \label{Bound_Eps}
\end{equation}%
with probability one.
\end{theorem}

\textbf{Remark: }An explicit description of the algorithm involved in the
Monte Carlo procedure of Theorem \ref{th:main2} is given in Algorithm II at
the end of Section \ref{Sub_TES}, and the discussion that follows it.

Given $\left\{ Z(t):t\in D_{n\left( \varepsilon \right) }\right\} $ so that (%
\ref{Bound_Eps}) holds, the discussion in the remark that follows Algorithm
II explains how to further simulate $\left\{ Z(t):t\in D_{n^{\prime
}}\right\} $ for any $n^{\prime }>n\left( \varepsilon \right) $. This
refinement is useful in order to satisfy the important property (T4) given
in the Introduction. In detail, once $K_{\alpha }$, $K_{2\alpha }$, and $%
K_{R}$ have been simulated then $G$ has also been simulated and evaluated.
Consequently, given any sequence $\varepsilon _{m}<\varepsilon
_{m-1}<...<\varepsilon _{1}$ we just need to obtain $n_{i}$ such that $%
G\Delta _{n_{i}}^{2\alpha -\beta }\leq \varepsilon _{i}$. Then simulate $%
\left\{ Z(t):t\in D_{n_{i}}\right\} $ and construct $\hat{X}^{n_{i}}(\cdot )$
according to (\ref{eq:recursion1}). We let $X_{\varepsilon _{i}}\left(
t\right) =\hat{X}^{n_{i}}(t)$ and, owing to Theorem \ref{th:main}, we
immediately obtain%
\begin{equation*}
\sup_{t\in \lbrack 0,1]}||X_{\varepsilon _{i}}\left( t\right)
-X(t)||_{\infty }\leq \varepsilon _{i}
\end{equation*}%
with probability one, as desired.

\subsection{On Relaxing Boundedness Assumptions}

The construction of $\hat{X}^{n}(\cdot )$ in order to satisfy (\ref%
{Bound_Eps}) assumes that $||\mu ||_{\infty }\leq M$, $||\mu ^{\left(
1\right) }||_{\infty }\leq M$ and $||\sigma ^{(i)}||_{\infty }\leq M$ for $%
i=0,1,2,3$. Although these assumptions are strong, here we explain how to
relax them. Theorem \ref{th:main2} extends directly to the case in which $%
\mu $ and $\sigma $ are Lipschitz continuous, with $\mu $ differentiable and
$\sigma $ three times differentiable. Since $\mu $ and $\sigma $ are
Lipschitz continuous we know that $X\left( \cdot \right) $ has a strong
solution which is non-explosive.

We can always construct $\mu_{M}$ and $\sigma_{M}$ so that $\mu^{\left(
i\right) }\left( x\right) =\mu_{M}^{\left( i\right) }\left( x\right) $ for $%
\left\Vert x\right\Vert _{\infty}\leq c_{M}$ and $i=0,1$, and $%
\sigma^{\left( i\right) }\left( x\right) =\sigma_{M}^{\left( i\right)
}\left( x\right) $ for $\left\Vert x\right\Vert _{\infty}\leq c_{M}$ for $%
i=0,1,2,3$. Also we can construct $c_{M}$, where $c_M\rightarrow\infty$ as $M \rightarrow \infty$, and $||\mu
_{M}||_{\infty}\leq M$, $||\mu_{M}^{\left( 1\right) }||_{\infty}\leq M$ and $%
||\sigma_{M}^{(i)}||_{\infty}\leq M$ for $i=0,1,2,3$.

For $M\geq1$ we consider the SDE (\ref{eq:main}) with $\mu_{M}$ and $%
\sigma_{M}$ as drift and diffusion coefficients, respectively, and let $%
X_{M}\left( \cdot\right) $ be the corresponding solution to (\ref{eq:main}).
We start by picking some $M_{0}\geq1$ such that $\varepsilon<c_{M_{0}}$ and
let $M=M_{0}$. Then run Algorithm II to produce $\{\hat{X}%
_{M}^{n}(t):t\in\lbrack0,1]\}$, which according to Theorem \ref{th:main2}
satisfies,%
\begin{equation*}
\sup_{t\in\lbrack0,1]}||\hat{X}_{M}^{n}(t)-X_{M}(t)||_{\infty}\leq%
\varepsilon.
\end{equation*}
Note that only Steps 5 to 8 in Algorithm II\ depend on the SDE (\ref{eq:main}%
), through the evaluation of $G$, which depends on $M$ and so we write $%
G_{M}:=G$. If $\sup_{t\in\lbrack0,1]}||\hat{X}_{M}^{n}(t)||_{\infty}\leq
c_{M}-\varepsilon$, then we must have that $X\left( t\right) =X_{M}\left(
t\right) $ for $t\in\lbrack0,1]$ and we are done. Otherwise, we let $%
M\longleftarrow2M$ and run again only Steps 5 to 8 of Algorithm II. We
repeat doubling $M$ and re-running Steps 5 to 8 (updating $G_{M}$) until we
obtain a solution for which $\sup_{t\in\lbrack0,1]}||\hat{X}%
_{M}^{n}(t)||_{\infty}\leq c_{M}-\varepsilon$. Eventually this must occur
because
\begin{equation*}
\lim_{M\rightarrow\infty}\sup_{t\in\lbrack0,1]}||X_{M}(t)-X(t)||_{\infty}=0
\end{equation*}
almost surely and $X\left( \cdot\right) $ is non explosive.

\bigskip

\subsection{The Evaluation of $G$} \label{Section_Computing_G}

We next summarize the way to calculate $G$ in terms of $M$, $K_{\alpha}$, $%
K_{2\alpha}$ and $K_{R}$. We write $\bar d=\max\{d, d^{\prime}\}$. \newline

\noindent\textbf{Procedure A.}

\begin{enumerate}
\item Find $\delta$ and $C_{i}(\delta)>0$ for $i=1,2,3$ that satisfies the
following relations:
\begin{align*}
C_{1}(\delta) \geq& C_{3}(\delta)\delta^{2\alpha}+M\delta^{1-\alpha
}+\bar dMK_{\alpha}+\bar d^{3}M^{2}K_{2\alpha}\delta^{\alpha} \\
C_{2}(\delta) \geq& C_{3}(\delta)\delta^{\alpha}+\bar d^{3}M^{2}K_{2\alpha }
\\
C_{3}(\delta) \geq& \frac{2}{1-2^{1-3\alpha}}\{ MC_{1}(\delta
)+\bar dMC_{1}(\delta)^{2}K_{\alpha}+\bar d^{2}MC_{2}(\delta)K_{\alpha} \\
&+2\bar d^{3}M^{2}C_{1}(\delta)K_{2\alpha}\}
\end{align*}
(Refer to the proof of Lemma \ref{lm:tech1} for one particular method to find such $C_i(\delta)$'s.)
\item Set $C_{1}=\frac{2}{\delta}C_{1}(\delta)$, $C_{2}=\frac{2}{\delta}%
(C_{2}(\delta)+MC_{1}+\bar dMC_{1}K_{\alpha})$ and
\begin{equation*}
C_{3} = \frac{2}{1-2^{1-3\alpha}}(MC_{1}+\bar dMC_{1}^{2}K_{%
\alpha}+\bar d^{2}MC_{2}K_{\alpha}+2\bar d^{3}M^{2}C_{1}K_{2\alpha})
\end{equation*}

\item Find $\delta^{\prime}$ and $B_{i}(\delta^{\prime})$ for $i=1,2,3$ that
satisfies the following relations:
\begin{align*}
B_{1}(\delta^{\prime}) >& B_{3}(\delta^{\prime})\delta^{\prime2\alpha
}+2M\delta^{\prime1-\alpha}+2MK_{\alpha}+4M^{2}K_{2\alpha}\delta^{%
\prime\alpha} \\
B_{2}(\delta^{\prime}) >& B_{3}(\delta^{\prime})\delta^{\prime\alpha}+
4M^{2}K_{2\alpha} \\
B_{3}(\delta^{\prime}) >& \frac{4}{1-2^{1-3\alpha}}\{
MB_{1}(\delta^{\prime})+MB_{1}(\delta^{\prime2})K_{\alpha}+MB_{2}(%
\delta^{\prime })K_{\alpha} \\
&+2M^{2}B_{1}(\delta^{\prime})K_{2\alpha}\}
\end{align*}

\item Set $B=\frac{2}{\delta^{\prime}}B_{1}(\delta^{\prime})$

\item Set $G_{1}=(1+B)C_{3}$

\item Find $\delta^{\prime\prime}$ and $C_{4}(\delta^{\prime\prime})$ such
that
\begin{align*}
B\delta^{\prime\prime\alpha} & \leq2^{\alpha+\beta}-2 \\
C_{4}(\delta^{\prime\prime}) & \geq2(1-\frac{2+B\delta^{\prime\prime\alpha}}{%
2^{\alpha+\beta}})^{-1}(B\bar d^{3}M^{2}K_{R}+2\bar d^{3}M^{2}C_{1}K_{R})
\end{align*}

\item Set $C_{4}=(1+B\delta^{\prime\prime\alpha})C_{4}(\delta^{\prime\prime
3}M^{2}K_{R}+2\bar d^{3}M^{2}C_{1}K_{R})/\delta^{\prime\prime}$

\item Set $G_{2}=C_{4}+\bar d^{3}M^{2}K_{R}$

\item Set $G=G_{1}+G_{2}$
\end{enumerate}

\begin{lemma} \label{lm:pA}
Given $K_{\alpha}$, $K_{2\alpha}$, $K_R$ and $M$, Procedure $A$ can be executed.
\end{lemma}

\begin{proof}
We prove the lemma by providing one particular method to find such $\delta$ and $C_i(\delta)$'s, $i=1,2,3$. The method to find $\delta^{\prime}$, $B_i(\delta^{\prime})$'s, for $i=1,2,3$, follows exactly the same rationale.

Set $C_1(\delta)=\bar dM||Z||_{\alpha}+1/2$, $C_2(\delta)=\bar d^3 M^2||A||_{2\alpha}+1/2$ and $C_3(\delta)=\frac{2}{1-2^{1-3\alpha}}(  MC_{1}(\delta)+\bar dMC_{1}(\delta)^{2}||Z||_{\alpha}+\bar d^{2}MC_{2}(\delta)||Z||_{\alpha} +\bar d^2M^2||Z||_{\alpha}+2\bar d^{3}M^{2}C_{1}(\delta)||A||_{\alpha})$. Then we can pick $\delta$ small enough, such that  $C_{3}(\delta)\delta^{2\alpha}+M\delta^{1-\alpha}+\bar d^{3}M^{2}||A||_{2\alpha}\delta^{\alpha}<1/2$ and $C_3(\delta)\delta^{\alpha}<1/2$.
\end{proof}

\section{The main idea of the algorithmic development} \label{sec:main_alg}
Based on Theorem \ref{th:main2}, our main task is to calculate/simulate the upper bound for
$||Z||_{\alpha}$, $||A||_{2\alpha}$ and $\Gamma_{R}$ respectively.
In this section, we will introduce the main idea of our algorithmic development.

The development can be decomposed into two tasks. The first one is to find an infinite sum representation
of the objects of interest. The second one is to truncate the infinite sum up to a finite but random
level so that the error induced by the remaining terms in the summation is suitably controlled. 
The second task calls for novel algorithmic constructions. Simulating infinitely many terms is impossible. 
We need to find an efficient way to extract enough information on the remaining terms after the truncation, 
so that we can obtain an almost sure bound on the contribution of the terms that are not simulated. 
We next carry out the two tasks one by one.

\subsection{Infinite sum representation of Brownian motion and L\'{e}vy area}
We start by introducing a wavelet synthesis of Brownian motion, $\{Z(t): 0\leq t\leq 1\}$, called the L\'{e}vy-Ciesielski construction of Brownian
motion (\citet{Steele_2001}).

First we need to define a step function $H(\cdot)$ on $[0,1]$ by
\begin{equation*}
H(t)=I\left( 0\leq t<1/2\right) -I\left( 1/2\leq t\leq1\right) .
\end{equation*}
We then define a family of functions
\begin{equation*}
H_{k}^{n}(t)=2^{n/2}H(2^{n-1}t-k+1)
\end{equation*}
for all $n\geq0$ and $1\leq k\leq2^{n-1}$. Set $H_{0}^{0}(t)=1$. Then one
obtains the following infinite sum representation of Brownian motion.

\begin{theorem}[L\'{e}vy-Ciesielski Construction]
If $\{W_{k}^{n}:1\leq k\leq 2^{n-1},n\geq0\}$ is a sequence of independent
standard normal random variables, then the series defined by
\begin{equation}
Z\left( t\right) =W_0^0\int_{0}^{t}H_{0}^{0}(s)\;ds+\sum_{n=1}^{\infty}\sum_{k=1}^{2^{n-1}}\left(
W_{k}^{n}\int_{0}^{t}H_{k}^{n}(s)\;ds\right)  \label{wav_syn_eq}
\end{equation}
converges uniformly on $[0,1]$ with probability one. Moreover, the process $%
\{Z\left( t\right) :t\in\lbrack0,1]\}$ is a standard Brownian motion on $%
[0,1]$.
\end{theorem}

Figure \ref{fig:wavelet} demonstrates the basic idea of the L\'{e}vy-Ciesielski
Construction using properties of the Brownian bridge. Specifically, as $Z(1)\sim N(0,1)$,
we set $Z(1)=W_0^0$. Conditional on the value of $Z(0)=0$ and $Z(1)$, $Z(1/2) \sim N(Z(1)/2, 1/4)$. Thus we
set $Z(1/2)=Z(1)/2+ 1/2 W_1^1$. In general, conditional on the value of $Z(t_k^{n-1})$ and $Z(t_{k+1}^{n-1})$, for $k=0,1,\dots, 2^{n-1}$,
$$Z(t_{2k+1}^{n})\sim N\left(\left(Z(t_k^{n-1})+Z(t_{k+1}^{n-1})\right)/2, \Delta_{n+1}\right)$$
Thus we set
$$Z(t_{2k+1}^{n})=\left(Z(t_k^{n-1})+Z(t_{k+1}^{n-1})\right)/2+\Delta_{n+1}^{1/2}W_{k+1}^{n}.$$
\begin{figure}
\vspace{6pc}
\caption[]{L\'{e}vy-Ciesielski Construction of Brownian Motion on $[0,1]$} \label{fig:wavelet}
\centering
\includegraphics[width=0.9\textwidth]{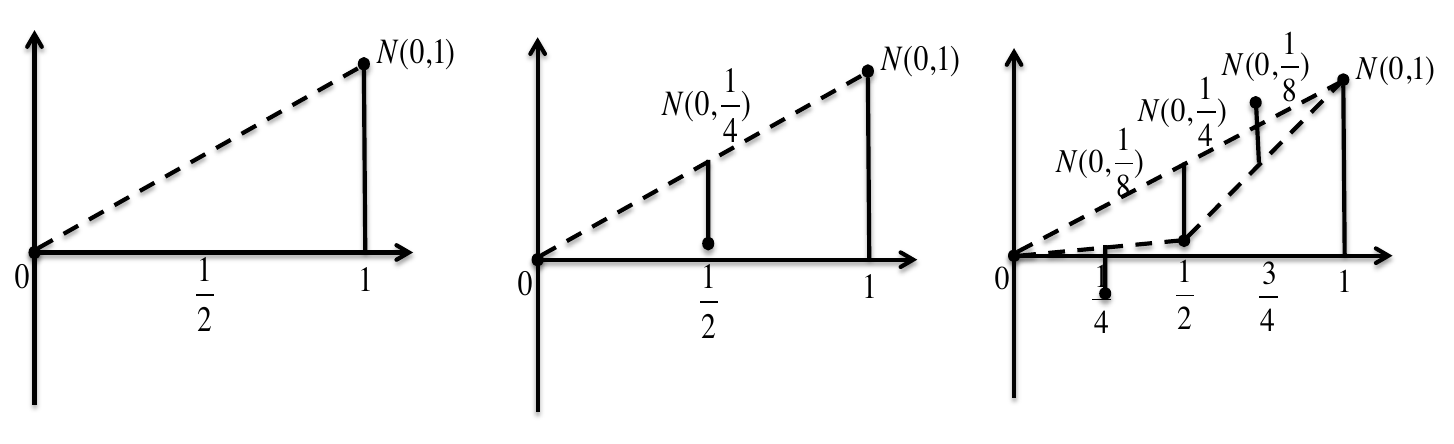}
\end{figure}

Eventually we will simulate the series up to a finite but random level $N_1$ to be discussed later.
By level we mean the order of dyadic discretization.
As we are simulating the discretization levels sequentially, we often
refer to \textquotedblleft time\textquotedblright\ when discussing levels.\\

We next analyze the L\'{e}vy area, $A_{i,j}(t_k^n, t_{k+1}^n)$, for $1\leq i,j \leq d^{\prime}$, $n \geq 1$, $0\leq k\leq 2^n-1$.
Using the algebraic property
\begin{align*}
A_{i,j}\left(  t_{k}^{n},t_{k+1}^{n}\right) = & A_{i,j}\left(  t_{2k}%
^{n+1},t_{2k+1}^{n+1}\right)  +A_{i,j}\left(  t_{2k+1}^{n+1},t_{2k+2}%
^{n+1}\right)\\
 &+\left(  Z_{i}\left(  t_{2k+1}^{n+1}\right)  -Z_{i}\left(
t_{2k}^{n+1}\right)  \right)  \left(  Z_{j}\left(  t_{2k+2}^{n+1}\right)
-Z_{j}\left(  t_{2k+1}^{n+1}\right)  \right) ,
\end{align*}
we have the following infinite sum representation of $A_{i,j}(t_k^n, t_{k+1}^n)$.

\begin{lemma}\label{Lem_Rep_Area}
For $n \geq 1$, $0\leq k\leq 2^n-1$,
\begin{align*}
A_{i,j}(t_{k}^{n},t_{k+1}^{n})=\sum_{h=n+1}^{\infty}\sum_{l=1}^{2^{h-n-1}}%
\{&\left(Z_{i}(t_{2^{h-n}k+2l-1}^{h})-Z_{i}(t_{2^{h-n}k+2l-2}^{h})\right) \\
&\times\left(Z_{j}(t_{2^{h-n}k+2l}^{h})-Z_{j}(t_{2^{h-n}k+2l-1}^{h})\right)%
\}.
\end{align*}
\end{lemma}

The inner summation terms in the expression for $A_{i,j}(t_{k}^{n},t_{k+1}^{n})$
motivate the definition of the following family of processes $%
(L_{i,j}^{n}\left( k\right) :k=0,1,...,2^{n-1}, n \geq 1)$.
\begin{align*}
& L_{i,j}^{n}(0):=0 \\
&
L_{i,j}^{n}(k):=L_{i,j}^{n}(k-1)+(Z_{i}(t_{2k-1}^{n})-Z_{i}(t_{2k-2}^{n}))(Z_{j}(t_{2k}^{n})-Z_{j}(t_{2k-1}^{n}))
\end{align*}
for $k=1,2,...,2^{n-1}$.\newline
Using this definition and Lemma \ref{Lem_Rep_Area}\ we can succinctly write $%
A_{i,j}(t_{k}^{n},t_{k+1}^{n})$ as%
\begin{equation}
A_{i,j}(t_{k}^{n},t_{k+1}^{n})=\sum_{h=n+1}^{%
\infty}(L_{i,j}^{h}(2^{h-n}(k+1))-L_{i,j}^{h}(2^{h-n}k)).  \label{Usef_Rep_A}
\end{equation}

\subsection{The idea of record breakers}
To truncate the infinite sum up to a finite but random level, we use a strategy called record breakers.
Specifically, we first define a sequence of
\textquotedblleft record breakers\textquotedblright. We then formulate the ``future" information we need to know
as a sequence of \textquotedblleft yes or no\textquotedblright\
questions. Specifically, the yes or no question is formulated as
``will there be a new record breaker?" and answering the yes/no question is equivalent to simulating
a properly defined Bernoulli random variable.

The definition of the record breakers need to
satisfy the following two conditions:\\

\begin{enumerate}
\item[C1.] The following event happens with probability one: beyond some random
but finite time, there will be no more record breakers.

\item[C2.] By knowing that there are no more record breakers, the contribution of
the terms that we have not simulated yet are well under control (i.e.
bounded by a user defined tolerance error).\\
\end{enumerate}

We next explain how the above strategy is applied to the Brownian motion and the L\'{e}vy area respectively.

We have $d^{\prime}$ independent Brownian motions and we will use $W_{i,k}^{n}$ for $%
i\in\{1,...,d^{\prime}\}$ to denote the $\left( n,k\right) $ coefficient in
the expansion (\ref{wav_syn_eq}) for the $i$-th Brownian motion.

For $||Z||_{\alpha}$, we say a record is broken at $(i,n,k)$, for $1\leq
i\leq d^{\prime}$, $n\geq0$ and $1\leq k\leq2^{n-1}$, if
\begin{equation*}
|W_{i,k}^{n}|>4\sqrt{n+1}.
\end{equation*}
Let $\bar N_{1}:=\max\{n\geq1:|W_{i,k}^{n}|>4\sqrt{n+1}\mbox{ for some }1\leq
k\leq2^{n-1}, 1\leq i\leq d^{\prime}\}$. It is the last time the record breaker happens.
The following Lemma shows that condition C1 is satisfied ($E[N_1]<\infty$ implies $P(N_1<\infty)=1$).
\begin{lemma} \label{lm:N1_finite_mean}
There exists an integer valued random variable $N_1$, with $E[N_1]<\infty$, such that for all $n > N_1$, $1\leq k \leq 2^n-1$ and $1 \leq i \leq d^{\prime}$
$W_{i,k}^n\leq 4\sqrt{n+1}.$
\end{lemma}

We next check condition C2. Define $V^{n}=\max_{1\leq k\leq2^{n-1}}|W_{k}^{n}|$. We have the following
auxiliary lemma.

\begin{lemma}
\label{Lem_HN_alpha}%
\begin{equation*}
\Vert Z\Vert_{\alpha}\leq2^{2\alpha+1}\sum_{n=0}^{\infty}2^{-n(\frac{1}{2}%
-\alpha)}V^{n}.
\end{equation*}
\end{lemma}

Once we found $N_{1}$, we have
\begin{align*}
||Z||_{\alpha} &\leq  2^{2\alpha+1}\sum_{n=0}^{N_{1}}2^{-n(1/2-\alpha
)}V^{n}+2^{2\alpha+3}\sum_{n=N_{1}+1}^{\infty}2^{-n(1/2-\alpha)}\sqrt{%
n+1} \\
& \leq 2^{2\alpha+1}\sum_{n=0}^{ N_{1} }2^{-n(1/2-\alpha)}V^{n}+2^{2\alpha+3}C\frac{%
2^{-1/2(N_1+1)(1/2-\alpha)}}{1-2^{-1/2(1/2-\alpha )}}.
\end{align*}
where $C=\max_{n\geq N_1+1}\{2^{-n/2(1/2-\alpha)}\sqrt{n+1}\}$.\\

For the L\'{e}vy area, we first notice that when $i=j$,
\begin{eqnarray*}
&&\sup_{n}\sup_{0\leq s<t\leq 1, s,t\in D_n}\frac{A_{i,i}(s,t)}{(t-s)^{2\alpha}}\\
&=&\sup_{n}\sup_{0\leq s<t\leq 1, s,t\in D_n}\frac{(B(t)-B(s))^2-(t-s)^2}{2(t-s)^{2\alpha}}\\
&\leq&\frac{||Z||_{\alpha}^{2}+1}{2},
\end{eqnarray*}
and
$$R_{i,i}^n(t_l^n,t_m^n)=0.$$
When $i\neq j$, the record breaker is defined for the random walk $L_{i,j}^n$'s.
Specifically, for $L$, we say a record is broken at $(i,j,n,k,k^{\prime})$, for $1\leq i,j\leq d^{\prime}$, $i\neq j$, $n\geq1, 0\leq k <
k^{\prime}<2^{n-1}$, if
\begin{equation*}
|L_{i,j}^{n}(k^{\prime})-L_{i,j}^{n}(k)|>(k^{\prime}-k)^{\beta}\Delta_{n}^{2\alpha},
\end{equation*}
where $\beta\in (1-\alpha,2\alpha)$.
Let $\bar N_{2}:=\max\{n \geq1: |L_{i,j}^{n}(k^{\prime})-L_{i,j}^{n}(k)| >
(k^{\prime}-k)^{\beta}\Delta_{n}^{2\alpha} \mbox{ for some } 0 \leq k < k^{\prime}
\leq2^{n-1}, 1\leq i,j \leq d^{\prime}, i\neq j\}$. It is the last time the record breaker happens.
The following lemma shows that condition C1 is satisfied.

\begin{lemma}
\label{Lem_L_Abs_Sum}There exists an integer valued random variable $N_{2}$, with $E[N_2]=o\left((1-2\alpha)^{-2}\right)$, such that for all $n >
N_{2}$ and all $0\leq l<m\leq2^{n-1}$ we have
$|L_{i,j}^{n}(m)-L_{i,j}^{n}(l)|\leq(m-l)^{\beta}\Delta_{n}^{2\alpha}$
for $\alpha \in (1/3, 1/2)$ and $\beta\in (1-\alpha, 2\alpha)$.
\end{lemma}

We next check condition C2. The following corollary follows directly from \eqref{Usef_Rep_A} and the definition of $R_{i,j}^n$.
\begin{corollary}
\label{Prop_Rep_R}%
For $i\neq j$,
\begin{equation*}
R_{i,j}^{n}(t_{l}^{n},t_{m}^{n})=\sum_{h=n+1}^{%
\infty}\left(L_{i,j}^{h}\left(2^{h-n}m\right)-L_{i,j}^{h}\left(2^{h-n}l\right)\right).
\end{equation*}
\end{corollary}

Then we have the following bounds for $||A||_{2\alpha}$ and $\Gamma_R$ based on the $N_2$.

\begin{lemma}
\label{Lem_Sum_Bounds} Suppose that $N_{2}$ is chosen according to Lemma \ref%
{Lem_L_Abs_Sum}. We define
\begin{equation*}
\Gamma_{L}:=\max\left\{1,\max_{1\leq i,j\leq d^{\prime}, i\neq j}\max_{n<N_{2}}\max_{0\leq l < m
\leq2^{n-1}}\left\{ \frac{|L_{i,j}^{n}(m)-L_{i,j}^{n}(l)|}{(m-l)^{\beta
}\Delta_{n}^{2\alpha}}\right\} \right\}.
\end{equation*}
Then
\begin{equation*}
\Gamma_{R} \leq\frac{2^{-(2\alpha-\beta)}}{1-2^{-(2\alpha-\beta)}}\Gamma_{L}
\end{equation*}
and
\begin{equation*}
||A||_{2\alpha} \leq\Gamma_{R} \frac{2}{1-2^{-2\alpha}} + ||Z||_{\alpha}^{2}%
\frac{2^{1-\alpha}}{1-2^{-\alpha}}.
\end{equation*}
\end{lemma}

In what follows, we shall explain how to simulate the random numbers
($N_{1}$ and $N_{2}$) jointly with the wavelet construction using the
``record breaker" strategy introduced in the previous section. Specifically, we first find all
the record breakers in sequence and then simulate the rest of the process
conditional on the information obtained by knowing the location of all the (finitely many) record breakers. The challenge lies in the fact
that the probability of success of the Bernoulli trials, which corresponds to the yes/no questions defined
in terms of the record breakers, is not known to us.
We start with the procedure
to simulate $N_1$ in Section \ref{SubS_TES_1}, which is built on a sandwiching idea. Then conditional on the value of $N_1$, we introduce the procedure to
simulation $N_2$ in Section \ref{sec:sim_levy} based on an acceptance-rejection scheme, where the proposal distribution is built on
some exponential tilting.

\section{Tolerance-Enforced Simulation of Bounds on $\protect\alpha$-H%
\"{o}lder Norms \label{SubS_TES_1}}
We first note that $N_{1}$ is not a stopping time
with respect to the filtration generated by $\{(W_{i,k}^{n}:0\leq k\leq
2^{n}-1, 1\leq i \leq d^{\prime}):n\geq1\}$.

For the simplicity of demonstration, we shall focus on the 1-dimensional case.
For $d^{\prime}>1$, we apply the same procedure for each Brownian motion.
In what follows in this subsection, we shall drop the subscription $i$.

We call a pair $\left( n,k\right)$ a record-broken-pair if $\left\vert
W_{k}^{n}\right\vert >4\sqrt{n+1}$. All pairs (both
record-broken-pairs and non record-broken-pairs) can be totally ordered
lexicographically, i.e. using $2^{n-1}+k$. The distribution of subsequent pairs at which records are
broken is not difficult to compute (because of the independence of $%
W_{k}^{n}$'s). So, using a sequential acceptance / rejection procedure we
can simulate all of the record-broken-pairs. Conditional on these pairs, the
distribution of the $\{(W_{k}^{n}:0\leq k\leq2^{n}-1):n\geq1\}$ is
straightforward to describe. Precisely, if $\left( k,n\right) $ is a
record-broken-pair, then $W_{k}^{n}$ is conditioned on $\left\vert
W_{k}^{n}\right\vert >4\sqrt{n+1}$, and thus is straightforward to simulate.
Similarly, if $\left( k,n\right) $ is not a record-broken-pair, then $%
W_{k}^{n}$ is conditioned on $\left\vert W_{k}^{n}\right\vert \leq4\sqrt {n+1%
}$, and also can be easily simulated.

The simulation of the record-broken-pairs has been studied in \cite%
{BlanchetChen_2013}.
The idea is to find all the record breakers sequentially until there are no more record breakers.
The challenge lies in sampling the Bernoulli random variable corresponding to the question ``whether there will be no more record breakers in the future".
We take sampling the first breaker as an example. The probability that there are no more record breakers
beyond $1$ is
$$p(1):=\prod_{n=1}^{\infty}\prod_{k=0}^{2^n-1}P\left(|W_k^n|\leq 4\sqrt{n+1}\right),$$
which involves evaluating the product of infinite many terms and we do not know its value in closed form.
However, we can find a sequence of upper bound and lower bounds of $p(1)$, which are defined as
$$U_h(1)=\prod_{r=1}^{h}P\left(|W_{k(r)}^{n(r)}|\leq 4\sqrt{\lfloor\log_2 r\rfloor+1}\right)$$
where $r=2^{n(r)-1}+k(r)$ and
$$D_h(1)=(1-h^{1-4^2/2})U_h$$
respectively. The upper and lower bounds satisfy that $D_h(1)<D_{h+1}(1)<p(1)<U_h(1)<U_{h+1}(1)$
and $\lim_{h\rightarrow\infty}D_h(1)=p(1)=\lim_{h\rightarrow\infty}U_h(1)$.
We also have that $U_h(1)-U_{h+1}(1)$ is equal to the probability that the first record breaker happens at position $h$.
Thus we can check whether the Bernoulli trial is a success or failure by updating the upper and lower bounds sequentially.
Moreover, if the Bernoulli trial is a failure (there are more record breakers beyond the current index),
we also know the index of the next record breaker.
We synthesize algorithm 2W in \cite{BlanchetChen_2013} for our purposes next.\newline

\noindent\textbf{Algorithm I: Simulate $N_{1}$ jointly with the
record-broken-pairs}


\textbf{Output}: A vector $S$ which gives all the indices $l=2^{n}+k$ such
that $(n,k)$ is a broken-record-pair.

\textbf{Step 0}: Initialize $R=0$ and $S$ to be an empty array.

\textbf{Step 1}: Set $U=1$, $D=0$. Simulate $V\sim$ Uniform$(0,1)$.

\textbf{Step 2}: While $U>V>D$, set $R\leftarrow R+1$ and $U\leftarrow
P(|W_{k}^{n}|\leq4\sqrt{\lfloor\log_2{R}\rfloor+1})\times U$ and $D\leftarrow$ $%
(1-R^{1-4^2/2})\times U$.

\textbf{Step 3}: If $V\geq U$, add $R$ to the end of $S$, i.e. $S=[S,R]$,
and return to Step 1.

\textbf{Step 4}: If $V\leq D$, $N_1=\lceil\log_2\max(S)\rceil$.

\textbf{Step 5}: Output $S$.\\
\noindent{\bf End of Algorithm I}\\

\textbf{Remark}: Observe that for every $l=2^{n-1}+k\in S$, we can generate $%
W_{k}^{n}$ conditional on the event $\{|W_{k}^{n}|>4\sqrt{n+1}\}$;
for other $l$ (i.e. $l\notin S$), generate $W_{k}^{n}$ given $%
\{|W_{k}^{n}|\leq4\sqrt{n+1}\}$. Note that at the end of Algorithm
1 and after simulating $W_{k}^{n}$ for $n \leq N_1$ one can compute
$$K_{\alpha}=2^{2\alpha+1}\sum_{n=0}^{ N_{1} }2^{-n(1/2-\alpha)}V^{n}+2^{2\alpha+3}C\frac{%
2^{-1/2(N_1+1)(1/2-\alpha)}}{1-2^{-1/2(1/2-\alpha )}},$$
where $C=\max_{n\geq N_1+1}\{2^{-n/2(1/2-\alpha)}\sqrt{n+1}\}$.

\section{Tolerance-Enforced Simulation for Bounds on $2\protect%
\alpha $-H\"{o}lder Norms of L\'{e}vy Areas} \label{sec:sim_levy}
The simulation of $N_{2}$, is a
lot more complicated, comparing to $N_1$, because there is fair amount of dependence on the
structure of the $L_{i,j}^{n}\left( k\right) $'s as one varies $n$. Let us
provide a general idea of our simulation procedure in order to set the stage
for the definitions and estimates that must be studied first.

Define
\begin{equation*}
\mathcal{F}_{n}\mathcal{=}\sigma\left\{(W_{i,k}^{m}:1\leq k\leq2^{m-1}):m\leq n\right\}.
\end{equation*}
and for the conditional expectation given $\mathcal{F}_{n}$ we write%
\begin{equation*}
E_{n}[\text{ }\cdot\text{ }]:=E[\text{ }\cdot\text{ }|\text{ }\mathcal{F}%
_{n}].
\end{equation*}
Suppose we have simulated $\{(W_{i,k}^{n}:0\leq k\leq2^{n}-1, 1\leq i \leq d^{\prime}):n\leq N\}$ for some
$N$ and define
\begin{eqnarray*}
\tau_{1}\left( N\right) =\inf\{n\geq
N+1&:& |L_{i,j}^{n}(m)-L_{i,j}^{n}(l)|>(m-l)^{\beta}\Delta_{n}^{2\alpha}\\
&&\text{ for some }0\leq l<m\leq2^{n-1}\}.
\end{eqnarray*}
Because of Lemma \ref{Lem_L_Abs_Sum} we have that the event $%
\{\tau_{1}\left( N\right) =\infty\}$ has positive probability. In what follows, we will
explain how to simulate a Bernoulli random variable with probability of success $%
P(\tau_{1}\left( N\right) =\infty|\mathcal{F}_{N})$. If such Bernoulli
is a success, then we have that $N_{2}=N$ and we would have
basically concluded the difficult part of the simulation procedure (the rest
of the process can be simulated under a series of conditioning events whose probability
increases to one as $n$ grows). If the Bernoulli is a
failure (i.e.\ its value is zero), then we will find $\tau_1({N})$ and simulte all the information up to $\tau_1(N)$.
We repeat the above Bernoulli trial with updated probability of success until we obtain a successful Bernoulli trial.

Now, part of the problem is that Algorithm I has been already executed, so $%
N\geq N_{1}$, in other words, while the random variables $%
\{W_{i,k}^{n}:1\leq k\leq2^{n-1}\}$ are independent (for fixed $n>N$), they are
no longer identically distributed. Instead, $W_{i,k}^{n}$ is standard
Gaussian conditional on the event $\{|W_{i,k}^{n}|\leq4\sqrt{n+1}\}$.
Nevertheless, if $n$ is large enough, all of the events $\{|W_{i,k}^{n}|\leq4%
\sqrt{n+1}\}$ will occur with high probability. So, we shall first proceed
to explain how to simulate a Bernoulli random variable with probability of
success $P(\tau_{1}\left( n^{\prime}\right) =\infty|\mathcal{F}%
_{n^{\prime}}) $ assuming $n^{\prime}$ is a deterministic number. The
procedure actually will produce both the outcome of the Bernoulli trial and
if such outcome is a failure (i.e. $\tau_{1}\left( n^{\prime}\right) <\infty$%
), also the sample path
\begin{equation*}
\{W_{i,k}^{m}:1\leq k\leq2^{m-1},n^{\prime}<m\leq\tau_{1}\left( n^{\prime
}\right) \}.
\end{equation*}

Our procedure is based on acceptance / rejection using a carefully chosen
proposal distribution for the $W_{i,k}^{n}$'s, $n\geq n^{\prime}$ based on exponential tilting
of $L_{i,j}^{n}\left( k\right) $'s, conditional on $\mathcal{F}%
_{n^{\prime}}$. To this end, we will need to compute the conditional moment generating function (conditional on $\mathcal{F}_{n^{\prime}}$) of $%
L_{i,j}^{n}\left( k\right)$'s and the family of distributions induced over
$W_{i,k}^{n}$'s and $W_{j,k}^{n}$'s under the exponentially tilting.
This will be done in Section \ref{Sub_CMGF_Tilting}.
Then, we need some large deviation estimates
to bound the likelihood ratio of a certain randomization procedure. These
bounds are developed in Section \ref{Sub_Cond_LD}. These are the main elements
needed to simulate $N_2$ together with the wavelet construction.  We introduce the actual randomization
procedure and the details of the algorithm in Section \ref{Sub_TES}.

\subsection{Conditional Moment Generating Functions and
Associated Exponential Tilting\label{Sub_CMGF_Tilting}}

In this section we characterize the distribution of $\{(W_{i,k}^{n+m}:1\leq
k\leq2^{n+m-1}): m\geq 1\}$ under the exponential tilting conditional on $\mathcal{F}_n$.

In order to reduce the length of some of the equations that follow, we
write, for each $r\in\{1,2,...,2^{n}\}$,
\begin{equation}
\Lambda_{i}^{n}(t_{r}^{n}):=Z_{i}(t_{r}^{n})-Z_{i}(t_{r-1}^{n}).
\label{Eq:INC_1}
\end{equation}
Then we have the following recursive relations for $\Lambda_i^n(t_r^n)$'s.

\begin{lemma} \label{lm:lambda}
For $k=1,2,....,2^{n-1}$
\begin{align*}
\Lambda_{i}^{n}(t_{2k-1}^{n}) & =\frac{1}{2}\Lambda_{i}^{n-1}(t_{k}^{n-1})+%
\Delta_{n+1}^{1/2}W_{i,k}^{n}. \\ 
\Lambda_{i}^{n}(t_{2k}^{n}) & =\frac{1}{2}\Lambda_{i}^{n-1}(t_{k}^{n-1})-%
\Delta_{n+1}^{1/2}W_{i,k}^{n},  
\end{align*}
\end{lemma}

From Lemma \ref{lm:lambda}, we can see that%
$$\mathcal{F}_n=\sigma\left\{Z(t_{k^{\prime}}^m)-Z(t_k^m): 0\leq k<k^{\prime}\leq2^{m-1}, m \leq n\right\}.$$

Assume that $k<k^{\prime}$, we will iteratively compute the conditional moment generating function as%
\begin{align}
& E_{n}\left[\exp\left( \theta_{0}\left\{L_{i,j}^{n+m}\left( k^{\prime}\right)
-L_{i,j}^{n+m}\left( k\right) \right\}\right) \right]  \label{Eq_iterative_exp} \\
& =E_{n}\left[E_{n+1}\left[...E_{n+m-1}\left[\exp\left( \theta_{0}\left\{L_{i,j}^{n+m}\left(
k^{\prime}\right) -L_{i,j}^{n+m}\left( k\right) \right\}\right) \right]...\right]\right].  \notag
\end{align}
Recall that, for $1\leq k\leq2^{n-1}$,%
\begin{equation*}
L_{i,j}^{n}\left( k\right) =\sum_{r=1}^{k}\Lambda_{i}^{n}\left(
t_{2r-1}^{n}\right) \Lambda_{j}^{n}\left( t_{2r}^{n}\right) .
\end{equation*}
We shall start from the expectation of $\exp\left(\theta_{0}\Lambda_{i}^{n+m}\left( t_{2r-1}^{n+m}\right)
\Lambda_{j}^{n+m}\left( t_{2r}^{n+m}\right) \right)$ conditional on $\mathcal{F}_{n+m-1}$.

\begin{corollary}
\label{Cor_E_nm1}%
For $i\neq j$,
\begin{align*}
& E_{n+m-1}\left[\exp\left(\theta_{0}\Lambda_{i}^{n+m}\left( t_{2r-1}^{n+m}\right)
\Lambda_{j}^{n+m}\left( t_{2r}^{n+m}\right) \right)\right] \\
=&\left( 1-\theta_{0}^{2}\Delta_{n+m}^{2}\right) ^{-1/2}\exp\left(
\theta_{1}\Lambda_{j}^{n+m-1}\left(t_{r}^{n+m-1}\right)%
\Lambda_{i}^{n+m-1}\left(t_{r}^{n+m-1}\right)\right) \\
& \times\exp\left(
\eta_{1}\Lambda_{j}^{n+m-1}\left(t_{r}^{n+m-1}\right)^{2}+\eta_{1}%
\Lambda_{i}^{n+m-1}\left(t_{r}^{n+m-1}\right)^{2}\right) ,
\end{align*}
where%
\begin{equation*}
\theta_{1}:=\theta_{0}\left(1-\theta_{0}^{2}\Delta_{n+m+1}^{2}\right)^{-1}/4, \text{ \ }%
\eta_{1}:=\theta_{0}^{2}\left( 1-\theta_{0}^{2}\Delta_{n+m+1}^{2}\right)
^{-1}\Delta_{n+m}/8.
\end{equation*}

Moreover, define
\begin{align*}
& P_{n+m,t_{r}^{n+m}}^{\prime}\left( W_{i,r}^{n+m}\in A,W_{j,r}^{n+m}\in
B\right) \\
=&\frac{E_{n+m-1}\left[ I\left(W_{i,r}^{n+m}\in A,W_{j,r}^{n+m}\in B\right)
\exp\left(\theta_{0}\Lambda_{i}^{n+m}\left( t_{2r-1}^{n+m}\right) \Lambda
_{j}^{n+m}\left( t_{2r}^{n+m}\right) \right)\right]}{E_{n+m-1}\left[\exp\left(\theta_{0}\Lambda
_{i}^{n+m}\left( t_{2r-1}^{n+m}\right) \Lambda_{j}^{n+m}\left(
t_{2r}^{n+m}\right) \right)\right]},
\end{align*}
then under $P_{n+m,t_{r}^{n+m}}^{\prime}$, and given $\mathcal{F}_{n+m-1}$,
we have that $(W_{i,r}^{n+m},W_{j,r}^{n+m})$ follows a Gaussian distribution
with covariance matrix%
\begin{equation*}
\Sigma_{n+m}^{i,j}\left( t_{r}^{n+m+1}\right) =\frac{1}{1-\theta_{0}^{2}%
\Delta_{n+m+1}^{2}}\left(
\begin{array}{cc}
1 & -\theta_{0}\Delta_{n+m+1} \\
-\theta_{0}\Delta_{n+m+1} & 1%
\end{array}
\right) ,
\end{equation*}
and mean vector%
\begin{equation*}
\mu_{n+m}^{i,j}\left( t_{r}^{n+m}\right) =\Sigma_{n+m}^{i,j}\left(
t_{r}^{n+m}\right) \left(
\begin{array}{c}
\theta_{0}\Delta_{n+m+1}^{1/2}\Lambda_{j}^{n+m-1}(t_{r}^{n+m-1})/2 \\
-\theta_{0}\Delta_{n+m+1}^{1/2}\Lambda_{i}^{n+m-1}(t_{r}^{n+m-1})/2%
\end{array}
\right) .
\end{equation*}
\end{corollary}

So, from Corollary \ref{Cor_E_nm1} we conclude that%
\begin{align}
& E_{n+m-1}\left[\exp\left(\theta_{0}\sum_{r=k+1}^{k^{\prime}}\Lambda_{i}^{n+m}\left(
t_{2r-1}^{n+m}\right) \Lambda_{j}^{n+m}\left( t_{2r}^{n+m}\right) \right)\right]  \notag \\
=&\left( 1-\theta_{0}^{2}\Delta_{n+m+1}^{2}\right) ^{-\left( k^{\prime
}-k\right)
/2}\exp\left(\theta_{1}\sum_{r=k+1}^{k^{\prime}}%
\Lambda_{j}^{n+m-1}\left(t_{r}^{n+m-1}\right)\Lambda_{i}^{n+m-1}\left(t_{r}^{n+m-1}\right)\right)  \notag
\\
&\times
\exp\left(\eta_{1}\sum_{r=k+1}^{k^{\prime}}%
\Lambda_{j}^{n+m-1}\left(t_{r}^{n+m-1}\right)^{2}+\eta_{1}\sum_{r=k+1}^{k^{\prime}}%
\Lambda_{i}^{n+m-1}\left(t_{r}^{n+m-1}\right)^{2}\right).  \label{Eq_Inner_E}
\end{align}
If $m\geq2$, we can continue taking the corresponding conditional
expectation given $\mathcal{F}_{n+m-2}$. Due to the recursive nature of (\ref%
{Eq_iterative_exp}) and the linear and quadratic terms that arise in (\ref%
{Eq_Inner_E}), it is convenient to consider%
\begin{align}
& \sum_{r=1}^{2^{n+m-1}}\theta_{1}\left( t_{r}^{n+m-1}\right) \Lambda
_{j}^{n+m-1}\left(t_{r}^{n+m-1}\right)\Lambda_{i}^{n+m-1}\left(t_{r}^{n+m-1}\right)
\label{Aux_Big_Sum_1} \\
& +\sum_{r=1}^{2^{n+m-1}}\eta_{1}\left( t_{r}^{n+m-1}\right) \left(\Lambda
_{j}^{n+m-1}\left(t_{r}^{n+m-1}\right)^{2}+\Lambda_{j}^{n+m-1}\left(t_{r}^{n+m-1}\right)^{2}\right),
\notag
\end{align}
where
\begin{equation*}
\theta_{1}\left( t_{r}^{n+m-1}\right) =\theta_{1}\times I\left(
r\in\{k+1,...,k^{\prime}\}\right),
\end{equation*}
\begin{equation*}
\eta_{1}\left( t_{r}^{n+m-1}\right) =\eta_{1}\times I\left(
r\in\{k+1,...,k^{\prime}\}\right).
\end{equation*}
We also introduce the following notations to simply the presentation of our tilting parameters.
 Due to the difference in the recursive relation for $\Lambda_i^n(t_r^n)$
between odd and even $r$'s, we recursively define for $l=2,...,m$.
\begin{align}
\theta_{+}^{l}\left( t_{r}^{n+m-l}\right) & =\theta_{l-1}\left(
t_{2r-1}^{n+m-l+1}\right) +\theta_{l-1}\left( t_{2r}^{n+m-l+1}\right),
\label{Many_Defs} \\
\theta_{-}^{l}\left( t_{r}^{n+m-l}\right) &= \theta_{l-1}\left(
t_{2r-1}^{n+m-l+1}\right) -\theta_{l-1}\left( t_{2r}^{n+m-l+1}\right),
\notag \\
\eta_{+}^{l}\left( t_{r}^{n+m-l}\right) & =\eta_{l-1}\left(
t_{2r-1}^{n+m-l+1}\right) +\eta_{l-1}\left( t_{2r}^{n+m-l+1}\right),  \notag
\\
\eta_{-}^{l}\left( t_{r}^{n+m-l}\right) & =\eta_{l-1}\left(
t_{2r-1}^{n+m-l+1}\right) -\eta_{l-1}\left( t_{2r}^{n+m-l+1}\right),  \notag
\\
\rho_{l}\left( t_{r}^{n+m-l}\right) & =\frac{\Delta_{n+m-l+2}\theta
_{+}^{l}\left( t_{r}^{n+m-l}\right) }{ 1-2\Delta_{n+m-l+2}\eta_{+}^{l}\left(
t_{r}^{m+n-l}\right) },  \notag \\
h_{l}\left( t_{r}^{n+m-l}\right) & =\frac{\Delta_{n+m-l+2}}{\left(
1-2\Delta_{n+m-l+2}\eta_{+}^{l}\left( t_{r}^{m+n-l}\right) \right) \left(
1-\rho_{l}\left( t_{r}^{n+m-l}\right) ^{2}\right) },  \notag
\end{align}
and set%
\begin{align*}
& \eta_{l}\left( t_{r}^{m+n-l}\right) \\
= &\frac{\eta_{+}^{l}\left( t_{r}^{m+n-l}\right) }{4} \\
& +\frac{h_{l}\left( t_{r}^{m+n-l}\right) }{8}\{ \theta_{-}^{l}\left(
t_{r}^{m+n-l}\right) ^{2} +4\eta_{-}^{l}\left( t_{r}^{m+n-l}\right)^{2} \\
& +4\theta_{-}^{l}\left( t_{r}^{m+n-l}\right) \eta_{-}^{l}\left(
t_{r}^{m+n-l}\right) \rho_{l}\left( t_{r}^{m+n-l}\right) \} , \\
& \theta_{l}\left( t_{r}^{m+n-l}\right) \\
=& \frac{\theta_{+}^{l}\left( t_{r}^{m+n-l}\right) }{4} \\
& +h_{l}\left( t_{r}^{m+n-l}\right) \{ \theta_{-}^{l}\left(
t_{r}^{m+n-l}\right) \eta_{-}^{l}\left( t_{r}^{m+n-l}\right) \\
& +\frac{1}{4}\theta_{-}^{l}\left( t_{r}^{m+n-l}\right) ^{2}g_{l}\left(
t_{r}^{m+n-l}\right) +\eta_{-}^{l}\left( t_{r}^{m+n-l}\right) ^{2}\rho
_{l}\left( t_{r}^{m+n-l}\right) \} .
\end{align*}
Finally, we decompose (\ref{Aux_Big_Sum_1}) into two parts (the cross term and the quadratic term) by defining
\begin{align*}
A\left( t_{r}^{n+m-l}\right) & =\theta_{l-1}\left( t_{2r-1}^{n+m-l+1}\right)
\Lambda_{j}^{n+m-l+1}(t_{2r-1}^{n+m-l+1})\Lambda
_{i}^{n+m-l+1}(t_{2r-1}^{n+m-l+1}) \\
& +\theta_{l-1}\left( t_{2r}^{n+m-l+1}\right)
\Lambda_{j}^{n+m-l+1}(t_{2r}^{n+m-l+1})%
\Lambda_{i}^{n+m-l+1}(t_{2r}^{n+m-l+1}), \\
B\left( t_{r}^{n+m-l}\right) & =\eta_{l-1}\left( t_{2r-1}^{n+m-l+1}\right)
(\Lambda_{j}^{n+m-l+1}(t_{2r-1}^{n+m-l+1})^{2}+%
\Lambda_{j}^{n+m-l+1}(t_{2r-1}^{n+m-l+1})^{2}) \\
& +\eta_{l-1}\left( t_{2r}^{n+m-l+1}\right)
(\Lambda_{j}^{n+m-l+1}(t_{2r}^{n+m-l+1})^{2}+%
\Lambda_{j}^{n+m-l+1}(t_{2r}^{n+m-l+1})^{2}),
\end{align*}
and
\begin{equation*}
C\left( t_{r}^{n+m-l}\right) =\left( 1-2\Delta_{n+m-l+1}\eta_{+}^{l}\left(
t_{r}^{m+n-l}\right) \right) ^{-1}\left( 1-\rho_{l}\left(
t_{r}^{m+n-l}\right) ^{2}\right) ^{-1/2}.
\end{equation*}
Then (\ref{Aux_Big_Sum_1}) can be written as%
\begin{equation*}
\sum_{r=1}^{2^{n+m-2}}\left(A\left( t_{r}^{n+m-2}\right) +B\left(
t_{r}^{n+m-2}\right) \right),
\end{equation*}
and the following result is key in evaluating (\ref{Eq_iterative_exp}).

\begin{corollary}
\label{Cor_E_nmL}For $i\neq j$, $l=2,3,...,m$ and $r=1,2,...,2^{n+m-l}$%
\begin{align*}
& E_{n+m-l}\left[\exp\left(A\left( t_{r}^{n+m-l}\right) +B\left( t_{r}^{n+m-l}\right) \right)\right]
\\
= &C\left( t_{r}^{n+m-l}\right) \exp\left( \theta_{l}\left(
t_{r}^{m+n-l}\right) \Lambda_{i}\left( t_{r}^{m+n-l}\right)
\Lambda_{j}\left( t_{r}^{m+n-l}\right) \right) \\
& \times\exp\left(\eta_{l}\left( t_{r}^{m+n-l}\right) \left(\Lambda_{i}\left(
t_{r}^{m+n-l}\right) ^{2}+\Lambda_{j}\left( t_{r}^{m+n-l}\right)
^{2}\right)\right).
\end{align*}
Moreover, define
\begin{align*}
& P_{n+m-l+1,t_{r}^{n+m-l+1}}^{\prime}\left( W_{i,r}^{n+m-l+1}\in
A,W_{j,r}^{n+m-l+1}\in B\right) \\
=& \frac{E_{n+m-l}\left[ I\left(W_{i,r}^{n+m-l+1}\in A,W_{j,r}^{n+m-l+1}\in
B\right) \exp\left(A\left( t_{r}^{n+m-l}\right) +B\left( t_{r}^{n+m-l}\right) \right)\right]}{%
E_{n+m-l}\left[\exp\left(A\left( t_{r}^{n+m-l}\right) +B\left( t_{r}^{n+m-l}\right) \right)\right]},
\end{align*}
then under $P_{n+m-l+1,t_{r}^{n+m-l+1}}^{\prime}$, and given $\mathcal{F}%
_{n+m-l}$, we have that $(W_{i,r}^{n+m-l+1},W_{j,r}^{n+m-l+1})$ follows a
Gaussian distribution with covariance matrix%
\begin{align*}
& \Sigma_{n+m-l+1}^{i,j}\left( t_{r}^{n+m-l+1}\right) \\
= & \frac{1}{1-\rho_{l}\left( t_{r}^{m+n-l}\right) ^{2}} \\
&\times \left(
\begin{array}{cc}
\left( 1-2\Delta_{n+m-l+1}\eta_{+}^{l}\left( t_{r}^{m+n-l}\right) \right) ^{-1}
& g_{l}\left( t_{r}^{m+n-l}\right) \\
g_{l}\left( t_{r}^{m+n-l}\right) & \left(
1-2\Delta_{n+m-l+1}\eta_{+}^{l}\left( t_{r}^{m+n-l}\right) \right) ^{-1}%
\end{array}
\right)
\end{align*}
where
$g_l(t_r^{m+n-l})=\Delta_{n+m-l+2}\theta_+^l\left(t_r^{n+m-l}\right)\left(1-2\Delta_{n+m-l+2}\eta_+^l\left(t_r^{n+m-l}\right)\right)^{-2}$.
and mean vector%
\begin{align*}
& \mu_{n+m}^{i,j}\left( t_{r}^{n+m-l+1}\right) \\
= & \Delta_{n+m-l+1}^{1/2}\Sigma_{n+m-l+1}^{i,j}\left( t_{r}^{n+m-l+1}\right)
\\
&\times \left(
\begin{array}{c}
\Lambda_{i}\left( t_{r}^{n+m-l}\right) \eta_{-}^{l}\left(
t_{r}^{n+m-l}\right) +\frac{1}{2}\Lambda_{j}\left( t_{r}^{n+m-l}\right)
\theta_{-}^{l}\left( t_{r}^{n+m-l}\right) \\
\Lambda_{j}\left( t_{r}^{n+m-l}\right) \eta_{-}^{l}\left(
t_{r}^{n+m-l}\right) +\frac{1}{2}\Lambda_{i}\left( t_{r}^{n+m-l}\right)
\theta_{-}^{l}\left( t_{r}^{n+m-l}\right)%
\end{array}
\right) .
\end{align*}
\end{corollary}

Using Corollary \ref{Cor_E_nmL} we conclude that%
\begin{align*}
& E_{n+m-l}\left[\exp\left(\sum_{r=1}^{2^{n+m-l}}\left(A\left( t_{r}^{n+m-l}\right) +B\left(
t_{r}^{n+m-l}\right) \right)\right)\right] \\
=& {\displaystyle\prod\limits_{r=1}^{2^{n+m-l}}} C\left(
t_{r}^{n+m-l}\right) \times\exp\left( \sum_{r=1}^{2^{n+m-l-1}}\left(A\left(
t_{r}^{n+m-l-1}\right) +B\left( t_{r}^{n+m-l-1}\right) \right)\right).
\end{align*}
Therefore, combining Corollary \ref{Cor_E_nm1} and repeatedly iterating the
previous expression we conclude that
\begin{align}
& E_{n}\left[\exp(\theta_{0}\{L_{i,j}^{n+m}(k^{\prime})-L_{i,j}^{n+m}(k)\})\right]
\notag \\
=& \left( 1-\theta_{0}^{2}\Delta_{n+m}^{2}\right) ^{-\left( k^{\prime
}-k\right) /2}{\displaystyle\prod\limits_{l=2}^{m}} {\displaystyle%
\prod\limits_{r=1}^{2^{n+m-l}}} C\left( t_{r}^{n+m-l}\right)  \notag \\
& \times\exp\left(\sum_{r=1}^{2^{n}}\theta_{m}\left( t_{r}^{n}\right)
\Lambda_{i}\left( t_{r}^{n}\right) \Lambda_{j}\left( t_{r}^{n}\right)
+\sum_{r=1}^{2^{n}}\eta_{m}\left( t_{r}^{n}\right) \left\{\Lambda_{i}\left(
t_{r}^{n}\right) ^{2}+\Lambda_{j}\left( t_{r}^{n}\right) ^{2} \right\}\right).
\label{E_n_FIN}
\end{align}

\subsection{Conditional\ Large Deviations Estimates for $%
L_{i,j}^{n}\left( k\right)$\label{Sub_Cond_LD}}

We wish to estimate, for $1\leq i,j \leq d^{\prime}$, $i\neq j$, $k^{\prime}>k$ and $k^{\prime},k\in
\{0,1,...,2^{n+m-1}\},$%
\begin{align*}
& P_{n}\left( |L_{i,j}^{n+m}(k^{\prime})-L_{i,j}^{n+m}(k)|\text{ }>\left(
k^{\prime}-k\right) ^{\beta}\Delta_{n+m}^{2\alpha}\right) \\
\leq&\exp(-\theta_{0}\left( k^{\prime}-k\right) ^{\beta}\Delta
_{n+m}^{2\alpha})\times\{E_{n}[\exp(\theta_{0}\{L_{i,j}^{n+m}(k^{\prime
})-L_{i,j}^{n+m}(k)\})] \\
& +E_{n}[\exp(-\theta_{0}\{L_{i,j}^{n+m}(k^{\prime})-L_{i,j}^{n+m}(k)\})]\}.
\end{align*}
We borrow some intuition from the proof of Lemma \ref{Lem_L_Abs_Sum} and
select%
\begin{equation}
\theta_{0}(m,k^{\prime},k):=\theta_{0}=\frac{\gamma}{\left( k^{\prime
}-k\right) ^{1/2}\Delta_{n}^{2\alpha^{\prime}}\Delta_{m}}.
\label{Theta_0_select}
\end{equation}
We will drop the dependence on $(m,k^{\prime},k)$ for brevity. In addition,
we pick $\gamma\leq1/4$ and $\alpha^{\prime}\in(\alpha,1/2)$ so that%
\begin{equation*}
\exp(-\theta_{0}\left( k^{\prime}-k\right) ^{\beta}\Delta_{n+m}^{2\alpha
})=\exp(-\gamma\left( k^{\prime}-k\right)
^{\beta-1/2}\Delta_{n}^{2(\alpha-\alpha^{\prime})}\Delta_{m}^{2\alpha-1})
\end{equation*}
Our next task is to control the $E_{n}\left[\exp(\theta_{0}\{L_{i,j}^{n+m}(k^{%
\prime})-L_{i,j}^{n+m}(k)\})\right]$, which is the purpose of the following result,
proved in the appendix to this section.

\begin{lemma}
\label{Lemma_MGF_Clean} For $i\neq j$, suppose that $\theta_{0}$ is chosen according to (\ref%
{Theta_0_select}), and $n$ is chosen such that
\begin{equation}
\max_{r\leq2^{n}}\{\left\vert \Lambda_{i}\left( t_{r}^{n}\right) \right\vert
,\left\vert \Lambda_{j}\left( t_{r}^{n}\right) \right\vert \}\leq\Delta
_{n}^{\alpha^{\prime}}  \label{Cond_on_n}
\end{equation}
and for $\varepsilon_{0}\in\left(0,1/2\right) $
\begin{equation}  \label{Cond_on_n_2}
\left\vert
\sum_{r=l+1}^{m}\Lambda_{i}(t_{r}^{n})\Lambda_{j}(t_{r}^{n})\right\vert
\leq\varepsilon_{0}(m-l)^{\beta}\Delta _{n}^{2\alpha^{\prime}}
\mbox{ for
all } 0\leq l < m \leq2^{n}
\end{equation}
with $\alpha^{\prime}\in\left( \alpha,1/2\right) $, then
\begin{equation*}
E_{n}[\exp(\theta_{0}\{L_{i,j}^{n+m}(k^{\prime})-L_{i,j}^{n+m}(k)\})]\leq
4\exp\left( \varepsilon_{0}\gamma(k^{\prime}-k)^{\beta-1/2}\right) .
\end{equation*}
\end{lemma}

\textbf{Remark:} It is very important to note that due to Lemma \ref%
{lm:N1_finite_mean} we can always continue simulating the $W_{i,k}^{n}$'s
(maybe conditional on $\left\{\left\vert W_{i,k}^{n}\right\vert \leq4\sqrt{n+1}\right\}$
in case $n>N_{1}$) to make sure that (\ref{Cond_on_n}) holds for some $n$.
Similarly, condition (\ref{Cond_on_n_2}) can be simultaneously enforced with
(\ref{Cond_on_n}) because of Lemma \ref{Lem_L_Abs_Sum}. Actually, Lemma \ref%
{lm:N1_finite_mean} and Lemma \ref{Lem_L_Abs_Sum} indicate that conditions (%
\ref{Cond_on_n}) and (\ref{Cond_on_n_2}) will occur eventually for all $n$
larger than some random threshold. Our simulation algorithms will
ultimately detect such threshold, but Lemma \ref{Lemma_MGF_Clean} does not
require that we know that threshold.\newline

As a consequence of Lemma \ref{Lemma_MGF_Clean}, using Chernoff's bound, we
obtain the following proposition.

\begin{proposition}
\label{Prop_Nice_LD} For $i\neq j$, if $n$ is chosen such that (\ref{Cond_on_n}) and (\ref%
{Cond_on_n_2}) hold, then
\begin{align*}
& P_{n}\left( |L_{i,j}^{n+m}(k^{\prime})-L_{i,j}^{n+m}(k)|\text{ }>\left(
k^{\prime}-k\right) ^{\beta}\Delta_{n+m}^{2\alpha}\right) \\
\leq& 8\exp\left( -\frac{1}{2}\gamma\left( k^{\prime}-k\right)
^{\beta-1/2}\Delta
_{n}^{2(\alpha-\alpha^{\prime})}\Delta_{m}^{2\alpha-1}\right) .
\end{align*}
\end{proposition}

\subsection{Joint Tolerance-Enforced Simulation for $\protect\alpha$-H\"{o}%
lder Norms and Proof of Theorem \protect\ref{th:main2}.\label{Sub_TES}}

Define%
\begin{eqnarray*}
\mathcal{C}_{n}(m)&=&\{|L_{i,j}^{n+m}(k^{\prime})-L_{i,j}^{n+m}(k)|>(k^{\prime
}-k)^{\beta}\Delta_{n+m}^{2\alpha}\\
&&\text{ for some }0\leq k<k^{\prime
}<2^{n+m-1}, 1\leq i,j\leq d^{\prime}, i\neq j\},
\end{eqnarray*}
and put $\tau_{1}\left( n\right) =\inf\{m\geq1:\mathcal{C}_{n}(m)$ occurs$\}$%
. We write $\mathcal{\bar{C}}_{n}(m)$ for the complement of $\mathcal{C}%
_{n}(m)$, so that
\begin{equation*}
P_{n}(\tau_{1}\left( n\right) <\infty)=\sum_{m=1}^{\infty}P\left( \mathcal{C}%
_{n}(m)\cap\cap_{l=1}^{m-1}\mathcal{\bar{C}}_{n}(l)\right) .
\end{equation*}

To facilitate the explanation, we next introduce a few more notations. Let
\begin{equation*}
\omega_{n:n+m}:=\{W_{i,k}^{l}:0\leq k\leq2^{n}-1,1\leq i\leq
d^{\prime},n<l\leq n+m\}.
\end{equation*}
In addition, define
\begin{align*}
v_{n}(k,k^{\prime}|m) :=& 8\exp\left( -\frac{1}{2}\gamma\left( k^{\prime
}-k\right) ^{\beta-1/2}\Delta_{n}^{2(\alpha-\alpha^{\prime})}\Delta
_{m}^{2\alpha-1}\right) \\
&\times I\left( 0\leq k<k^{\prime}\leq2^{n+m-1}\right) I\left( m\geq1\right)
\\
b_{n}(m) :=& \sum_{0\leq k<k^{\prime}\leq2^{m+n-1}}v_{n}(k,k^{\prime}|m) \\
q_{n}(k,k^{\prime}|m) :=& \frac{v_{n}(k,k^{\prime}|m)}{b_{n}(m)}
\end{align*}
and
\begin{equation*}
P_{n,m}^{i,j,k,k^{\prime}}\left( \mathcal{\omega}_{n:n+m}\in\cdot\right) =
\frac{E_{n}\left[I\left( \mathcal{\omega}_{n:n+m}\in\cdot\right) \exp\left(
\theta_{0}\{L_{i,j}^{n+m}(k^{\prime})-L_{i,j}^{n+m}(k)\}\right)\right] }{E_{n}\left[
\exp\left( \theta_{0}\{L_{i,j}^{n+m}(k^{\prime})-L_{i,j}^{n+m}(k)\}\right)\right] }.
\end{equation*}
We also denote
\begin{equation*}
\psi_{n}(m,i,j,k,k^{\prime}):=\log E_{n}\left[ \exp\left(
\theta_{0}\left\{L_{i,j}^{n+m}(k^{\prime})-L_{i,j}^{n+m}(k^{\prime})\right\}\right)\right]
\end{equation*}

Observe that
\begin{align*}
b_{n}\left( m\right) & = \sum_{0\leq k<k^{\prime}\leq2^{n+m-1}} 8\exp\left( -%
\frac{1}{2}\gamma\left( k^{\prime}-k\right)
^{\beta-1/2}\Delta_{n}^{2(\alpha-\alpha^{\prime})}\Delta_{m}^{2\alpha-1}%
\right) \\
& \leq2^{2(m+n)+3}\exp\left( -\frac{1}{2}\gamma\Delta_{n}^{2(\alpha
-\alpha^{\prime})}\Delta_{m}^{2\alpha-1}\right) .
\end{align*}
Thus, $b_{n}(m) \rightarrow0$ as $n \rightarrow\infty$. Then we can select
any probability mass function $\{g(m): m \geq1\}$, e.g. $g(m)=e^{-1}/(m-1)!$
for $m \geq1$, by assuming that $n$ is sufficiently large,
\begin{equation*}
g(m) \geq d^{\prime2} b_{n}(m)
\end{equation*}

Now consider the following procedure, which we called Procedure Aux, Aux for
``auxiliar", which is given for pedagogical purposes, because as we shall see
shortly it is not directly applicable but useful to understand the nature
of the method that we shall ultimately use.\newline

\noindent\textbf{Procedure Aux}

\textbf{Input: }We assume that we have simulated $\{(W_{i,k}^{n}:0\leq
k<2^{l}):l\leq n\}\}$.

\textbf{Output: }A Bernoulli $F$ with parameter $P_{n}\left( \tau
_{1}(n)<\infty\right) $, and if $F=1$, also
\begin{equation*}
\mathcal{\omega}_{n:\tau_1(n)}=\{W_{i,k}^{l}:1\leq k\leq2^{l}-1,1\leq i\leq
d^{\prime },n<l\leq \tau_1(n)\}
\end{equation*}
conditional on the event $\tau_{1}(n)<\infty$.

\textbf{Step 1}: Sample $M$ according to $g\left( m\right) $.

\textbf{Step 2}: Given $M=m$ sample $I$ and $J$ ($I\neq J$) uniformly
over the set $\{1,2,...,d^{\prime}\}$.Then, sample $%
K^{\prime},K $ from $q_{n}\left( k,k^{\prime}|m\right) $.

\textbf{Step 3}: Given $M=m$, $I=i$,$J=j$,$K=k$, and $K^{\prime}=k^{\prime}$%
, simulate $\omega_{n:n+m}$ from $P_{n,m}^{i,j,k,k^{\prime}}\left(
\cdot\right) $. Note that simulation from $P_{n,m}^{i,j,k,k^{\prime}}\left(
\cdot\right) $ can be done according to Corollary \ref{Cor_E_nmL}.

\textbf{Step 4}: Compute
\begin{align*}
&\Xi_{n}(m,i,j,k,k^{\prime},\omega_{n:n+m}) \\
=&\frac{1}{g(m)\left(d^{\prime}(d^{\prime}-1)\right)^{-1}q_{n}(k,k^{\prime}|m)\exp\left(
\theta_{0}\{L_{i,j}^{n+m}(k^{\prime})-L_{i,j}^{n+m}(k)\}-%
\psi_{n}(m,i,j,k,k^{\prime})\right) },
\end{align*}
and
\begin{equation*}
\mathcal{N}_{n}\left( m\right) =\sum_{1\leq i,j\leq d^{\prime}, i\neq j}\sum_{1\leq
h<h^{\prime}\leq2^{n+m-1}}I\left( \left\vert L_{i,j}^{n+m}(h^{\prime
})-L_{i,j}^{n+m}(h)\right\vert >(h-h^{\prime})^{\beta}\Delta_{n+m}^{2\alpha
}\right) .
\end{equation*}

\textbf{Step 5}: Simulate $U$ uniformly distributed on $[0,1]$ independent
of everything else and output
\begin{align*}
F=& I\{ U<I\left( \left\{ \left\vert
L_{i,j}^{n+m}(k^{\prime})-L_{i,j}^{n+m}(k)\right\vert
>(k-k^{\prime})^{\beta}\Delta_{n+m}^{2\alpha }\right\} \cap\cap_{l=1}^{m-1}%
\mathcal{\bar{C}}_{n}(l)\right) \\
&\times \Xi_{n}(m,i,j,k,k^{\prime},\omega_{n:n+m})/\mathcal{N}_{n}(m)\}.
\end{align*}

If $F=1$, also output $\omega_{n:n+m}$.\newline
\noindent{\bf End of Procedure Aux}\\

We first notice that when $\left\vert L_{i,j}^{n+m}(k^{\prime})-L_{i,j}^{n+m}(k)\right\vert >(k-k^{\prime})^{\beta}\Delta_{n+m}^{2\alpha}$,
$$g(m)\left(d^{\prime}(d^{\prime}-1)\right)^{-1}q_{n}(k,k^{\prime}|m)\exp\left(
\theta_{0}\{L_{i,j}^{n+m}(k^{\prime})-L_{i,j}^{n+m}(k)\}-%
\psi_{n}(m,i,j,k,k^{\prime})\right)>1.$$
Thus $\Xi_{n}(m,i,j,k,k^{\prime},\omega_{n:n+m})<1$. That is to say the likelihood ration function is bounded and the Bernoulli
random variable $F$ is well defined.

We claim that the output $F$ is distributed as a Bernoulli random variable
with parameter $P_{n}\left( \tau_{1}(n)<\infty\right) $. Moreover, we claim
that if $F=1$, then, $\mathcal{\omega}_{n:n+M}$ is distributed according to $%
P_{n}\left( \mathcal{\omega}_{n:\tau_{1}(n)}\in\cdot\text{\ }|\text{ }%
\tau_{1}(n)<\infty\right) $. We first verify the claim that the outcome in
Step 5 follows a Bernoulli with parameter $P_{n}\left(
\tau_{1}(n)<\infty\right) $. In order to see this, let $Q_{n}$ denote the
distribution induced by Procedure Aux. Note that%
\begin{align*}
& Q_{n}(U<I\left( \left\{ \left\vert
L_{i,j}^{n+M}(K^{\prime})-L_{i,j}^{n+M}(K)\right\vert
>(K-K^{\prime})^{\beta}\Delta_{n+M}^{2\alpha }\right\} \cap\cap_{l=1}^{M-1}%
\mathcal{\bar{C}}_{n}(l)\right) \\
& \times\Xi_{n}(M,I,J,K,K^{\prime},\omega_{n:n+M})/\mathcal{N}_{n}\left(
m\right) ) \\
= & E^{Q_{n}}[I\left( \left\{ \left\vert
L_{i,j}^{n+M}(K^{\prime})-L_{i,j}^{n+M}(K)\right\vert
>(K-K^{\prime})^{\beta}\Delta_{n+M}^{2\alpha }\right\} \cap\cap_{l=1}^{M-1}%
\mathcal{\bar{C}}_{n}(l)\right) \\
& \times\Xi_{n}(M,I,J,K,K^{\prime},\omega_{n:n+M})/\mathcal{N}_{n}\left(
m\right) ] \\
= & \sum_{m=1}^{\infty}\sum_{1\leq i,j\leq d^{\prime}}\sum_{1\leq
k<k^{\prime}\leq2^{n+m-1}}E^{Q_{n}}[I\left( \left\{ \left\vert
L_{i,j}^{n+m}(k^{\prime})-L_{i,j}^{n+m}(k)\right\vert
>(k-k^{\prime})^{\beta}\Delta_{n+m}^{2\alpha}\right\} \cap\cap_{l=1}^{m-1}%
\mathcal{\bar{C}}_{n}(l)\right) \\
& \times\frac{dP_{n}}{dP_{n,m}^{i,j,k,k^{\prime}}}(\omega_{n:n+m})\times
\frac{1}{\mathcal{N}_{n}\left( m\right) }] \\
= & \sum_{m=1}^{\infty}\sum_{1\leq i,j\leq d^{\prime}}\sum_{1\leq
k<k^{\prime}\leq2^{n+m-1}}E_{n}\left( \frac{I\left( \left\{ \left\vert
L_{i,j}^{n+m}(k^{\prime})-L_{i,j}^{n+m}(k)\right\vert >(k-k^{\prime})^{\beta
}\Delta_{n+m}^{2\alpha}\right\} \cap\cap_{l=1}^{m-1}\mathcal{\bar{C}}%
_{n}(l)\right) }{\mathcal{N}_{n}\left( m\right) }\right) \\
= & \sum_{m=1}^{\infty}P_{n}\left( \mathcal{C}_{n}(m)\cap\cap_{l=1}^{m-1}%
\mathcal{\bar{C}}_{n}(l)\right)\\
=& P_{n}(\tau_{1}\left( n\right) <\infty).
\end{align*}
Similarly, for the second claim,
\begin{align*}
& Q_{n}\left( \mathcal{\omega}_{n:n+M}\in A\text{ }|U<I\left( \mathcal{C}%
_{n}(M)\cap\cap_{l=1}^{M-1}\mathcal{\bar{C}}_{n}(l)\right) \Xi
_{n}(M,I,J,K,K^{\prime},\omega_{n:n+M})\right) \\
= & \sum_{m=1}^{\infty}E^{Q_{n}}\left( \mathcal{\omega}_{n:n+m}\in A\text{ },%
\text{ }\frac{dP_{n,m}^{I,J,K,K^{\prime}}}{dP_{n}}\left( \mathcal{\omega }%
_{n:n+m}\right) I\left( \mathcal{C}_{n}(m)\cap\cap_{l=1}^{M-1}\mathcal{\bar{C%
}}_{n}(l)\right) \right) /P_{n}(\tau_{1}\left( n\right) <\infty) \\
= & \sum_{m=1}^{\infty}P_{n}\left( \mathcal{\omega}_{n:n+m}\in A\text{ },%
\text{ }\tau_{1}\left( n\right) =m\right) /P_{n}(\tau_{1}\left( n\right)
<\infty) \\
= & P_{n}\left( \mathcal{\omega}_{n:n+\tau_{1}(n)}\in A\text{\ }|\text{ }%
\tau_{1}(n)<\infty\right)
\end{align*}

The deficiency of Procedure Aux is that it does not recognize that $n>N_{1}$%
. Let us now account for this fact and note that conditional on $\mathcal{F}%
_{N_{1}}$ we have that $W_{i,k}^{n}$'s are i.i.d. $N(0,1)$ but conditional
on $\{|W_{i,k}^{n}|\leq4\sqrt{n+1}\}$ for all $n>N_{1}$. Define
\begin{equation*}
\mathcal{H}_{m}^{n}=\{|W_{i,k}^{h}|\leq4\sqrt{h+1}:0\leq k\leq2^{h}-1,n<h\leq
n+m\}.
\end{equation*}
In order to simulate $P_{N_{1}}\left( \tau_{1}(N_{1})<\infty\right) $ we
modify step 3 of Procedure Aux. Specifically, we have\newline

\noindent\textbf{Procedure B}

\textbf{Input: }We assume that we have simulated $\{(W_{i,k}^{l}:0\leq
k<2^{l}):l\leq n\}$. So, the $W_{i,k}^{m}$'s are i.i.d. $N(0,1)$ but
conditional on $\{|W_{i,k}^{m}|$ $<4\sqrt{m+1}\}$ for all $m>n$. We also
assume that conditions (\ref{Cond_on_n}) and (\ref{Cond_on_n_2}) hold in
Lemma \ref{Lemma_MGF_Clean}; note the discussion following Lemma \ref%
{Lemma_MGF_Clean} which notes that this can be assumed at the expense of
simulating additional $W_{i,k}^{m}$'s (with $\{|W_{i,k}^{m}|$ $<4\sqrt{m+1}%
\} $ if $m>N_{1}$).

\textbf{Output: }A Bernoulli $F$ with parameter $P_{n}(\tau_{1}(n)<\infty ,%
\mathcal{H}_{\infty}^{n})$, and if $F=1$, also
\begin{equation*}
\mathcal{\omega}_{n:n+\tau_{1}(n)}=\{W_{i,k}^{l}:1\leq k\leq2^{n},1\leq
i\leq d^{\prime},n<l\leq n+\tau_{1}\left( n\right) \}
\end{equation*}
conditional on $\tau_{1}(n)<\infty$ and on $\mathcal{H}_{\infty}^{n}$.

\textbf{Step 1}: Sample $M$ according to $g\left( m\right) $.

\textbf{Step 2}: Given $M=m$ sample $I$ and $J$ ($I\neq J$) uniformly
over the set $\{1,2,...,d^{\prime}\}$.Then, sample $%
K^{\prime},K $ from $q_{n}\left( k,k^{\prime}|m\right) $.

\textbf{Step 3}: Given $M=m$, $I=i$,$J=j$,$K=k$, and $K^{\prime}=k^{\prime}$%
, simulate $\omega_{n:n+m}$ from $P_{n,m}^{i,j,k,k^{\prime}}\left(
\cdot\right) $. Note that simulation from $P_{n,m}^{i,j,k,k^{\prime}}\left(
\cdot\right) $ can be done according to Corollary \ref{Cor_E_nmL}.
Check if $\mathcal{H}_m^n$ occurs. If yes, continue to Step 4; otherwise, go back to Step 1.

\textbf{Step 4}: Compute
\begin{align*}
&\Xi_{n}(m,i,j,k,k^{\prime},\omega_{n:n+m}) \\
=&\frac{1}{g(m)\left(d^{\prime}(d^{\prime}-1)\right)^{-1}q_{n}(k,k^{\prime}|m)\exp\left(
\theta_{0}\{L_{i,j}^{n+m}(k^{\prime})-L_{i,j}^{n+m}(k)\}-%
\psi_{n}(m,i,j,k,k^{\prime})\right) },
\end{align*}
and
\begin{equation*}
\mathcal{N}_{n}\left( m\right) =\sum_{1\leq i,j\leq d^{\prime},i\neq j}\sum_{1\leq
k<k^{\prime}\leq2^{n+m-1}}I\left( \left\vert L_{i,j}^{n+m}(k^{\prime
})-L_{i,j}^{n+m}(k)\right\vert >(k-k^{\prime})^{\beta}\Delta_{n+m}^{2\alpha
}\right) .
\end{equation*}

\textbf{Step 5}: Simulate $U$ uniformly distributed on $[0,1]$ independent
of everything else and output
\begin{align*}
F=I\{U<&\frac{I\left( \mathcal{H}_{m}^{n}\cap\left\{ \left\vert
L_{i,j}^{n+m}(k^{\prime})-L_{i,j}^{n+m}(k)\right\vert >(k-k^{\prime})^{\beta
}\Delta_{n+m}^{2\alpha}\right\} \cap\cap_{l=1}^{M-1}\mathcal{\bar{C}}%
_{n}(l)\right) P\left( \mathcal{H}_{\infty}^{n+m}\right) }{P\left( \mathcal{H%
}_{\infty}^{n}\right) } \\
& \times\Xi_{n}(m,i,j,k,k^{\prime},\omega_{n:n+m})/\mathcal{N}_{n}(m)\}
\end{align*}
(Notice that $P(\mathcal{H}_{\infty}^{n+m} )/P(\mathcal{H}_{\infty}^{n})=P(%
\mathcal{H}_{n+m}^{n})$ and can be computed in finite steps.)\newline
If $F=1$, also output $\omega_{n:n+m}.$\newline
\noindent{\bf End of Procedure B}\\

Let $\tilde{Q}_{n}$ denote the distribution induced by Procedure $B$%
. Following the same analysis as that given for Procedure Aux, we can verify
that
\begin{align*}
\tilde{Q}_{n}(U< & \frac{I\left( \mathcal{H}_{m}^{n}\cap\left\{ \left\vert
L_{i,j}^{n+m}(k^{\prime})-L_{i,j}^{n+m}(k)\right\vert >(k-k^{\prime})^{\beta
}\Delta_{n+m}^{2\alpha}\right\} \cap\cap_{l=1}^{M-1}\mathcal{\bar{C}}%
_{n}(l)\right) P\left( \mathcal{H}_{\infty}^{n+m}\right) }{P\left( \mathcal{H%
}_{\infty}^{n}\right) } \\
& \times\Xi_{n}(m,i,j,k,k^{\prime},\omega_{n:n+m})/\mathcal{N}%
_{n}(m)))=P_{n}\left( \tau_{1}(n)<\infty|\mathcal{H}_{\infty}^{n}\right) .
\end{align*}
And if the Bernoulli trial is a success, then, $\mathcal{\omega}_{n:n+M}$ is
distributed according to
\begin{equation*}
P_{n}\left( \mathcal{\omega}_{n:n+\tau_{1}(n)}\in\cdot\text{\ }|\text{ }%
\tau_{1}(n)<\infty,\mathcal{H}_{\infty}^{n}\right) .
\end{equation*}

Finally, if $\tau_{1}\left( n\right) =\infty$, we may still need to simulate $%
\mathcal{\omega}_{n:n+m}$ for any $m\geq1$, but now, conditional on $%
\{\tau_{1}(n)=\infty,\mathcal{H}_{\infty}^{n}\}$. Note that
\begin{align*}
&P_{n}\left( \mathcal{\omega}_{n:n+m}\in A\text{\ }|\text{ }%
\tau_{1}(n)=\infty,\mathcal{H}_{\infty}^{n}\right) \\
=& \frac{P_{n}\left( \mathcal{\omega}_{n:n+m}\in A\text{\ },\text{ }%
\tau_{1}(n)=\infty ,\mathcal{H}_{\infty}^{n}\right) }{P_{n}\left(
\tau_{1}(n)=\infty ,\mathcal{H}_{\infty}^{n}\right) } \\
=& \frac{E_{n}I(\mathcal{\omega}_{n:n+m}\in A\text{,}\tau_{1}(n)>m,\mathcal{H%
}_{m}^{n})P_{n+m}(\tau_{1}(n+m)=\infty,\mathcal{H}_{\infty }^{n+m})}{%
P_{n}\left( \tau_{1}(n)=\infty,\mathcal{H}_{\infty}^{n}\right) }.
\end{align*}
Thus we can sample $\mathcal{\omega}_{n:n+m}$ from $P_{n}\left(
\cdot\right) $ and accept the path with probability
\begin{equation*}
I(\tau_{1}(n)>m,\mathcal{H}_{m}^{n})P_{n+m}(\tau_{1}(n+m)=\infty ,\mathcal{H}%
_{\infty}^{n+m}).
\end{equation*}
This clearly can be done since we can easily simulate Bernoulli's with
probability
\begin{equation*}
P_{n+m}(\tau(n+m)=\infty,\mathcal{H}_{\infty}^{n+m})=P_{n+m}(\tau
_{1}(n+m)=\infty\text{ }|\text{ }\mathcal{H}_{\infty}^{n+m})P_{n+m}(\mathcal{%
H}_{\infty}^{n+m}).
\end{equation*}

We summarize the algorithm as follows:\newline

\noindent\textbf{Algorithm II: Simulate $N_{1}$ and $N_{2}$ jointly with $%
W_{i,k}^{n}$'s for $1\leq n\leq N_{0}$, where $N_{0}$ is chosen such that $%
\sup_{t\in\lbrack0,1]}||\hat{X}^{N_{0}}(t)-X(t)||_{\infty}\leq\varepsilon$ }

\textbf{Input:} The parameters required to run Algorithm I, and Procedures A
and B. These are the tilting parameters $\theta_{0}$'s.

\textbf{Step 1}: Simulate $N_{1}$ jointly with $W_{i,k}^{m}$'s for $0\leq
m\leq N_{1}$ using Algorithm I (see the remark that follows after Algorithm
I). Let $n=N_{1}$.

\textbf{Step 2}: If any of the conditions (\ref{Cond_on_n}) and (\ref%
{Cond_on_n_2}) from Lemma \ref{Lemma_MGF_Clean} are not satisfied keep
simulating $W_{i,k}^{m}$'s for $m>n$ until the first level $m>n$ for which
conditions (\ref{Cond_on_n}) and (\ref{Cond_on_n_2}) are satisfied. Redefine
$n$ to be such first level $m$.

\textbf{Step 3}: Run Procedure B and obtain as output $F$ and if $F=1$ also
obtain $\omega_{n:n+\tau(n)}$.

\textbf{Step 4}: If $\tau(n)<\infty$ (i.e. $F=1$) set $n\longleftarrow%
\tau(n) $ and go back to Step 2. Otherwise, go to Step 4.

\textbf{Step 5}: Calculate $G$ according to Procedure \textit{A} and solve
for $N_{0}$ such that $G\Delta_{N_{0}}^{2\alpha-\beta}<\varepsilon$.

\textbf{Step 6}: If $N_{0}>n$ sample $\omega_{n:N_{0}}$ from $P_{n}(\cdot)$
and sample a Bernoulli random variable, $I$ with probability of success $%
P_{N_{0}}(\tau(N_{0})=\infty,\mathcal{H}_{\infty}^{N_{0}})$.

\textbf{Step 7}: If $I=0$, go back to Step 6.

\textbf{Step 8}: Output $\omega_{0:N_{0}}$.\newline
\noindent{\bf End of Algorithm II}\\

We obtain $\{W_{i,k}^{l}:0\leq k<2^{l},l\leq N_{0},1\leq i\leq d\}$ from
Algorithm II. We have from recursions in Lemma \ref{lm:lambda} how to obtain%
\begin{equation}
\{(Z_{i}(t_{r}^{l})-Z_{i}(t_{r-1}^{l})):1\leq r\leq2^{l},1\leq l\leq
N_{0},1\leq i\leq d\}  \label{incrr_algo_II}
\end{equation}
and then we can compute $\{\hat{X}^{N_{0}}(t):t\in D_{N_{0}}\}$ using
equation (\ref{eq:recursion1}).\newline

\textbf{Remark: }Observe that after completion of Algorithm II, one can
actually continue the simulation of increments in order to obtain an
approximation with an error $\varepsilon^{\prime}<\varepsilon$. In
particular, this is done by repeating Steps 4 to 8. Start from Step 4 with $%
n=N_{0}$. The value of $G$ has been computed, it does not depend on $%
\varepsilon$. However, one needs to recompute $N_{0}:=N_{0}\left(
\varepsilon^{\prime}\right) $ such that $G\Delta_{N_{0}}^{2\alpha-\beta}<%
\varepsilon^{\prime}$. Then we can implement Steps 5 to 8 without change.
One obtains an output that, as before, can be transformed into (\ref%
{incrr_algo_II}) via the recursions (\ref{Eq:INC_1}), yielding $\{\hat{X}%
^{N_{0}\left( \varepsilon^{\prime }\right) }(t):t\in D_{N_{0}\left(
\varepsilon^{\prime}\right) }\}$ with a guaranteed error smaller than $%
\varepsilon^{\prime}$ in uniform norm with probability 1.

\section{Rough Differential Equations, Error Analysis, and The Proof of
Theorem \protect\ref{th:main}\label{Sect_RDE}}

The analysis in this section follows closely the discussion from \cite{Davie_2007} Section 3 and Section 7; see also \cite{FV_2010} Chapter 10. We
made some modifications to account for the drift of the process and also to
be able to explicitly calculate the constant $G$. Let us start with the
definition of a solution to \eqref{eq:main} using the theory of rough
differential equations. We first provide a definition of the solution of \eqref{eq:main} in a pathwise sense, following \cite{Davie_2007}.

\begin{definition}
\label{Def_RDE}$X(\cdot)$ is a solution of \eqref{eq:main} on $[0,1]$ if $%
X(0)=x(0)$ and for almost every sample path $\{ Z_{j}(\cdot):j=1,2,\dots, d \}$ it holds
\begin{align*}
& |X_{i}(t)-X_{i}(s)-\mu_{i}(X(s))(t-s)-\sum_{j=1}^{d^{\prime}}\sigma
_{i,j}(X(s))(Z_{j}(t)-Z_{j}(s)) \\
& -\sum_{j=1}^{d^{\prime}}\sum_{l=1}^{d}\sum_{m=1}^{d^{\prime}}\partial
_{l}\sigma_{i,j}(X(s))\sigma_{l,m}(X(s))A_{m,j}(s,t)|=o(t-s)
\end{align*}
for $i=1,2,\dots, d$ and $0\leq s<t\leq1$, where $A_{i,j}\left( \cdot\right) $
satisfies
\begin{equation}
A_{i,j}(r,t)=A_{i,j}(r,s)+A_{i,j}(s,t)+(Z_{i}(s)-Z_{i}(r))(Z_{j}(t)-Z_{j}(s))
\label{Eq:RDE_Area}
\end{equation}
for $0\leq r<s<t\leq1$.
\end{definition}

The previous definition is motivated by the following Taylor-type
development,

\begin{align*}
& X_{i}(t+h) \\
= &
X_{i}(t)+\int_{t}^{t+h}\mu_{i}(X(u))du+\sum_{j=1}^{d^{\prime}}\int_{t}^{t+h}%
\sigma_{i,j}(X(u))dZ_{j}(u) \\
\approx & X_{i}(t)+\int_{t}^{t+h}\mu_{i}(X(u))du \\
& +\sum_{j=1}^{d^{\prime}}\int_{t}^{t+h}\sigma_{i,j}\left( X(t)+\mu
(X(t))(u-t)+\sigma(X(t))(Z(u)-Z(t))\right) dZ_{j}(u) \\
\approx & X_{i}(t)+\mu_{i}(X(t))h+\sum_{j=1}^{d^{\prime}}\sigma
_{i,j}(X(t))(Z_{j}(t+h)-Z_{j}(t)) \\
& +\sum_{j=1}^{d^{\prime}}\sum_{l=1}^{d}\sum_{m=1}^{d^{\prime}}\partial
_{l}\sigma_{i,j}(X(t))\sigma_{l,m}(X(t))%
\int_{t}^{t+h}(Z_{m}(u)-Z_{m}(t))dZ_{j}(u).
\end{align*}
The previous Taylor development suggests defining $A_{i,j}(s,t):=%
\int_{s}^{t}(Z_{i}(u)-Z_{i}(s))dZ_{j}(u)$. Depending on how one interprets $%
A(s,t)$, e.g. via It\^o or Stratonovich integrals, one obtains a solution $%
X(\cdot)$ which is interpreted in the corresponding context.

In order to obtain the It\^o interpretation of the solution to equation (\ref%
{eq:main}) via definition (\ref{Def_RDE}) we shall interpret the integrals
in the sense of It\^o. In addition, as we shall explain, some technical
conditions (in addition to the standard Lipschitz continuity typically
required to obtain a strong solution) must be imposed in order to
enforce the existence of a unique solution to (\ref{Def_RDE}).

There are two sources of errors when using $\hat{X}^{n}$ in equation (\ref%
{eq:recursion1}) to approximate $X$. One is the discretization on the dyadic
grid, but assuming that $A_{i,j}\left( t_{k}^{n},t_{k+1}^{n}\right) $ is
known; this type of analysis is the one that is most common in the
literature on rough paths (see \cite{Davie_2007}). The second source of
error arises due to the fact that $A_{i,j}\left(
t_{k}^{n},t_{k+1}^{n}\right) $ is not known for $i\neq j$. Thus we divide the proof of
Theorem \ref{th:main} into two steps (two propositions), each dealing with
one source of error.

Similar to $\hat{X}^{n}(t)$, we define $\{X^{n}(t):t\in D_{n}\}$ by the
following recursion: given $X^{n}(0)=X(0)$,
\begin{align}
X_{i}^{n}(t_{k+1}^{n})= &
X_{i}^{n}(t_{k}^{n})+\mu_{i}(X^{n}(t_{k}^{n}))\Delta_{n}+\sum_{j=1}^{d^{%
\prime}}%
\sigma_{i,j}(X^{n}(t_{k}^{n}))(Z_{j}(t_{k+1}^{n})-Z_{j}(t_{k}^{n}))
\notag  \label{eq:recursion2} \\
& +\sum_{j=1}^{d^{\prime}}\sum_{l=1}^{d}\sum_{m=1}^{d^{\prime}}\partial
_{l}\sigma_{i,j}(X^{n}(t_{k}^{n}))%
\sigma_{l,m}(X^{n}(t_{k}^{n}))A_{m,j}(t_{k}^{n},t_{k+1}^{n}),
\end{align}
and for $t\in\lbrack0,1]$, we let $X^{n}(t)=X^{n}(\lfloor t\rfloor)$, where
in this context $\lfloor t\rfloor=\max\{s\in D_{n}:s\leq t\}$.

\begin{proposition}
\label{prop:error1} Under the conditions of Theorem \ref{th:main}, we can
compute a constant $G_{1}$ explicitly in terms of $M$, $||Z||_{\alpha}$ and $%
||A||_{2\alpha}$, such that for $n$ large enough
\begin{equation*}
||X^{n}(t)-X(t)||_{\infty}\leq G_{1} \Delta_{n}^{3\alpha-1}.
\end{equation*}
\end{proposition}

The proof of Proposition \ref{prop:error1} will be given after introducing
some definitions and key auxiliary results. We denote
\begin{equation*}
I_{i}^{n}(r,t):=X_{i}^{n}(t)-X_{i}^{n}(r)-\mu_{i}(X^{n}(r))(t-r)-\sum
_{j=1}^{d^{\prime}}\sigma_{i,j}(X^{n}(r))(Z_{j}(t)-Z_{j}(r))
\end{equation*}
and
\begin{equation*}
J_{i}^{n}(r,t):=I_{i}^{n}(r,t)-\sum_{j=1}^{d^{\prime}}\sum_{l=1}^{d}\sum
_{m=1}^{d^{\prime}}\partial_{l}\sigma_{i,j}(X^{n}(r))%
\sigma_{l,m}(X^{n}(r))A_{m,j}(r,t).
\end{equation*}

The following lemmas introduce the main technical results for the proof of
Proposition \ref{prop:error1}.

\begin{lemma} 
\label{lm:tech1} Under the conditions of Theorem \ref{th:main}, there exist
constants $C_{1}$, $C_{2}$ and $C_{3}$ that depend only on $M$, $%
||Z||_{\alpha }$ and $||A||_{2\alpha}$, such that for any large enough $n$
and $r,t\in D_{n}$,
\begin{align*}
||X^{n}(t)-X^{n}(r)||_{\infty} & \leq C_{1}|t-r|^{\alpha}, \\
|I^{n}(r,t)||_{\infty} & \leq C_{2}|t-r|^{2\alpha},
\end{align*}
and
\begin{equation*}
||J^{n}(r,t)||_{\infty}\leq C_{3}|t-r|^{3\alpha}.
\end{equation*}
\end{lemma}

\begin{proof}
For $r\leq s\leq t$, $r,s,t\in D_{n}$, we have the following
important recursions:
\begin{align*}
I_{i}^{n}(r,t)=&I_{i}^{n}(r,s)+I_{i}^{n}(s,t)+(\mu_{i}(X^{n}(s))-\mu_{i}%
(X^{n}(r)))(t-s)\\
&+\sum_{j=1}^{d^{\prime}}(\sigma_{i,j}(X^{n}(s))-\sigma
_{i,j}(X^{n}(r)))(Z_{j}(t)-Z_{j}(s))
\end{align*}
and
\begin{align}
&  J_{i}^{n}(r,t)\nonumber\label{eq:recur}\\
=  &  J_{i}^{n}(r,s)+J_{i}^{n}(s,t)+\left(  \mu_{i}(X^{n}(s))-\mu_{i}%
(X^{n}(r))\right)  (t-s)\nonumber\\
&  +\sum_{j=1}^{d^{\prime}}[\sigma_{i,j}(X^{n}(s))-\sigma_{i,j}(X^{n}%
(r))-\sum_{l=1}^{d}\partial_{l}\sigma_{i,j}(X^{n}(r))(X_{l}^{n}(s)-X_{l}%
^{n}(r))\nonumber\\
&  +\sum_{l=1}^{d}\partial_{l}\sigma_{i,j}(X^{n}(r))I_{l}^{n}(r,s)+\sum_{l=1}^{d}\partial_{l}\sigma_{i,j}(X^{n}(r))\mu_{i}(X^{n}(r))(s-r)](Z_{j}%
(t)-Z_{j}(s))\nonumber\\
&  +\sum_{j=1}^{d^{\prime}}\sum_{l=1}^{d}\sum_{m=1}^{d^{\prime}}\left[
\partial_{l}\sigma_{i,j}(X^{n}(s))\sigma_{l,m}(X^{n}(s))-\partial_{l}%
\sigma_{i,j}(X^{n}(r))\sigma_{l,m}(X^{n}(r))\right]  A_{m,j}(s,t)
\end{align}

We next divide the proof into two parts. We first prove that there exists a
small enough constant $\delta>0$ and three large enough constants
$C_{1}(\delta)$, $C_{2}(\delta)$ and $C_{3}(\delta)$, all independent of $n$,
such that for $|t-r|<\delta$, $||X^{n}(t)-X^{n}(r)||_{\infty}\leq C_{1}%
(\delta)|t-r|^{\alpha}$, $||I^{n}(r,t)||_{\infty}\leq C_{2}(\delta
)|t-r|^{2\alpha}$ and $||J^{n}(r,t)||_{\infty}\leq C_{3}(\delta)|t-r|^{3\alpha
}$. We prove it by induction. First we have $J^{n}(r,r)=0$ and $J^{n}%
(r,r+\Delta_{n})=0$. Suppose the result hold for all pairs of $r_{0},t_{0}\in
D_{n}$ with $|t_{0}-r_{0}|<|t-r|$. We then pick $s\in D_{n}$ as the largest
point between $r$ and $t$ such that $|s-r|\leq|t-r|/2$. Then we also have
$|s+\Delta_{n}-r|>|t-r|/2$ and $|t-(s+\Delta_{n})|<|t-r|/2$.
For simplicity of notation, we denote $\bar d=\max\{d,d^{\prime}\}$.\\

As
\begin{align*}
X_{i}^{n}(t)-X_{i}^{n}(s)=  &  J_{i}^{n}(s,t)+\mu_{i}(X^{n}(s))(t-s)+\sum
_{j=1}^{d^{\prime}}\sigma_{i,j}(X^{n}(s))(Z_{j}(t)-Z_{j}(s))\\
&  +\sum_{j=1}^{d^{\prime}}\sum_{l=1}^{d}\sum_{m=1}^{d^{\prime}}\partial
_{l}\sigma_{i,j}(X^{n}(s))\sigma_{l,m}(X^{n}(s))A_{m,j}(s,t),
\end{align*}
we have
\begin{align*}
&|X_{i}^{n}(t)-X_{i}^{n}(s)|\\
\leq& C_{3}(\delta)|t-s|^{3\alpha
}+M|t-s|+\bar dM||Z||_{\alpha}|t-s|^{\alpha}+\bar d^{3}M^{2}||A||_{2\alpha
}|t-s|^{2\alpha}\\
\leq& (C_{3}(\delta)\delta^{2\alpha}+M\delta^{1-\alpha}+\bar dM||Z||_{\alpha
}+\bar d^{3}M^{2}||A||_{2\alpha}\delta^{\alpha})|t-s|^{\alpha}\\
\leq& C_{1}(\delta)|t-s|^{\alpha}%
\end{align*}
for $C_{1}(\delta)\geq C_{3}(\delta)\delta^{2\alpha}+M\delta^{1-\alpha
}+\bar dM||Z||_{\alpha}+\bar d^{3}M^{2}||A||_{2\alpha}\delta^{\alpha}$.\newline And as
\[
I_{i}^{n}(s,t)=J_{i}^{n}(s,t)+\sum_{j=1}^{d^{\prime}}\sum_{l=1}^{d}\sum
_{m=1}^{d^{\prime}}\partial_{l}\sigma_{i,j}(X^{n}(s))\sigma_{l,m}%
(X^{n}(s))A_{m,j}(s,t),
\]
we have
\begin{align*}
|I_{i}^{n}(s,t)|& \leq C_{3}(\delta)|t-s|^{3\alpha}+\bar d^{3}M^{2}||A||_{2\alpha
}|t-s|^{2\alpha}\\
&\leq (C_{3}(\delta)\delta^{\alpha}+\bar d^{3}M^{2}||A||_{2\alpha
})|t-s|^{2\alpha}\leq C_{2}(\delta)|t-s|^{2\alpha}%
\end{align*}
for $C_{2}(\delta)\geq C_{3}(\delta)\delta^{\alpha}+\bar d^{3}M^{2}||A||_{2\alpha}%
$.\newline We now analyze the recursion \eqref{eq:recur} term by term. First,
\[
|\mu_{i}(X^{n}(s))-\mu_{i}(X^{n}(r))|\leq MC_{1}(\delta)|s-r|^{\alpha},
\]%
\[
|\sigma_{i,j}(X^{n}(s))-\sigma_{i,j}(X^{n}(r))-\sum_{l=1}^{d}\partial
_{l}\sigma_{i,j}(X^{n}(r))(X_{l}^{n}(s)-X_{l}^{n}(r))|\leq MC_{1}(\delta
)^{2}|s-r|^{2\alpha},
\]%
\[
|\sum_{l=1}^{d}\partial_{l}\sigma_{i,j}(X^{n}(r))I_{l}^{n}(r,s)|\leq
\bar dMC_{2}(\delta)|s-r|^{2\alpha},
\]
\[
\sum_{l=1}^{d}\partial_{l}\sigma_{i,j}(X^{n}(r))\mu_{i}(X^{n}(r))(s-r) \leq
\bar dM^2|s-r|
\]
and
\[
|\partial_{l}\sigma_{i,j}(X^{n}(s))\sigma_{l,m}(X^{n}(s))-\partial_{l}%
\sigma_{i,j}(X^{n}(r))\sigma_{l,m}(X^{n}(r))|\leq2M^{2}C_{1}(\delta
)|s-r|^{\alpha}.
\]
Then
\begin{align*}
&|J_{i}^{n}(r,t)|\\
\leq &  |J_{i}^{n}(r,s)|+|J_{i}^{n}(s,t)|\\
&  +(MC_{1}(\delta)+\bar dMC_{1}(\delta)^{2}||Z||_{\alpha}+\bar d^{2}MC_{2}%
(\delta)||Z||_{\alpha}+\bar d^2M^2||Z||_{\alpha}\\
& +2\bar d^{3}M^{2}C_{1}(\delta)||A||_{\alpha})|t-r|^{3\alpha}%
\end{align*}
Likewise, we have
\begin{align*}
& |J_{i}^{n}(s,t)|\\
\leq &  |J_{i}^{n}(s,s+\Delta_{n})|+|J_{i}^{n}(s+\Delta
_{n},t)|\\
&  +(MC_{1}(\delta)+\bar dMC_{1}(\delta)^{2}||Z||_{\alpha}+\bar d^{2}MC_{2}%
(\delta)||Z||_{\alpha}+\bar d^2M^2||Z||_{\alpha}\\
& +2\bar d^{3}M^{2}C_{1}(\delta)||A||_{\alpha})|t-s|^{3\alpha}\\
=  &  |J_{i}^{n}(s+\Delta_{n},t)|\\
&  +(MC_{1}(\delta)+\bar dMC_{1}(\delta)^{2}||Z||_{\alpha}+\bar d^{2}MC_{2}%
(\delta)||Z||_{\alpha}+\bar d^2M^2||Z||_{\alpha}\\
& +2\bar d^{3}M^{2}C_{1}(\delta)||A||_{\alpha})|t-s|^{3\alpha}.
\end{align*}
Then
\begin{align*}
& |J_{i}^{n}(r,t)|\\
\leq &  |J_{i}^{n}(r,s)|+|J_{i}^{n}(s+\Delta_{n},t)|\\
&  +2(MC_{1}(\delta)+\bar dMC_{1}(\delta)^{2}||Z||_{\alpha}+\bar d^{2}MC_{2}%
(\delta)||Z||_{\alpha}+\bar d^2M^2||Z||_{\alpha}\\
& +2\bar d^{3}M^{2}C_{1}(\delta)||A||_{\alpha})|t-s|^{3\alpha}\\
\leq &  \{  2^{1-3\alpha}C_{3}(\delta)+2(MC_{1}(\delta)+\bar dMC_{1}%
(\delta)^{2}||Z||_{\alpha}+\bar d^{2}MC_{2}(\delta)||Z||_{\alpha}\\
&+\bar d^2M^2||Z||_{\alpha}+2\bar d^{3}M^{2}C_{1}(\delta)||A||_{\alpha})\}  |t-s|^{3\alpha}\\
\leq &  C_{3}(\delta)|t-s|^{3\alpha},
\end{align*}
for
\begin{eqnarray*}
(1-2^{1-3\alpha})C_{3}(\delta)&\geq&2(MC_{1}(\delta)+\bar dMC_{1}(\delta)^{2}%
||Z||_{\alpha}+\bar d^{2}MC_{2}(\delta)||Z||_{\alpha}\\
&&+\bar d^2M^2||Z||_{\alpha}+2\bar d^{3}M^{2}C_{1}(\delta)||A||_{\alpha}).
\end{eqnarray*}
Therefore, if we deliberately choose $\delta$, $C_{1}(\delta)$,
$C_{2}(\delta)$ and $C_{3}(\delta)$ such that
\begin{align} \label{eq:in_delta}
C_{1}(\delta)  &  \geq C_{3}(\delta)\delta^{2\alpha}+M\delta^{1-\alpha
}+\bar dM||Z||_{\alpha}+\bar d^{3}M^{2}||A||_{2\alpha}\delta^{\alpha}\nonumber \\
C_{2}(\delta)  &  \geq C_{3}(\delta)\delta^{\alpha}+\bar d^{3}M^{2}||A||_{2\alpha}\nonumber \\
C_{3}(\delta)  &  \geq \frac{2}{1-2^{1-3\alpha}}(  MC_{1}(\delta
)+\bar dMC_{1}(\delta)^{2}||Z||_{\alpha}+\bar d^{2}MC_{2}(\delta)||Z||_{\alpha}\nonumber \\
& +\bar d^2M^2||Z||_{\alpha}+2\bar d^{3}M^{2}C_{1}(\delta)||A||_{\alpha})
\end{align}
Then we have for $|t-r|<\delta$,
\begin{align*}
||X^{n}(t)-X^{n}(r)||_{\infty}  &  \leq C_{1}(\delta)|t-r|^{\alpha},\\
||I^{n}(r,t)||_{\infty}  &  \leq C_{2}(\delta)|t-r|^{2\alpha},\\
||J^{n}(r,t)||_{\infty}  &  \leq C_{3}(\delta)|t-r|^{3\alpha},
\end{align*}
The existence of $\delta$, $C_{1}(\delta)$, $C_{2}(\delta)$ and $C_{3}(\delta)$, satisfying the system of inequalities \eqref{eq:in_delta}, follows from Lemma \ref{lm:pA}.

We now extend the analysis to the case when
$|t-r|>\delta$. For $n$ large enough ($\Delta_{n}<\delta/2$), if
$|t-r|>\delta$, we can always find points $s_{i}\in D_{n}$ and $r=s_{0}%
<s_{1}<\cdots<s_{k}=t$ such that $\max_{1\leq i\leq k}|s_{i}-s_{i-1}|<\delta$
and $\min_{1\leq i\leq k}|s_{i}-s_{i-1}|\geq\delta/2$. Then
\[
|X_{i}^{n}(t)-X_{i}^{n}(r)|\leq\sum_{l=1}^{k}|X_{i}^{n}(s_{l})-X_{i}%
^{n}(s_{l-1})|\leq kC_{1}(\delta)|t-r|^{\alpha}\leq\frac{2}{\delta}%
C_{1}(\delta)|t-r|^{\alpha}%
\]
Let $C_{1}=\frac{2}{\delta}C_{1}(\delta)$ and we can write $||X^{n}%
(t)-X^{n}(r)||_{\infty}\leq C_{1}|t-r|^{\alpha}$. Next,
\begin{align*}
|I_{i}^{n}(r,t)|\leq &  \sum_{l=1}^{k}\{|I_{i}^{n}(s_{l-1},s_{l})|+|(\mu
_{i}(X^{n}(s_{l}))-\mu_{i}(X^{n}(s_{0})))(s_{l}-s_{l-1})|\\
&  +|\sum_{j=1}^{d^{\prime}}(\sigma_{ii}(X^{n}(s_{l}))-\sigma_{ij}(X^{n}(s_{0})))(Z_j(s_{l+1}%
)-Z_j(s_{l}))|\}\\
\leq &  k[C_{2}(\delta)|t-r|^{2\alpha}+MC_{1}|t-r|^{1+\alpha}+dMC_{1}%
||Z||_{\alpha}|t-r|^{2\alpha}]\\
\leq &  \frac{2}{\delta}(C_{2}(\delta)+MC_{1}+\bar dMC_{1}||Z||_{\alpha
})|t-r|^{2\alpha}%
\end{align*}
By setting $C_{2}=\frac{2}{\delta}(C_{2}(\delta)+MC_{1}+\bar dMC_{1}||Z||_{\alpha
})$, we have $||I^{n}(r,t)||_{\infty}\leq C_{2}|t-r|^{2\alpha}$.\newline
Now following the same induction analysis on $J_{i}^{n}(s,t)$ as we did in the
case $|t-s|<\delta$, we have
\begin{align*}
|J_{i}^{n}(r,t)|\leq& \frac{2}{2^{3\alpha}}C_{3}|t-r|^{3\alpha}\\
&+2(MC_{1}+\bar dMC_{1}^{2}||Z||_{\alpha}+\bar d^{2}MC_{2}||Z||_{\alpha}+2\bar d^{3}M^{2}%
C_{1}||A||_{\alpha})|t-r|^{3\alpha}%
\end{align*}
If we choose
\[
C_{3}=\frac{2}{1-2^{1-3\alpha}}(MC_{1}+\bar dMC_{1}^{2}||Z||_{\alpha}+\bar d^{2}%
MC_{2}||Z||_{\alpha}+2\bar d^{3}M^{2}C_{1}||A||_{\alpha}),
\]
then $||J^{n}(r,t)||_{\infty}\leq C_{3}|t-s|^{3\alpha}$.\newline
\end{proof}

\begin{lemma}
\label{lm:tech2} Let $x(0)$ and $\tilde x(0)\in R^{d}$ be two different
vectors. We denote $X^{n}(t)$ and $\tilde X^{n}(t)$ for $t\in D_{n}$ as the $%
n$-th dyadic approximation defined by \eqref{eq:recursion2} with initial
value $x(0)$ and $\tilde x(0)$ respectively. Under the conditions of Theorem %
\ref{th:main}, there exists a constant $B$, independent of $n$, such that
for $t \in D_{n}$,
\begin{equation*}
||X^{n}(t)-\tilde X^{n}(t)-(X^{n}(0)-\tilde X^{n}(0))||_{\infty}\leq
Bt^{\alpha}||X^{n}(0)-\tilde X^{n}(0)||_{\infty}.
\end{equation*}
Moreover,
\begin{equation*}
||X^{n}(t)-\tilde X^{n}(t)||_{\infty}\leq(1+B)||X^{n}(0)-\tilde
X^{n}(0)||_{\infty}.
\end{equation*}
\end{lemma}

\begin{proof}
Let
\[
Y_{i,h}^{n}(t)=\frac{X_{i}^{n}(t)-\tilde X_{i}^{n}(t)}{||X_{h}^{n}(0)-\tilde
X_{h}^{n}(0)||_{\infty}}%
\]
We define $0/0=0$.\newline Then following the recursion \eqref{eq:recursion2},
we have
\begin{align}
\label{eq:recursion5} &  Y_{i}^{n}(t_{k+1}^{n})\nonumber\\
=  &  Y_{i}^{n}(t_{k}^{n})+\frac{\mu_{i}(X^{n}(t_{k}^{n}))-\mu_{i}(\tilde
X^{n}(t_{k}^{n}))}{||X^{n}(0)-\tilde X^{n}(0)||_{\infty}}\Delta
_{n}\nonumber\\
&+\sum_{j=1}^{d^{\prime}}\frac{\sigma_{i,j}(X^{n}(t_{k}^{n}))-\sigma
_{i,j}(\tilde X^{n}(t_{k}^{n}))}{||X^{n}(0)-\tilde X^{n}(0)||_{\infty
}}(Z_{j}(t_{k+1}^{n})-Z_{j}(t_{k}^{n}))\nonumber\\
&  +\sum_{j=1}^{d^{\prime}}\sum_{l=1}^{d}\sum_{m=1}^{d^{\prime}}\frac
{\partial_{l}\sigma_{i,j}(X^{n}(t_{k}^{n}))\sigma_{l,m}(X^{n}(t_{k}%
^{n}))-\partial_{l}\sigma_{i,j}(\tilde X^{n}(t_{k}^{n}))\sigma_{l.m}(\tilde
X^{n}(t_{k}^{n}))}{||X^{n}(0)-\tilde X^{n}(0)||_{\infty}}A_{m,j}%
(t_{k}^{n},t_{k+1}^{n})
\end{align}
Then \eqref{eq:recursion2} and \eqref{eq:recursion5} together define an
recursion to generate $X^{n}$, $\tilde X^{n}$ and $Y^{n}$. Following Lemma
\ref{lm:tech1}, there exists a constant $B$ that depends only on $M$,
$||Z||_{\alpha}$ and $||A||_{2\alpha}$, such that
\[
||Y^{n}(t)-Y^{n}(0)||_{\infty}\leq Bt^{\alpha}.
\]
Thus,
\[
||X^{n}(t)-\tilde X^{n}(t)-(X^{n}(0)-\tilde X^{n}(0))||_{\infty}\leq
Bt^{\alpha}||X^{n}(0)-\tilde X^{n}(0)||_{\infty},
\]
and
\[
||X^{n}(t)-\tilde X^{n}(t)||_{\infty}\leq(1+B)||X^{n}(0)-\tilde X^{n}%
(0)||_{\infty}.
\]

\end{proof}

We are now ready to prove Proposition \ref{prop:error1}.

\begin{proof}
[Proof of Proposition \ref{prop:error1}]From Lemma \ref{lm:tech1} we have
$||X^{n}(t)-X^{n}(r)||_{\infty}\leq C_{1}|t-r|^{\alpha}$. By Arzela-Ascoli
Theorem, there exits a subsequence of $\{X^{n}\}$ that converges uniformly to
some continuous function $X$ on $[0,1]$. Moreover we have
$||X(t)-X(r)||_{\infty}\leq C_{1}|t-r|^{\alpha}$ and
\begin{align*}
&  |X_{i}(t)-X_{i}(r)-\mu_{i}(X(r)-\sum_{j=1}^{d^{\prime}}\sigma
_{i,j}(X(r))(Z_{j}(t)-Z_{j}(r))\\
&  -\sum_{j=1}^{d^{\prime}}\sum_{l=1}^{d}\sum_{m=1}^{d^{\prime}}\partial
_{l}\sigma_{i,j}(X(r))\sigma_{l.m}(X(r))A_{m,j}(r,t)|<C_{2}|t-r|^{3\alpha}%
\end{align*}
Therefore, the limit $X$ is a solution to the SDE.\newline Let $X^{n,(s)}%
(t;X(s)):=X^{n}(t-s)|X^{n}(0)=X(s)$. Specifically, we have $X^{n,(0)}%
(t;X(0))=X^{n}(t)$ with $X^{n}(0)=X(0)$, and $X^{n,(t)}(t;X(t))=X(t)$. Then we
can write
\[
X^{n}(t_{m}^{n})-X(t_{m}^{n})=\sum_{k=1}^{m}\left(  X^{n,(t_{k}^{n})}%
(t_{m}^{n};X(t_{k}^{n}))-X^{n,(t_{k-1}^{n})}(t_{m}^{n};X(t_{k-1}^{n}))\right)
\]
By Lemma \ref{lm:tech2}, $||X^{n,(t_{k}^{n})}(t_{m};X(t_{k}^{n}%
))-X^{n,(t_{k-1}^{n})}(t_{m};X(t_{k-1}^{n}))||_{\infty}\leq(1+B)||X(t_{k}%
^{n})-X^{n,t_{k-1}^{n}}(t_{k}^{n};X(t_{k-1}^{n}))||_{\infty}$. We also have
\begin{align*}
&  |X_{i}(t_{k}^{n})-X_{i}^{n,(t_{k-1}^{n})}(t_{k}^{n};X(t_{k-1}^{n}))|\\
=  &  |X_{i}(t_{k}^{n})-X_{i}(t_{k-1}^{n})-\mu_{i}(X(t_{k-1}^{n})(t_{k}%
^{n}-t_{k-1}^{n})-\sum_{j=1}^{d^{\prime}}\sigma_{i,j}(X(t_{k-1}^{n}%
))(Z_{j}(t_{k}^{n})-Z_{j}(t_{k-1}^{n}))\\
&  -\sum_{j=1}^{d^{\prime}}\sum_{l=1}^{d}\sum_{m=1}^{d^{\prime}}\partial
_{l}\sigma_{i,j}(X(t_{k-1}^{n}))\sigma_{l,m}(X(t_{k-1}^{n}))A_{m,j}%
(t_{k-1}^{n},t_{k}^{n})|\\
\leq &  C_{3}|t_{k}^{n}-t_{k-1}^{n}|^{3\alpha}%
\end{align*}
Thus,
\begin{align*}
||X^{n}(t_{m}^{n})-X(t_{m}^{n})||_{\infty}  &  \leq\sum_{k=1}^{m}%
||X^{n,(t_{k}^{n})}(t_{m}^{n};X(t_{k}^{n}))-X^{n,(t_{k-1}^{n})}(t_{m}%
^{n};X(t_{k-1}^{n}))||_{\infty}\\
&  \leq m(1+B)C_{3}\Delta_{n}^{3\alpha}\\
&  \leq(1+B)C_{3}\Delta_{n}^{3\alpha-1}.
\end{align*}
\end{proof}

Next we turn to the analysis of the error induced by approximating the L\'{e}%
vy area.

\begin{proposition}
\label{prop:error2} Under the conditions of Theorem \ref{th:main}, we can
compute a constant $G_{2}$ explicitly in terms of $M$, $||Z||_{\alpha}$, $%
||A||_{2\alpha}$ and $\Gamma_{R}$, such that for $n$ large enough
\begin{equation*}
||\hat{X}^{n}(t)-X^{n}(t)||_{\infty}\leq G_{2}\Delta_{n}^{2\alpha-\beta},
\end{equation*}
where $\beta\in (1-\alpha, 2\alpha)$.
\end{proposition}

The proof of Proposition \ref{prop:error2} uses a similar technique as the
proof of Proposition \ref{prop:error1} and also relies on some auxiliary
results. Let
\begin{eqnarray*}
U_{i}^{n}(s,t)&:=&\hat{X}_{i}^{n}(t)-X_{i}^{n,(s)}(t;\hat{X}^{n}(s))\\
&&+\sum_{j=1}^{d^{\prime}}\sum_{l=1}^{d}\sum_{m=1}^{d^{\prime}}\partial_{l}%
\sigma_{i,j}(\hat{X}^{n}(s))\sigma_{l,m}(\hat{X}^{n}(s))R_{m,j}^{n}(s,t).
\end{eqnarray*}

We first prove the following technical result.

\begin{lemma}
\label{lm:tech3} Under the conditions of Theorem \ref{th:main}, there exists
a constant $C_{4}$, that depends only on $M$, $||Z||_{\alpha}$, $%
||A||_{2\alpha }$ and $\Gamma_{R}$, such that
\begin{equation*}
||U^{n}(r,t)||_{\infty} \leq
C_{4}|t-r|^{\alpha+\beta}\Delta_{n}^{2\alpha-\beta}
\end{equation*}
\end{lemma}

\begin{proof}
For $0\leq r<s<t\leq1$, $r,s,t\in D_{n}$, we have
\begin{align*}
&U_{i}^{n}(r,t)\\
=&  U_{i}^{n}(r,s)+U_{i}^{n}(s,t)\\
&  +\left[  X_{i}^{n,(s)}(t;\hat{X}^{n}(s))-X_{i}^{n,(r)}(t;\hat{X}%
^{n}(r))-(\hat{X}_{i}^{n}(s)-X_{i}^{n,(r)}(s;\hat{X}^{n}(r)))\right] \\
&  -\sum_{j=1}^{d^{\prime}}\sum_{l=1}^{d}\sum_{m=1}^{d^{\prime}}\left(
\partial_{l}\sigma_{i,j}(\hat{X}^{n}(s))\sigma_{l,m}(\hat{X}^{n}%
(s))-\partial_{l}\sigma_{i,j}(\hat{X}^{n}(r))\sigma_{l,m}(\hat{X}%
^{n}(r))\right)  R_{m,j}^{n}(s,t)
\end{align*}
From Lemma \ref{lm:tech2},
\begin{align*}
&  |X_{i}^{n,(s)}(t;\hat{X}^{n}(s))-X_{i}^{n,(r)}(t;\hat{X}^{n}(r))-\left(
\hat{X}_{i}^{n}(s)-X_{i}^{n,(r)}(s;\hat{X}^{n}(r))\right)  |\\
\leq &  B|t-s|^{\alpha}||\hat{X}^{n}(s)-X^{n,(r)}(s;\hat{X}^{n}(r))||_{\infty}%
\end{align*}
From Lemma \ref{lm:tech1},
\begin{align*}
&  \left\vert \left(  \partial_{l}\sigma_{i,j}(\hat{X}^{n}(s))\sigma
_{l,m}(\hat{X}^{n}(s))-\partial_{l}\sigma_{i,j}(\hat{X}^{n}(r))\sigma
_{l,m}(\hat{X}^{n}(r))\right)  R_{m,j}^{n}(s,t)\right\vert \\
\leq &  2M^{2}C_{1}|s-r|^{\alpha}\Gamma_{R}|t-s|^{\beta}\Delta_{n}%
^{2\alpha-\beta}\\
\leq &  2M^{2}C_{1}\Gamma_{R}|t-r|^{\alpha+\beta}\Delta_{n}^{2\alpha-\beta}%
\end{align*}
Therefore,
\begin{align}
&  ||U^{n}(r,t)||_{\infty}\nonumber\label{eq:recursion3}\\
\leq &  ||U^{n}(r,s)||_{\infty}+||U^{n}(s,t)||_{\infty}+B|t-s|^{\alpha}%
||\hat{X}^{n}(s)-X^{n,(r)}(s;\hat{X}^{n}(r))||_{\infty}\nonumber\\
&  +2\bar d^{3}M^{2}C_{1}\Gamma_{R}|t-r|^{\alpha+\beta}\Delta_{n}^{2\alpha-\beta
}\nonumber\\
\leq &  ||U^{n}(r,s)||_{\infty}+||U^{n}(s,t)||_{\infty}+B|t-s|^{\alpha}%
||U^{n}(r,s)||_{\infty}\nonumber\\
&  +B|t-s|^{\alpha}\max_{i}\{|\sum_{j=1}^{d^{\prime}}\sum_{l=1}^{d}\sum
_{m=1}^{d^{\prime}}\partial_{l}\sigma_{i,j}(\hat{X}^{n}(r))\sigma_{l,m}%
(\hat{X}^{n}(r))R_{m,j}^{n}(r,s)|\}\nonumber\\
&  +2\bar d^{3}M^{2}C_{1}\Gamma_{R}|t-r|^{\alpha+\beta}\Delta_{n}^{2\alpha-\beta
}\nonumber\\
\leq &  (1+B|t-s|^{\alpha})||U^{n}(r,s)||_{\infty}+||U^{n}(s,t)||_{\infty
}\nonumber\\
&  +(B\bar d^{3}M^{2}\Gamma_{R}+2\bar d^{3}M^{2}C_{1}\Gamma_{R})|t-r|^{\alpha+\beta
}\Delta_{n}^{2\alpha-\beta}%
\end{align}
where $\bar d=\max\{d,d^{\prime}\}$.\\
Like the proof of Lemma \ref{lm:tech1}, we divide the proof into two parts. We
first prove that there exist a small enough constant $\delta>0$ and a large
enough constant $C_{4}(\delta)$, both independent of $n$, such that for
$|t-r|<\delta$, $|U^{n}(r,t)|\leq C_{4}(\delta)|t-r|^{\alpha+\beta}\Delta
_{n}^{2\alpha-\beta}$\textit{.} And we prove it by induction. First we have
$U_{t_{k}^{n},t_{k}^{n}}^{n}=0$ and $U_{t_{k}^{n},t_{k+1}^{n}}^{n}=0$. Suppose
the bound holds for all pairs $r_{0},t_{0}\in D_{n}$ with $|t_{0}%
-r_{0}|<|t-r|$. We pick $s\in D_{n}$ as the largest point between $r$ and $t$
such that $|s-r|\leq1/2|t-r|$. Then we also have $|(s+\Delta_{n})-r|>1/2|t-r|$
and $|t-(s+\Delta_{n})|<1/2|t-r|$.
\begin{align*}
||U^{n}(r,t)||_{\infty}\leq&(1+B|t-s|^{\alpha})||U^{n}(r,s)||_{\infty}%
+||U^{n}(s,t)||_{\infty}\\
&+(B\bar d^{3}M^{2}\Gamma_{R}+2\bar d^{3}M^{2}C_{1}\Gamma
_{R})|t-r|^{\alpha+\beta}\Delta_{n}^{2\alpha-\beta}%
\end{align*}
and
\begin{align*}
&  ||U^{n}(s,t)||_{\infty}\\
\leq &  (1+B\Delta_{n}^{\alpha})||U^{n}(s,s+\Delta_{n})||_{\infty}%
+||U^{n}(s+\Delta_{n},t)||_{\infty}\\
&+(B\bar d^{3}M^{2}\Gamma_{R}+2\bar d^{3}M^{2}%
C_{1}\Gamma_{R})|t-s|^{\alpha+\beta}\Delta_{n}^{2\alpha-\beta}\\
\leq &  ||U^{n}(s+\Delta_{n},t)||_{\infty}+(B\bar d^{3}M^{2}\Gamma_{R}+2\bar d^{3}%
M^{2}C_{1}\Gamma_{R})|t-r|^{\alpha+\beta}\Delta_{n}^{2\alpha-\beta}%
\end{align*}
Therefore,
\begin{align*}
&  ||U^{n}(r,t)||_{\infty}\\
\leq &  (1+B\delta^{\alpha})||U^{n}(r,s)||_{\infty}+||U^{n}(s+\Delta
_{n},t)||_{\infty}\\
&+2(B\bar d^{3}M^{2}\Gamma_{R}+2\bar d^{3}M^{2}C_{1}\Gamma
_{R})|t-r|^{\alpha+\beta}\Delta_{n}^{2\alpha-\beta}\\
\leq &  \frac{2+B\delta^{\alpha}}{2^{\alpha+\beta}}C_{4}(\delta)|t-r|^{\alpha
+\beta}\Delta_{n}^{2\alpha-\beta}+2(B\bar d^{3}M^{2}\Gamma_{R}+2\bar d^{3}M^{2}%
C_{1}\Gamma_{R})|t-r|^{\alpha+\beta}\Delta_{n}^{2\alpha-\beta}%
\end{align*}
If we pick $\delta$ and $C_{4}(\delta)$ such that
\[
B\delta^{\alpha}\leq2^{\alpha+\beta}-2
\]
and
\[
(1-\frac{2+B\delta^{\alpha}}{2^{\alpha+\beta}})C_{4}(\delta)\geq2(B\bar d^{3}%
M^{2}\Gamma_{R}+2\bar d^{3}M^{2}C_{1}\Gamma_{R}),
\]
Then $||U^{n}(r,t)||_{\infty}\leq C(\delta)|t-r|^{\alpha+\beta}\Delta
_{n}^{2\alpha-\beta}$. We next extend the result to the case when
$|t-r|>\delta$. We can always divide the interval $[r,t]$ into smaller
intervals of length less than $\delta$, specifically, for $n$ large enough, we
consider $r=s_{0}<s_{1}<\cdots<s_{k}=t$ where $s_{i}\in D_{n}$ and
$1/2\delta<|s_{i}-s_{i-1}|<\delta$ for $i=1,2,\dots,k$. Then $k<2|t-r|/\delta
\leq2/\delta$ and
\begin{align*}
&  ||U^{n}(r,t)||_{\infty}\\
\leq &  (1+B|s_{1}-s_{0}|^{\alpha})||U^{n}(s_{0},s_{0})||_{\infty}%
+||U^{n}(s_{1},s_{2})||_{\infty}\\
&+(B\bar d^{3}M^{2}\Gamma_{R}+2\bar d^{3}M^{2}C_{1}%
\Gamma_{R})|t-r|^{\alpha+\beta}\Delta_{n}^{2\alpha-\beta}\\
\leq &  \sum_{i=1}^{k}(1+B\delta^{\alpha})||U^{n}(s_{i-1},s_{i})||_{\infty
}+k(B\bar d^{3}M^{2}\Gamma_{R}+2\bar d^{3}M^{2}C_{1}\Gamma_{R})|t-r|^{\alpha+\beta
}\Delta_{n}^{2\alpha-\beta}\\
\leq &  (1+B\delta^{\alpha})C_{4}(\delta)\Delta_{n}^{2\alpha-\beta}\sum
_{i=1}^{k}|s_{i}-s_{i-1}|^{\alpha+\beta}\\
&+k(B\bar d^{3}M^{2}\Gamma_{R}+2\bar d^{3}%
M^{2}C_{1}\Gamma_{R})|t-r|^{\alpha+\beta}\Delta_{n}^{2\alpha-\beta}\\
\leq &  (1+B\delta^{\alpha})C_{4}(\delta)|t-r|^{\alpha+\beta}\Delta
_{n}^{2\alpha-\beta}+\frac{2}{\delta}(B\bar d^{3}M^{2}\Gamma_{R}+2\bar d^{3}M^{2}%
C_{1}\Gamma_{R})|t-r|^{\alpha+\beta}\Delta_{n}^{2\alpha-\beta}\\
\leq &  C_{4}|m-k|^{\alpha+\beta}\Delta_{n}^{2\alpha-\beta}%
\end{align*}
for $C_{4}\geq(1+B\delta^{\alpha})C_{4}(\delta)+2(B\bar d^{3}M^{2}\Gamma_{R}%
+2\bar d^{3}M^{2}C_{1}\Gamma_{R})/\delta$.\newline
\end{proof}

We are now ready to prove Proposition \ref{prop:error2}.

\begin{proof}
[Proof of Proposition \ref{prop:error2}]From Lemma \ref{lm:tech3}, we have
\[
||U^{n}(0,t)||_{\infty}\leq C_{4}t^{\alpha+\beta}\Delta_{n}^{2\alpha-\beta}.
\]
Then
\begin{align*}
|\hat{X}_{i}^{n}(t)-X_{i}^{n}(t)|  &  \leq|U_{i}^{n}(0,t)|+\sum_{j=1}^{d}%
\sum_{l=1}^{d}\sum_{m=1}^{d}|\partial_{l}\sigma_{i,j}(X(0))\sigma
_{l,m}(X(0))||R_{m,j}^{n}(0,t)|\\
&  \leq C_{4}t^{\alpha+\beta}\Delta_{n}^{2\alpha-\beta}+\bar d^{3}M^{2}\Gamma
_{R}t^{\beta}\Delta_{n}^{2\alpha-\beta}\\
&  \leq(C_{4}+\bar d^{3}M^{2}\Gamma_{R})\Delta_{n}^{2\alpha-\beta}.
\end{align*}
\end{proof}

\section{Numerical implementation}
We conducted some numerical experiments. The goal to is demonstrate that the algorithms are implementable and correct.
We would also like to explain some limitations in our implementation process and hope these would provide directions for future improvement of the framework developed here.

\begin{enumerate}
\item For values of ${X(t):0 \leq t \leq1}$ which fluctuate around numerical values around, say 1, (assuming that drift and diffusion coefficients also take these values), Procedure A obtains a value of the parameter $G$ of order $10^3$. Thus for a reasonable level of accuracy, doing the computations implied by this size of $G$, one would generate about $20$ wavelet levels, which corresponds to about $2^{20}$ normal random variables. This amount is manageable in a standard single processor, but the amount could go out of hand in a standard computing environment if $G$ is of size, say 100. A potential way to mitigate this issue would be to simulate a properly scaled down version of the path and scale everything back once we have simulated the path, or, alternatively, to make this portion of the procedure run in parallel computing cores.
\item We have some freedom in picking the parameter $\alpha\in(1/3, 1/2)$ and $\beta \in (1-\alpha, 2\alpha)$, but there is a tradeoff. From Theorem 2.1,
  we want $2\alpha-\beta$ as close to $1/2$ as possible ($\alpha$ close to $1/2$ and $\beta$ close to $1-\alpha$). On the other hand, for the upper bound of $||Z||_{\alpha}$ and due to our procedure for finding $N_2$ (Section 5.2), we want $\alpha$ to be reasonably small and $\beta$ to be reasonably large. The point is, even if the theoretical complexity as $\epsilon$ decreases is driven by Theorem 2.1, we observed that in practice, given a fixed $\epsilon$, it might be better to choose $\alpha$ somewhat small, but within the range $(1/3,1/2)$.
\end{enumerate}

For our numerical experiments we simulated a 2 dimensional geometric Brownian Motion.
\begin{eqnarray*}
dX_1(t)&=&\mu_1X_1(t) dt + \sigma_1 X_1(t)dB_1(t)\\
dX_2(t) &=& \mu_2 X_2(t) dt +\rho \sigma_2 X_2(t) dB_1(t)+\sqrt{1-\rho^2}\sigma_2 X_2(t) dB_2(t)
\end{eqnarray*}
with initial value $X_1(0)$ and $X_2(0)$. We recognize that this SDE has a closed form solution, this is useful because we want to compare the output of our method and the output of an algorithm that does not take advantage of the Euler discretization. 
The previous SDE has the following closed form solution,
\begin{eqnarray*}
X_1(t)&=&X_1(0)\exp\left(\left(\mu_1-\sigma_1^2/2\right)t + \sigma_1 B_1(t)\right)\\
X_2(t) &=& X_2(0)\exp\left(\left(\mu_1-\sigma_2^2/2\right)t + \rho\sigma_2B_1(t)+ \sqrt{1-\rho^2}\sigma_2B_2(t)\right)
\end{eqnarray*}

Note that, the solution to this SDE is a continuous function of the Brownian motion under the uniform topology; so a Tolerance Enforced Simulation procedure using the closed form expression is much easier to design and, therefore, it can be used as a benchmark. Note that continuity of solution of the SDE under uniform norm does not imply that by only controlling the error of the wavelet approximation to Brownian motion in uniform metric, one can approximate to a given (deterministic) tolerance the error of the solution to the SDE \emph{when applying the Euler scheme}. In order to guarantee that the Euler scheme yields an error which is bounded by a user defined (deterministic) tolerance with probability one, one needs to apply our procedure.

Figure \ref{fig:plot} provides one numerical illustration of the performance of our algorithm. The light color is the path produced by our algorithm using the Euler scheme with a random truncation (which captures enough information to enforce a deterministic error in path space). The dark color is the simulation obtained by using a TES in uniform norm for the closed form expression. We observe that the two are indeed very close to each other. In particular, the  recursively constructed path is within $\epsilon$ ($\epsilon=0.1$) error bound of the true path. In fact, it appears that the constants are probably pessimistic in the sense that the actual error is much smaller that the prescribed guaranteed error. It might be worth to optimize the various tuning parameters in the algorithm, due to its complexity, however, we prefer to leave this task for future research.


\begin{figure}
\vspace{6pc}
\caption[]{Simulation of the Geometric Brownian Motion on $[0,1]$ ($X_1(0)=X_2(0)=1, \mu_1=\mu_2=1, \sigma_1=\sigma_2=0.5, \rho=0.25$)} \label{fig:plot}
\centering
\includegraphics[width=0.8\textwidth]{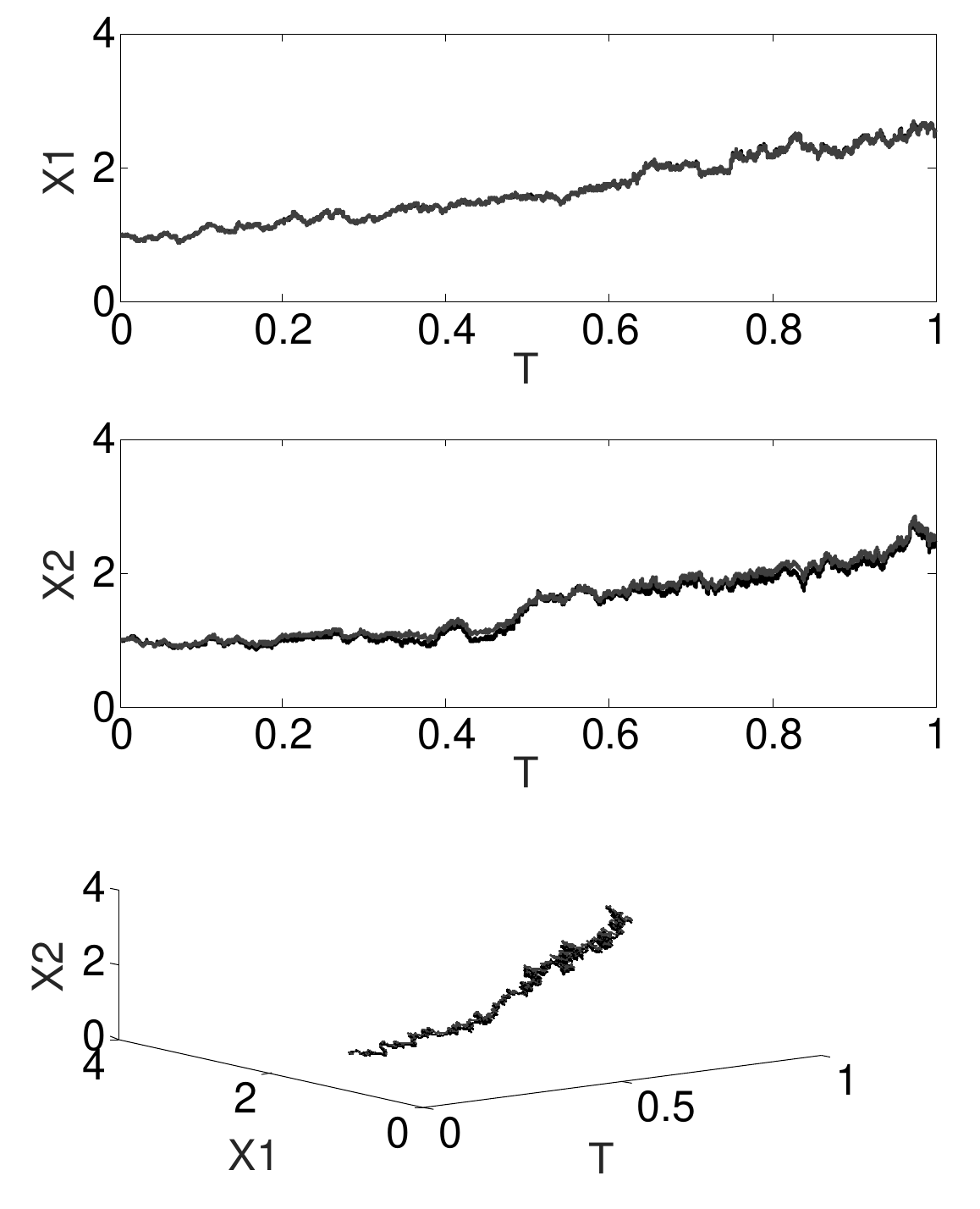}
\end{figure}

\section*{Acknowledgements}
The authors would like to thank the associate editor and two referees for their detailed and insightful comments
and suggestions. We also thank Yi Zhu for his help with the implementation of the algorithms. 
NSF support from grants CMMI 1069064 and DMS 1320550 is gratefully acknowledged.

\appendix

\section{Proofs of results in Section 3}\label{app:sec3}
We start by recalling the following
algebraic property of the L\'{e}vy areas: for each $0\leq r<s<t$
\begin{equation}
A_{i,j}\left( r,t\right) =A_{i,j}\left( r,s\right) +A_{i,j}\left( s,t\right)
+\left( Z_{i}\left( s\right) -Z_{i}\left( r\right) \right) \left(
Z_{j}\left( t\right) -Z_{j}\left( s\right) \right) .  \label{Eq-A-use}
\end{equation}
Using this property and a simple use of the Borel-Cantelli lemma we can
obtain the proof of Lemma \ref{Lem_Rep_Area}.

\begin{proof}
[Proof of Lemma \ref{Lem_Rep_Area}]
We use (\ref{Eq-A-use}) repeatedly. First,
note that%
\begin{align*}
A_{i,j}\left(  t_{k}^{n},t_{k+1}^{n}\right) = & A_{i,j}\left(  t_{2k}%
^{n+1},t_{2k+1}^{n+1}\right)  +A_{i,j}\left(  t_{2k+1}^{n+1},t_{2k+2}%
^{n+1}\right)\\
 &+\left(  Z_{i}\left(  t_{2k+1}^{n+1}\right)  -Z_{i}\left(
t_{2k}^{n+1}\right)  \right)  \left(  Z_{j}\left(  t_{2k+2}^{n+1}\right)
-Z_{j}\left(  t_{2k+1}^{n+1}\right)  \right)  .
\end{align*}
We continue, this time splitting $A_{i,j}\left(  t_{2k}^{n+1},t_{2k+1}%
^{n+1}\right)  $ and $A_{i,j}\left(  t_{2k+1}^{n+1},t_{2k+2}^{n+1}\right)  $,
thereby obtaining%
\begin{align*}
&  A_{i,j}\left(  t_{k}^{n},t_{k+1}^{n}\right) \\
=&\left(  Z_{i}\left(  t_{2k+1}^{n+1}\right)  -Z_{i}\left(  t_{2k}%
^{n+1}\right)  \right)  \left(  Z_{j}\left(  t_{2k+2}^{n+1}\right)
-Z_{j}\left(  t_{2k+1}^{n+1}\right)  \right) \\
&  +A_{i,j}\left(  t_{2^{2}k}^{n+2},t_{2^{2}k+1}^{n+2}\right)  +A_{i,j}\left(
t_{2^{2}k+1}^{n+2},t_{2^{2}k+2}^{n+2}\right) \\
&  +\left(  Z_{i}\left(  t_{2^{2}k+1}^{n+2}\right)  -Z_{i}\left(  t_{2^{2}%
k}^{n+2}\right)  \right)  \left(  Z_{j}\left(  t_{2^{2}k+2}^{n+2}\right)
-Z_{j}\left(  t_{2^{2}k+1}^{n+2}\right)  \right) \\
&  +A_{i,j}\left(  t_{2^{2}k+2}^{n+2},t_{2^{2}k+3}^{n+2}\right)
+A_{i,j}\left(  t_{2^{2}k+3}^{n+2},t_{2^{2}k+4}^{n+2}\right) \\
&  +\left(  Z_{i}\left(  t_{2^{2}k+3}^{n+2}\right)  -Z_{i}\left(  t_{2^{2}%
k+2}^{n+2}\right)  \right)  \left(  Z_{j}\left(  t_{2^{2}k+4}^{n+2}\right)
-Z_{j}\left(  t_{2^{2}k+3}^{n+2}\right)  \right)  .
\end{align*}
Suppose by iterating the previous splitting procedure $m$ times, we have
\begin{align}
&A_{i,j}\left(  t_{k}^{n},t_{k+1}^{n}\right) \label{Aux_A_Der}\\
=&\sum_{h=n+1}^{m}\sum
_{l=1}^{2^{h-n-1}}\left[Z_{i}\left(t_{2^{h-n}k+2l-1}^{h}\right)-Z_{i}\left(t_{2^{h-n}k+2l-2}^{h}\right)\right]
\left[Z_{j}\left(t_{2^{h-n}k+2l}^{h}\right)-Z_{j}\left(t_{2^{h-n}k+2l-1}^{h}\right)\right]\nonumber \\
& +\sum_{l=1}^{2^{m-n-1}}A_{i,j}\left(  t_{2^{m-n}k+2l-2}^{m},t_{2^{m-n}%
k+2l-1}^{m}\right)  +\sum_{l=1}^{2^{m-n-1}}A_{i,j}\left(  t_{2^{m-n}%
k+2l-1}^{m},t_{2^{m-n}k+2l}^{m}\right)  .\nonumber
\end{align}
Then for the $(m+1)$-th iteration, we have
\begin{align}
&A_{i,j}\left(  t_{k}^{n},t_{k+1}^{n}\right) \nonumber\\
=&\sum_{h=n+1}^{m}\sum_{l=1}^{2^{h-n-1}}\left[Z_{i}\left(t_{2^{h-n}k+2l-1}^{h}\right)-Z_{i}\left(t_{2^{h-n}k+2l-2}%
^{h}\right)\right]\left[Z_{j}\left(t_{2^{h-n}k+2l}^{h}\right)-Z_{j}\left(t_{2^{h-n}k+2l-1}^{h}\right)\right]\nonumber\\
&+\sum_{l=1}^{2^{m-n-1}}\{A_{i,j}\left(  t_{2(2^{m-n}k+2l-2)}^{m+1},t_{2(2^{m-n}k+2l-2)+1}^{m+1}\right) \nonumber\\
&+A_{i,j}\left(  t_{2(2^{m-n}k+2l-2)+1}^{m+1},t_{2(2^{m-n}k+2l-2)+2}^{m+1}\right) \nonumber \\
&+\left[Z_i\left( t_{2(2^{m-n}k+2l-2)+1}^{m+1}\right)-Z_i\left( t_{2(2^{m-n}k+2l-2)}^{m+1}\right)\right] \nonumber\\
&\times \left[Z_i\left( t_{2(2^{m-n}k+2l-2)+2}^{m+1}\right)-Z_i\left( t_{2(2^{m-n}k+2l-2)+1}^{m+1}\right)\right]\} \nonumber\\
&+\sum_{l=1}^{2^{m-n-1}}\{A_{i,j}\left(  t_{2(2^{m-n}k+2l-1)}^{m+1},t_{2(2^{m-n}k+2l-1)+1}^{m+1}\right) \nonumber\\
&+A_{i,j}\left(  t_{2(2^{m-n}k+2l-1)+1}^{m+1},t_{2(2^{m-n}k+2l-1)+2}^{m+1}\right) \nonumber \\
&+\left[Z_i\left( t_{2(2^{m-n}k+2l-1)+1}^{m+1}\right)-Z_i\left( t_{2(2^{m-n}k+2l-1)}^{m+1}\right)\right] \nonumber\\
&\times \left[Z_i\left( t_{2(2^{m-n}k+2l-1)+2}^{m+1}\right)-Z_i\left( t_{2(2^{m-n}k+2l-1)+1}^{m+1}\right)\right]\} \nonumber \\
=&\sum_{h=n+1}^{m+1}\sum_{l=1}^{2^{h-n-1}}\left[Z_{i}\left(t_{2^{h-n}k+2l-1}^{h}\right)-Z_{i}\left(t_{2^{h-n}k+2l-2}%
^{h}\right)\right]\left[Z_{j}\left(t_{2^{h-n}k+2l}^{h}\right)-Z_{j}\left(t_{2^{h-n}k+2l-1}^{h}\right)\right]\nonumber \\
& +\sum_{l=1}^{2^{m-n}}A_{i,j}\left(  t_{2^{m+1-n}k+2l-2}^{m+1},t_{2^{m+1-n}k+2l-1}^{m+1}\right)
+\sum_{l=1}^{2^{m+1-n-1}}A_{i,j}\left(  t_{2^{m+1-n}k+2l-1}^{m+1},t_{2^{m+1-n}k+2l}^{m+1}\right).\nonumber
\end{align}
Thus, \eqref{Aux_A_Der} holds by induction.

We next claim that
\begin{equation}
\sum_{l=1}^{2^{m-n-1}}A_{i,j}\left(  t_{2^{m-n}k+2l-2}^{m},t_{2^{m-n}%
k+2l-1}^{m}\right)  +\sum_{l=1}^{2^{m-n-1}}A_{i,j}\left(  t_{2^{m-n}%
k+2l-1}^{m},t_{2^{m-n}k+2l}^{m}\right)  \rightarrow0 \label{Lim_A}%
\end{equation}
almost surely as $m\rightarrow\infty$. To see this note that
\begin{align*}
&  P\left(  \left\vert \sum_{l=1}^{2^{m-n-1}}A_{i,j}\left(  t_{2^{h-n}%
k+2l-2}^{h},t_{2^{h-n}k+2l-1}^{h}\right)  \right\vert >1/m\right) \\
&  \leq m^{2}\sum_{l=1}^{2^{m-n-1}}E\left[A_{i,j}^{2}\left(  t_{2^{m-n}k+2l-2}%
^{m},t_{2^{m-n}k+2l-1}^{m}\right) \right]=m^{2}2^{m-n+1}E\int_{0}^{\Delta_{m}}%
Z_{i}^{2}\left(  s\right)  ds\\
&  =m^{2}2^{m-n}\Delta_{m}^{2}=2^{-n}m^{2}\Delta_{m}.
\end{align*}
Since $\sum_{m=1}^{\infty}m^{2}\Delta_{m}<\infty$, we conclude by
Borel-Cantelli's lemma that, almost surely, for $m$ large enough
\[
\left\vert \sum_{l=1}^{2^{m-n-1}}A_{i,j}\left(  t_{2^{h-n}k+2l-2}%
^{h},t_{2^{h-n}k+2l-1}^{h}\right)  \right\vert <1/m
\]
Thus, we have \eqref{Lim_A} holds almost surely and therefore, from
\eqref{Aux_A_Der}, by sending $m\rightarrow\infty$ we obtain the conclusion of
the lemma.
\end{proof}

\begin{proof}[Proof of Lemma \ref{lm:N1_finite_mean}]
Let $N_{i,1}=\max\{n\geq1:|W_{i,k}^{n}|>4\sqrt{n+1}\mbox{ for some }1\leq k\leq2^{n-1}\}$. Then $N_1=\max_i\{N_{i,1}\}$.
\begin{eqnarray*}
E[N_{i,1}]&=&\sum_{n=1}^{\infty}P(N_{i,1}\geq n)\\
&\leq& \sum_{n=1}^{\infty}\sum_{m=n}^{\infty}\sum_{k=1}^{2^{m-1}}P(|W_{i,k}^m|>4\sqrt{m+1})\\
&\leq& \sum_{n=1}^{\infty}\sum_{m=n}^{\infty} 2^{m-1}\exp(-8m)\\
&\leq& \sum_{n=1}^{\infty}\frac{\exp(-(8-\log2)n)}{1-\exp(-(8-\log2))}\\
&=&\frac{\exp(-(8-\log2))}{(1-\exp(-(8-\log2)))^2}<\infty
\end{eqnarray*}
Thus $E[N]<\infty$. We also notice that $E[N_{1}]$ is independent of our choice of $\alpha$ and $\beta$.
\end{proof}

\begin{proof}[Proof of Lemma \ref{Lem_HN_alpha}]
For any interval $[t,t+\delta]\subset\lbrack0,1]$, there exists $m\in\mathbb{Z}^+$, such that $2^{-(m+1)}\leq
\delta\leq2^{-m}$. We next divide the analysis into two cases.\\
\begin{itemize}
\item[Case1.] There exist two level $m$ dyadic points $t_{k}^{m}$ and $t_{k+1}^{m}$,
such that $[t,t+\delta]\subset[t_{k}^{m},t_{k+1}^{m}]$.
\item[Case 2.] There exist three level $m$ dyadic points $t_{k}^{m}$, $t_{k+1}^{m}$ and $t_{k+1}^m$,
such that $t\in[t_{k}^{m},t_{k+1}^{m}]$ and $t+\delta\in [t_{1}^{m},t_{k+2}^{m}]$.
\end{itemize}
In Case 1, using the L\'{e}vy-Ciesielski construction, we have
\[
|Z(t+\delta)-Z(t)|\leq2^{-m}V^0+\sum_{n=1}^{m}2^{-m+\frac{n-1}{2}}V^{n}+\sum
_{n=m+1}^{\infty}2^{-\frac{n+1}{2}}V^{n}.
\]
Since $\delta\geq 2^{-(m+1)}$, we have
\begin{align*}
&\frac{|Z(t+\delta)-Z(t)|}{\delta^{\alpha}}  \\
\leq& 2^{2\alpha}\left(2^{-(1-\alpha)m-\alpha}V^0+\sum_{n=1}%
^{m}2^{-(1-\alpha)m-\alpha+\frac{n-1}{2}}V^{n}+\sum_{n=m+1}^{\infty}2^{-\frac{n+1}%
{2}+\alpha(m+1)}V^{n}\right)\\
\leq & 2^{2\alpha}\left(\sum_{n=0}^{m}2^{-(1-\alpha)n+\frac{n}{2}}V^{n}%
+\sum_{n=m+1}^{\infty}2^{-\frac{n}{2}+\alpha n}V^{n}\right)\\
\leq & 2^{2\alpha}\sum_{n=0}^{\infty}2^{-n(\frac{1}{2}-\alpha)}V^{n}.
\end{align*}
Similar to Case 1, in Case 2, we have
\begin{eqnarray*}
|Z(t+\delta)-Z(t)| &\leq& |Z(t_{k+1}^m)-Z(t)|+|Z(t+\delta)-Z(t_{k+1}^m)|\\
&\leq& 2\left(V^0+\sum_{n=1}^{m}2^{-m+\frac{n-1}{2}}V^{n}+\sum
_{n=m+1}^{\infty}2^{-\frac{n+1}{2}}V^{n}\right).
\end{eqnarray*}
Then
\[
\frac{|Z(t+\delta)-Z(t)|}{\delta^{\alpha}} \leq 2^{2\alpha+1}\sum_{n=0}^{\infty}2^{-n(\frac{1}{2}-\alpha)}V^{n}.
\]
As the interval $[t,t+\delta]$ is arbitrarily chosen, we obtain the result.
\end{proof}

\begin{proof}
[Proof of Lemma \ref{Lem_L_Abs_Sum}]
For $i\neq j$, let $N_{i,j,2}=\max\{n: |L_{i,j}^{n}(m)-L_{i,j}^{n}(l)|>(m-l)^{\beta}\Delta
_{n}^{2\alpha}\text{ for some }0\leq l<m\leq2^{n-1}\}$. Then $N_2=\max_{1\leq i,j\leq d^{\prime}, i\neq j}\{N_{i,j,2}\}$.\\

Fix any $(i,j)$ pair, Define
\[
\mathcal{C}_{n}=\{|L_{i,j}^{n}(m)-L_{i,j}^{n}(l)|>(m-l)^{\beta}\Delta
_{n}^{2\alpha}\text{ for some }0\leq l<m\leq2^{n-1}\}.
\]
We will show that the events $\{\mathcal{C}_{n}:n\geq0\}$ occur finitely many
times. Note that%
\begin{equation}
P\left(  \mathcal{C}_{n}\right)  \leq\sum_{0\leq l<m\leq2^{n-1}}2P\left(  \left(
L_{i,j}^{n}(m)-L_{i,j}^{n}(l)\right)  >(m-l)^{\beta}\Delta_{n}^{2\alpha
}\right)  . \label{Sum_PC}%
\end{equation}
Also observe that for fixed $m$ and $n$, $L_{i,j}^{n}\left(  m\right)  $ is
the sum of $m$ i.i.d. random variables, each of which is distributed as
$(Z_{i}(t_{1}^{n})-Z_{i}(t_{0}^{n}))(Z_{j}(t_{2}^{n})-Z_{j}(t_{1}^{n}))$ and
\[
E\left[\exp\left(  \theta(Z_{i}(t_{1}^{n})-Z_{i}(t_{0}^{n}))(Z_{j}(t_{2}^{n}%
)-Z_{j}(t_{1}^{n}))\right) \right] =\left(  1-\theta^{2}\Delta_{n}^{2}\right)
^{-1/2}.
\]
We apply Chernoff's bound and have%
\begin{align*}
&P\left(  \left(  L_{i,j}^{n}(m)-L_{i,j}^{n}(l)\right)  >(m-l)^{\beta}%
\Delta_{n}^{2\alpha}\right)\\
\leq& \exp\left(  -\theta\left(  m-l\right)
^{\beta}\Delta_{n}^{2\alpha}-\frac{1}{2}(m-l)\log\left(  1-\theta^{2}%
\Delta_{n}^{2}\right)  \right)  .
\end{align*}
Select $\theta=\theta^{\prime}\left(  m-l\right)  ^{-1/2}\Delta_{n}^{-1}$ for
$\theta^{\prime}\in\left(  0,1/4\right)  $%
\[
P\left(  \left(  L_{i,j}^{n}(m)-L_{i,j}^{n}(l)\right)  >(m-l)^{\beta}%
\Delta_{n}^{2\alpha}\right)  \leq\exp\left(  -\theta^{\prime}\left(
m-l\right)  ^{\beta-1/2}\Delta_{n}^{2\alpha-1}+1\right)  .
\]
Hence,
\begin{equation} \label{eq:upper}
P\left(  \mathcal{C}_{n}\right)  \leq\sum_{0\leq l<m\leq2^{n-1}}2\exp\left(
-\theta^{\prime}\left(  m-l\right)  ^{\beta-1/2}\Delta_{n}^{2\alpha
-1}+1\right)  \leq2^{2n}\exp\left(  -\theta^{\prime}2^{n\left(  1-2\alpha
\right)  }\right)  .
\end{equation}
We notice that $2\alpha<1$.
$$E[N_{i,j,2}]=\sum_{n=1}^{\infty} P(N_{i,j,2}\geq n)$$
and
$$P(N_{i,j,2} \geq n)\leq\sum_{m=n}^{\infty} P(C_m)$$
From \eqref{eq:upper}, we denote
$$M:=\min\left\{n: 2^{2n}\exp\left(  -\theta^{\prime}2^{n\left(1-2\alpha\right)}\right) < 1/4\right\}.$$
Then $M=o\left((1-2\alpha)^{-2}\right)$. We also notice that for $n \geq M$,
$$2^{2(n+1)}\exp\left(  -\theta^{\prime}2^{(n+1)\left(1-2\alpha\right)}\right) < \left(2^{2n}\exp\left(  -\theta^{\prime}2^{n\left(1-2\alpha\right)}\right)\right)^2$$
Thus, $\sum_{m=M+k}^{\infty} P(\mathcal{C}_m) \leq (1/2)^k$ and
$$E[N_{i,j,2}] \leq M-1+\sum_{n=M}^{\infty}\sum_{m=n}^{\infty}P(\mathcal{C}_m)\leq M$$
Thus, $E[N_2]=o\left((1-2\alpha)^{-2}\right)$.
 \end{proof}

\bigskip

The proof of Corollary \ref{Prop_Rep_R} follows directly from Lemma \ref%
{Lem_Rep_Area} and Lemma \ref{Lem_L_Abs_Sum}.

\begin{proof}
[Proof of Corollary \ref{Prop_Rep_R}]Using Lemma \ref{Lem_Rep_Area} we obtain
that%
\begin{equation}
R_{i,j}^{n}(t_{l}^{n},t_{m}^{n})=\sum_{k=l+1}^{m}\sum_{h=n+1}^{\infty}%
(L_{i,j}^{h}(2^{h-n}(k+1))-L_{i,j}^{h}(2^{h-n}k)). \label{Aux_R_1}%
\end{equation}
On the other hand, due to Lemma \ref{Lem_L_Abs_Sum} if $n\geq N_{2}$%
\begin{align*}
&\sum_{k=l+1}^{m}\sum_{h=n+1}^{\infty}|L_{i,j}^{h}(2^{h-n}(k+1))-L_{i,j}%
^{h}(2^{h-n}k)|\\
\leq&\sum_{k=l+1}^{m}\sum_{h=n+1}^{\infty}(2^{-n}(k+1)-2^{-n}%
k)^{\beta}\Delta_{h}^{2\alpha-\beta}<\infty
\end{align*}
because $\beta<2\alpha$. Thus (by Fubini's theorem) the order of the
summations in (\ref{Aux_R_1}) can be exchanged and we obtain the result.
\end{proof}

\begin{proof}
[Proof of Lemma \ref{Lem_Sum_Bounds}]We start by showing the bound on
$\Gamma_{R}$. By the definition of $\Gamma_{L}$, for any $n$%
\[
|L_{i,j}^{n}(m)-L_{i,j}^{n}(l)|\leq\Gamma_{L} (m-l)^{\beta}\Delta_{n}%
^{2\alpha}.
\]
Consequently, for any $0 \leq l < m \leq2^{n-1}$,
\begin{align*}
|R_{i,j}^{n}(t_{l}^{n},t_{m}^{n})|  &  \leq\sum_{h=n+1}^{\infty}\left\vert
L_{i,j}^{h}(2^{h-n}m)-L_{i,j}^{h}(2^{h-n}l)\right\vert \\
&  \leq\sum_{h=n+1}^{\infty}\Gamma_{L}(m-l)^{\beta}2^{(h-n)\beta}\Delta
_{h}^{2\alpha}=\Gamma_{L}(m-l)^{\beta}\Delta_{n}^{\beta}\sum_{h=n+1}^{\infty
}\Delta_{h}^{2\alpha-\beta}\\
&  =\Gamma_{L}(t_{m}^{n}-t_{l}^{n})^{\beta}\Delta_{n}^{2\alpha-\beta}%
\frac{2^{-(2\alpha-\beta)}}{1-2^{-(2\alpha-\beta)}}.
\end{align*}
Therefore, we conclude that
\begin{align*}
\Gamma_{R}  &  :=\max_{1\leq i,j\leq d^{\prime}}\sup_{n\geq0}\sup_{0\leq
s<t\leq1,s,t\in D_{n}}\frac{|R_{i,j}^{n}(s,t)|}{|t-s|^{\beta}\Delta
_{n}^{2\alpha-\beta}}\\
&  \leq\Gamma_{L}\frac{2^{-(2\alpha-\beta)}}{1-2^{-(2\alpha-\beta)}}.
\end{align*}

Let $r(n,l,m)=\min\{h: |t_{m}^{n}-t_{l}^{n}|\geq\Delta_{h}\}$. For simplicity of
notation, we define the following sequence of operators of time:
\[
\underline{s}^{h}(t_{l}^{n})=\min\{t_{k}^{h}: t_{k}^{h}\geq t_{l}^{n}\}
\]
\[
\bar{s}^{h}(t_{m}^{n})=\max\{t_{k}^{h}: t_{k}^{h} \leq t_{m}^{n}\}
\]
for $r(n,l,m) \leq h \leq n$.\newline Then
\begin{align*}
&  |A_{i,j}(t_{l}^{n}, t_{m}^{n})|\\
&  \leq|A_{i,j}(t_{l}^{n}, \underline{s}^{n-1}(t_{l}^{n}))|+|A_{i,j}%
(\underline{s}^{n-1}(t_{l}^{n}),\bar{s}^{n-1}(t_{m}^{n}))|+|A_{i,j}(\bar
{s}^{n-1}(t_{m}^{n}),t_{m}^{n})|\\
&  +|Z_{i}(\underline{s}^{n-1}(t_{l}^{n}))-Z_{i}(t_{l}^{n})||Z_{j}(\bar
{s}^{n-1}(t_{m}^{n}))-Z_{j}(\underline{s}^{n-1}(t_{l}^{n}))|\\
&  +|Z_{i}(\bar{s}^{n-1}(t_{m}^{n}))-Z_{i}(t_{l}^{n})||Z_{j}(t_{m}^{n}%
)-Z_{j}(\bar{s}^{n-1}(t_{m}^{n}))|
\end{align*}
Suppose by iterating the above procedure up to level  $\gamma$, where $r(n,l,m)<\gamma<n$, we have
\begin{eqnarray*}
&&|A_{i,j}(t_{l}^{n}, t_{m}^{n})|\\
&\leq& \sum_{h=\gamma+1}^{n} |A_{i,j}(\underline{s}^{h}(t_{l}^{n}), \underline {s}^{h-1}(t_{l}^{n}))|+|A_{i,j}(\underline{s}^{\gamma}(t_{l}^{n}),\bar{s}^{\gamma}(t_{m}^{n}))| \\
&& +\sum_{h=\gamma+1}^{n} |A_{i,j}(\bar{s}^{h}(t_{m}^{n}), \bar{s}^{h-1}(t_{m}^{n}))| \nonumber\\
&&+ \sum_{h=\gamma+1}^{n} |Z_{i}(\underline{s}^{h}(t_{l}^{n}))-Z_{i}(\underline{s}^{h-1}(t_{l}^{n}))||Z_{j}(\bar{s}^{h-1}(t_{m}^{n}))-Z_{j}(\underline{s}^{h-1}(t_{l}^{n}))| \\
&&  +\sum_{h=\gamma+1}^{n}|Z_{i}(\bar{s}^{h-1}(t_{m}^{n}))-Z_{i}(\underline{s}^{h}(t_{l}^{n}))||Z_{j}(\bar{s}^{h}(t_{m}^{n}))-Z_{j}(\bar{s}^{h-1}(t_{m}^{n}))|
\end{eqnarray*}
Then for level $\gamma-1$, as $\underline{s}^{h-1}(\underline{s}^h(t_l^n))=\underline{s}^{h-1}(t_l^n)$ and $\bar{s}^{h-1}(\bar{s}^h(t_m^n))=\bar{s}^{h-1}(t_m^n)$ for $h<n$, we have
\begin{eqnarray*}
&&|A_{i,j}(t_{l}^{n}, t_{m}^{n})|\\
&\leq& \sum_{h=\gamma+1}^{n} |A_{i,j}(\underline{s}^{h}(t_{l}^{n}), \underline {s}^{h-1}(t_{l}^{n}))|\\
&&+|A_{i,j}(\underline{s}^{\gamma}(t_{l}^{n}),\underline{s}^{\gamma-1}(\underline{s}^{\gamma}(t_{l}^{n})))|\\
&& + |A_{i,j}(\underline{s}^{\gamma-1}(\underline{s}^{\gamma}(t_{l}^{n})),\bar{s}^{\gamma-1}(\bar{s}^{\gamma}(t_{m}^{n})))|+|A_{i,j}(\bar{s}^{\gamma-1}(\bar{s}^{\gamma}(t_{m}^{n})),\bar{s}^{\gamma}(t_{m}^{n}))| \\
&& +|Z_{i}(\underline{s}^{\gamma-1}\underline{s}^{\gamma}((t_{l}^{n})))-Z_{i}(\underline{s}^{\gamma}(t_{l}^{n}))||Z_{j}(\bar{s}^{\gamma-1}(\bar{s}^{\gamma}((t_{m}^{n})))-Z_{j}(\underline{s}^{\gamma-1}(\underline{s}^{\gamma}(t_{l}^{n})))|\\
&&  +|Z_{i}(\bar{s}^{\gamma-1}(\bar{s}^{\gamma}(t_{m}^{n})))-Z_{i}(\underline{s}^{\gamma}(t_{l}^{n}))||Z_{j}(\bar{s}^{\gamma}(t_{m}^{n}))-Z_{j}(\bar{s}^{\gamma-1}(\bar{s}^{\gamma}(t_{m}^{n})))|\\
&& +\sum_{h=r(n,l,m)+1}^{n} |A_{i,j}(\bar{s}^{h}(t_{m}^{n}), \bar{s}^{h-1}(t_{m}^{n}))|\\
&=& \sum_{h=\gamma}^{n} |A_{i,j}(\underline{s}^{h}(t_{l}^{n}), \underline {s}^{h-1}(t_{l}^{n}))|+|A_{i,j}(\underline{s}^{\gamma-1}(t_{l}^{n}),\bar{s}^{\gamma-1}(t_{m}^{n}))| \\
&& +\sum_{h=\gamma}^{n} |A_{i,j}(\bar{s}^{h}(t_{m}^{n}), \bar{s}^{h-1}(t_{m}^{n}))| \nonumber\\
&&+ \sum_{h=\gamma}^{n} |Z_{i}(\underline{s}^{h-1}(t_{l}^{n}))-Z_{i}(\underline{s}^{h}(t_{l}^{n}))||Z_{j}(\bar{s}^{h-1}(t_{m}^{n}))-Z_{j}(\underline{s}^{h-1}(t_{l}^{n}))| \\
&&  +\sum_{h=\gamma}^{n}|Z_{i}(\bar{s}^{h-1}(t_{m}^{n}))-Z_{i}(\underline{s}^{h}(t_{l}^{n}))||Z_{j}(\bar{s}^{h}(t_{m}^{n}))-Z_{j}(\bar{s}^{h-1}(t_{m}^{n}))|
\end{eqnarray*}
Thus, the following inequality holds by induction.
\begin{align*}
&  |A_{i,j}(t_{l}^{n}, t_{m}^{n})|\\
\leq& \sum_{h=r(n,l,m)+1}^{n} |A_{i,j}(\underline{s}^{h}(t_{l}^{n}), \underline
{s}^{h-1}(t_{l}^{n}))|+|A_{i,j}(\underline{s}^{r(n,l,m)}(t_{l}^{n}),\bar{s}^{r(n,l,m)}(t_{m}^{n}))|\\
& +\sum_{h=r(n,l,m)+1}^{n} |A_{i,j}(\bar{s}^{h}(t_{m}^{n}), \bar{s}^{h-1}(t_{m}^{n}))|\\
&  + \sum_{h=r(n,l,m)+1}^{n} |Z_{i}(\underline{s}^{h-1}(t_{l}^{n}))-Z_{i}(\underline
{s}^{h}(t_{l}^{n}))||Z_{j}(\bar{s}^{h-1}(t_{m}^{n}))-Z_{j}(\underline
{s}^{h-1}(t_{l}^{n}))|\\
&  +\sum_{h=r(n,l,m)+1}^{n}|Z_{i}(\bar{s}^{h-1}(t_{m}^{n}))-Z_{i}(\underline{s}%
^{h}(t_{l}^{n}))||Z_{j}(\bar{s}^{h}(t_{m}^{n}))-Z_{j}(\bar{s}^{h-1}(t_{m}%
^{n}))|
\end{align*}

We make the following important observations,
\begin{align*}
\underline{s}^{h-1}(t_{l}^{n})-\underline{s}^{h}(t_{l}^{n})  &  =
\begin{cases}
0 & \text{if } \underline{s}^{h-1}(t_{l}^{n})=\underline{s}^{h}(t_{l}^{n})\\
\Delta_{h} & \text{otherwise}%
\end{cases}
\end{align*}
\begin{align*}
\bar{s}^{h}(t_{m}^{n})-\bar{s}^{h-1}(t_{m}^{n})  &  =
\begin{cases}
0 & \text{if } {s}^{h-1}(t_{m}^{n})=\bar{s}^{h}(t_{m}^{n})\\
\Delta_{h} & \text{otherwise}%
\end{cases}
\end{align*}
\begin{align*}
\bar{s}^{r(n,l,m)}(t_{m}^{n})-\underline{s}^{r(n,l,m)}(t_{l}^{n})  &  =
\begin{cases}
0 & \text{if } \underline{s}^{r(n,l,m)}(t_{l}^{n})=\bar{s}^{r(n,l,m)}(t_{m}^{n})\\
\Delta_{r(n,l,m)} & \text{otherwise.}%
\end{cases}
\end{align*}
Then
\begin{align*}
&  \frac{|A_{i,j}(t_{l}^{n}, t_{m}^{n})|}{(t_{m}^{n}-t_{l}^{n})^{2\alpha}}\\
\leq& \sum_{h=r+1}^{n} \Gamma_{R} \frac{\Delta_{h}^{2\alpha}}{\Delta
_{r(n,l,m)}^{2\alpha}} + \Gamma_{R} + \sum_{h=r+1}^{n} \Gamma_{R} \frac{\Delta
_{h}^{2\alpha}}{\Delta_{r(n,l,m)}^{2\alpha}} + \sum_{h=r+1}^{n}||Z||_{\alpha}%
^{2}\frac{\Delta_{h}^{\alpha}}{\Delta_{r(n,l,m)}^{\alpha}}+\sum_{h=r+1}%
^{n}||Z||_{\alpha}^{2}\frac{\Delta_{h}^{\alpha}}{\Delta_{r}^{\alpha}}\\
\leq& \Gamma_{R} \frac{2}{1-2^{-2\alpha}} + ||Z||_{\alpha}^{2}\frac
{2^{1-\alpha}}{1-2^{-\alpha}}.
\end{align*}
Therefore,
\begin{align*}
||A||_{2\alpha}  &  := \max_{1\leq i\leq j\leq d^{\prime}}\sup_{n\geq1}%
\sup_{0\leq s<t\leq1;s,t\in D_{n}}\frac{\left\vert A_{i,j}\left(  s\right)
\right\vert }{\left\vert t-s\right\vert ^{2\alpha}}\\
&  \leq\Gamma_{R} \frac{2}{1-2^{-2\alpha}} + ||Z||_{\alpha}^{2}\frac
{2^{1-\alpha}}{1-2^{-\alpha}}.
\end{align*}
\end{proof}

\section{Proofs of results in Section 5}
\subsubsection{Proof of results in Section \ref{Sub_CMGF_Tilting}}

\begin{proof}[Proofs of Lemma \ref{lm:lambda}]
We first notice that $t_{2k}^n=t_{k}^{n-1}$ for $k=0,1,2,\dots, 2^{n-1}$. From the L\'{e}vy-Ciesielski Construction, we have
$$Z_i(t_{2k-1}^n)= \frac{1}{2}(Z_i(t_{k-1}^{n-1})+Z_i(t_{k}^{n-1}))+\Delta_{n+1}^{1/2}W_{i,k}^n$$
Then
$$\Lambda_{i}^{n}(t_{2k-1}^{n})=Z_i(t_{2k-1}^n)-Z_i(t_{k-1}^{n-1})=\frac{1}{2}(Z_i(t_{k-1}^{n-1})-Z_i(t_{k}^{n-1}))+\Delta_{n+1}^{1/2}W_{i,k}^n$$
and
$$\Lambda_{i}^{n}(t_{2k}^{n})=Z_i(t_{k}^{n-1})-Z_i(t_{2k-1}^{n})=\frac{1}{2}(Z_i(t_{k-1}^{n-1})-Z_i(t_{k}^{n-1}))-\Delta_{n+1}^{1/2}W_{i,k}^n.$$
\end{proof}

Before we prove Corollary \ref{Cor_E_nm1}, we first provide the following auxiliary result which summarizes basic
computations of moment generating functions of quadratic forms of bivariate
Gaussian random variables.

\begin{lemma}
\label{Lem_Gen_Tilting_Gaussian}Suppose that $Y$ and $Z$ are i.i.d. $N\left(
0,1\right) $ random variables, then for any numbers $%
a_{1},a_{2},b,c_{1},c_{2}\in R$ define%
\begin{equation*}
\phi\left( a,b,c\right) :=E\exp\left(
a_{1}Y+a_{2}Z+bYZ+c_{1}Y^{2}+c_{2}Z^{2}\right) ,
\end{equation*}
then we have that if $\left\vert 2c_{i}\right\vert <1$ for $i=1,2$, and$\
\left\vert b\right\vert <\left( 1-2c_{1}\right) \left( 1-2c_{2}\right) $
\begin{align*}
\phi(a,b,c) =& \left( 1-2c_{1}\right) ^{-1/2}\left( 1-2c_{2}\right)
^{-1/2}\left( 1-(b\left( 1-2c_{1}\right) ^{-1/2}\left( 1-2c_{2}\right)
^{-1/2})^{2}\right) ^{-1/2} \\
& \times\exp\left( \frac{%
a_{1}^{2}(1-2c_{1})^{-1}+a_{2}^{2}(1-2c_{2})^{-1}+2a_{1}a_{2}b(1-2c_{1})^{-1}(1-2c_{2})^{-1}%
}{2(1-b^{2}(1-2c_{1})^{-1}(1-2c_{2})^{-1})}\right)
\end{align*}
Moreover, if we let
\begin{equation*}
P^{\prime}\left( Y\in dy,Z\in dz\right) =P\left( Y\in dy,Z\in dz\right)
\frac{\exp\left( a_{1}y+a_{2}z+byz+c_{1}y^{2}+c_{2}z^{2}\right) }{\phi\left(
\theta;a,b,c\right) },
\end{equation*}
then under $P^{\prime}\left( \cdot\right) $ we have that $\left( Y,Z\right) $
are distributed bivariate Gaussian with covariance matrix%
\begin{align*}
& \Sigma\left( a,b,c\right) \\
=& \frac{1}{1-b^{2}\left( 1-2c_{1}\right) ^{-1}\left( 1-2c_{2}\right)^{-1}}
\\
&\times\left(
\begin{array}{cc}
(1-2c_{1})^{-1} & b\left( 1-2c_{1}\right) ^{-1}\left( 1-2c_{2}\right) ^{-1}
\\
b\left( 1-2c_{1}\right) ^{-1}\left( 1-2c_{2}\right) ^{-1} & (1-2c_{2})^{-1}%
\end{array}
\right) ,
\end{align*}
and mean vector
\begin{equation*}
\mu\left( a,b,c\right) =\Sigma\left( a,b,c\right) \left(
\begin{array}{c}
a_{1} \\
a_{2}%
\end{array}
\right) .
\end{equation*}
\end{lemma}

\begin{proof}
First it follows easily that $E\exp\left(  c_{1}Y^{2}+c_{2}Z^{2}\right)
=(1-2c_{1})^{-1/2}(1-2c_{2})^{-1/2}$, and under the probability measure%
\[
P_{1}\left(  Y\in dy.Z\in dz\right)  =\frac{\exp\left(  c_{1}y^{2}+c_{2}%
z^{2}\right)  }{E\left[\exp\left(  c_{1}Y^{2}+c_{2}Z^{2}\right)\right]  }P\left(  Y\in
dy\right)  P\left(  Z\in dz\right)
\]
$Y$ and $Z$ are independent with distributions $N(1,(1-2c_{1})^{-1})$ and
$N(1,(1-2c_{2})^{-1})$, respectively. Therefore,
\begin{align*}
\phi\left(  a,b,c\right)  =& (1-2c_{1})^{-1/2}(1-2c_{2})^{-1/2}E_{1}%
\exp\left(  a_{1}Y+a_{2}Z+bYZ\right) \\
=& (1-2c_{1})^{-1/2}(1-2c_{2})^{-1/2}\\
&  \times E[\exp\{  a_{1}Y(1-2c_{1})^{-1/2}+a_{2}Z(1-2c_{2})^{-1/2}\\
& +b(1-2c_{1})^{-1/2}(1-2c_{2})^{-1/2}YZ\} ] .
\end{align*}
Now, given $\left\vert \theta\right\vert <1$ define $P_{2}\left(
\cdot\right)  $ via
\[
P_{2}\left(  Y\in dy,Z\in dz\right)  =\frac{P\left(  Y\in dy,Z\in dz\right)
\exp\left(  \chi yz\right)  }{E[\exp\left(  \chi YZ\right)]  }.
\]
Observe that
\[
P\left(  Y\in dy,Z\in dz\right)  \exp\left(  \chi yz\right)  =\frac{1}{2\pi
}\exp\left(  -y^{2}/2-z^{2}/2+\chi yz\right)
\]
and
\[
-y^{2}/2-z^{2}/2+\chi yz=-(y,z)\Sigma^{-1}\binom{y}{z}/2,
\]
where
\[
\Sigma^{-1}=\left(
\begin{array}
[c]{cc}%
1 & -\chi\\
-\chi & 1
\end{array}
\right)  ,
\]
and thus
\[
\Sigma=\frac{1}{1-\chi^{2}}\left(
\begin{array}
[c]{cc}%
1 & \chi\\
\chi & 1
\end{array}
\right)  .
\]
Therefore, under $P_{2}(\cdot)$, ($Y,Z$) is distributed bivariate normal with
mean zero and covariance matrix $\Sigma$, with
\[
\chi=b(1-2c_{1})^{-1/2}(1-2c_{2})^{-1/2}%
\]
and we also must have that if $\left\vert \chi\right\vert <1$,
\[
E\left[\exp\left(  \phi YZ\right)\right]  =\left(  1-\chi^{2}\right)  ^{-1/2}=\left(
1-(b\left(  1-2c_{1}\right)  ^{-1/2}\left(  1-2c_{2}\right)  ^{-1/2}%
)^{2}\right)  ^{-1/2}.
\]
Consequently, we conclude that
\begin{align*}
\phi\left(  a,b,c\right)   =& \left(  1-2c_{1}\right)  ^{-1/2}\left(
1-2c_{2}\right)  ^{-1/2}\left(  1-(b\left(  1-2c_{1}\right)  ^{-1/2}\left(
1-2c_{2}\right)  ^{-1/2})^{2}\right)  ^{-1/2}\\
&  \times E_{2}\left[\exp(a_{1}Y(1-2c_{1})^{-1/2}+a_{2}Z(1-2c_{2})^{-1/2})\right].
\end{align*}
The final expression for $\phi\left(  a,b,c\right)  $ is obtained from the
fact that
\begin{align*}
&  E_{2}\left[\exp(a_{1}Y(1-2c_{1})^{-1/2}+a_{2}Z(1-2c_{2})^{-1/2})\right]\\
=& \exp\left(  Var_{2}(a_{1}Y(1-2c_{1})^{-1/2}+a_{2}Z(1-2c_{2})^{-1/2}%
)/2\right)  .
\end{align*}
And $P^{\prime}\left(  \cdot\right)  $ is equivalent to a standard
exponentially tilting to the measure $P_{2}(\cdot)$ using as the natural
parameter the vector
\[
\left(  a_{1}(1-2c_{1})^{-1/2},a_{2}(1-2c_{2})^{-1/2}\right)  ,
\]
and thus under $P^{\prime}\left(  \cdot\right)  $ the covariance matrix is the
same as under $P_{2}(\cdot)$ and the mean vector is equal to $\mu\left(
a,b,c\right)  $.
\end{proof}

\bigskip

We now are ready to provide the proof of Corollary \ref{Cor_E_nm1}.

\begin{proof}
[Proof of Corollary \ref{Cor_E_nm1}]Let us examine the term of the form\\
$\Lambda_{i}^{n+m}\left(  t_{2r-1}^{n+m}\right)  \Lambda_{j}\left(
t_{2r}^{n+m}\right)$, for $i\neq j$,
\begin{align*}
&  \Lambda_{i}^{n+m}\left(  t_{2r-1}^{n+m}\right)  \Lambda_{j}\left(
t_{2r}^{n+m}\right) \\
=& (\Lambda_{i}^{n+m-1}(t_{r}^{n+m-1})/2+\Delta_{n+m+1}^{1/2}W_{i,r}%
^{n+m})(\Lambda_{j}^{n+m-1}(t_{r}^{n+m-1})/2-\Delta_{n+m+1}^{1/2}W_{j,r}%
^{n+m})\\
=& \Lambda_{i}^{n+m-1}(t_{r}^{n+m-1})\Lambda_{j}^{n+m-1}(t_{r}^{n+m-1}%
)/4-\Delta_{n+m+1}W_{i,r}^{n+m}W_{j,r}^{n+m}\\
&  +\Delta_{n+m+1}^{1/2}W_{i,r}^{n+m}\Lambda_{j}^{n+m-1}(t_{r}^{n+m-1}%
)/2-\Delta_{n+m+1}^{1/2}W_{j,r}^{n+m}\Lambda_{i}^{n+m-1}(t_{r}^{n+m-1})/2.
\end{align*}
Then, we have that Corollary \ref{Cor_E_nm1} follows immediately from Lemma
\ref{Lem_Gen_Tilting_Gaussian}.
\end{proof}

\bigskip

Finally, we provide the proof of Corollary \ref{Cor_E_nmL}.

\begin{proof}
[Proof of Corollary \ref{Cor_E_nmL}]Recall that for each $r\in\{1,2,...,2^{n}%
\}$,%
\[
\Lambda_{i}^{n}(t_{r}^{n}):=(Z_{i}(t_{r}^{n})-Z_{i}(t_{r-1}^{n})).
\]
So%
\begin{align*}
\Lambda_{i}^{n}(t_{2r-1}^{n})  &  =\Lambda_{i}^{n}(t_{r}^{n-1})/2+\Delta
_{n+1}^{1/2}W_{i,r}^{n},\\
\Lambda_{i}^{n}(t_{2r}^{n})  &  =\Lambda_{i}^{n}(t_{r}^{n-1})/2-\Delta
_{n+1}^{1/2}W_{i,r}^{n}.
\end{align*}
We perform the first iteration in full detail, the rest are immediate just
adjusting the notation. From Corollary \ref{Cor_E_nm1} we obtain that, for $i\neq j$,
\begin{align*}
&  E_{n+m-1}\exp\left(  \theta_{0}[L_{i,j}^{n+m}\left(  k^{\prime}\right)
-L_{i,j}^{n+m}\left(  k\right)  ]\right) \\
=& \exp\left(\frac{1}{2}\sum_{r=k+1}^{k^{\prime}}\frac{\theta_{0}^{2}\Delta_{n+m+1}%
}{4\left(  1-\theta_{0}^{2}\Delta_{n+m+1}^{2}\right)  }\Lambda_{i}\left(
t_{r}^{n+m-1}\right)  ^{2}+\frac{1}{2}\sum_{r=k+1}^{k^{\prime}}\frac
{\theta_{0}^{2}\Delta_{n+m+1}}{4\left(  1-\theta_{0}^{2}\Delta_{n+m}^{2}\right)
}\Lambda_{j}\left(  t_{r}^{n+m-1}\right)  ^{2}\right)\\
&  \times\exp\left(\sum_{r=k+1}^{k^{\prime}}\frac{\theta_{0}\Delta_{n+m+1}}{4\left(
1-\theta_{0}^{2}\Delta_{n+m+1}^{2}\right)  }\Lambda_{i}\left(  t_{r}%
^{n+m-1}\right)  \Lambda_{j}\left(  t_{r}^{n+m-1}\right) \right)\times(1-\theta
_{0}^{2}\Delta_{n+m+1}^{2})^{-(k^{\prime}-k)/2}.
\end{align*}
Using the definitions in (\ref{Many_Defs}) we have that the exponential component
\begin{align*}
&  \frac{1}{2}\sum_{r=k+1}^{k^{\prime}}\frac{\theta_{0}^{2}\Delta_{n+m+1}%
}{4\left(  1-\theta_{0}^{2}\Delta_{n+m+1}^{2}\right)  }\Lambda_{i}\left(
t_{r}^{n+m-1}\right)  ^{2}+\frac{1}{2}\sum_{r=k+1}^{k^{\prime}}\frac
{\theta_{0}^{2}\Delta_{n+m+1}}{4\left(  1-\theta_{0}^{2}\Delta_{n+m+1}^{2}\right)
}\Lambda_{j}\left(  t_{r}^{n+m-1}\right)  ^{2}\\
&  +\sum_{r=k+1}^{k^{\prime}}\frac{\theta_{0}\Delta_{n+m+1}}{4\left(
1-\theta_{0}^{2}\Delta_{n+m+1}^{2}\right)  }\Lambda_{i}\left(  t_{r}%
^{n+m-1}\right)  \Lambda_{j}\left(  t_{r}^{n+m-1}\right)
\end{align*}
is equal to
\begin{align*}
&  \sum_{r=1}^{2^{n+m-2}}[\eta_{1}\left(  t_{2r-1}^{n+m-1}\right)  \Lambda
_{i}\left(  t_{2r-1}^{n+m-1}\right)  ^{2}+\eta_{1}\left(  t_{2r}%
^{n+m-1}\right)  \Lambda_{i}\left(  t_{2r}^{n+m-1}\right)  ^{2}]\\
&  +\sum_{r=1}^{2^{n+m-2}}[\eta_{1}\left(  t_{2r-1}^{n+m-1}\right)
\Lambda_{j}\left(  t_{2r-1}^{n+m-1}\right)  ^{2}+\eta_{1}\left(
t_{2r}^{n+m-1}\right)  \Lambda_{j}\left(  t_{2r}^{n+m-1}\right)^{2}]\\
&  +\sum_{r=1}^{2^{n+m-2}}[\theta_{1}\left(  t_{2r-1}^{n+m-1}\right)
\Lambda_{i}\left(  t_{2r-1}^{n+m-1}\right)  \Lambda_{j}\left(  t_{2r-1}%
^{n+m-1}\right)\\
& +\theta_{1}\left(  t_{2r}^{n+m-1}\right)  \Lambda_{i}\left(
t_{2r}^{n+m-1}\right)  \Lambda_{j}\left(  t_{2r}^{n+m-1}\right)].
\end{align*}
We next expand each of the terms. To simplify the notation, we write
\[
x=W_{i,r}^{n+m-1}\text{ \ and \ }y=W_{j,r}^{n+m-1}.
\]
Define $\sqrt{\Delta}=\Delta_{n+m}^{1/2}$, put $u=\Lambda_{i}\left(
t_{r}^{n+m-2}\right)  $ and $v=\Lambda_{j}\left(  t_{r}^{n+m-2}\right)  $%
\begin{align*}
\Lambda_{i}\left(  t_{2r-1}^{n+m-1}\right)   &  =u/2+\sqrt{\Delta
}x\text{,\ \ \ \ }\Lambda_{i}\left(  t_{2r}^{n+m-1}\right)  =u/2-\sqrt{\Delta}
x,\\
\Lambda_{j}\left(  t_{2r}^{n+m-1}\right)   &  =v/2+\sqrt{\Delta}%
y\text{,\ \ \ \ }\Lambda_{j}\left(  t_{2r}^{n+m-1}\right)  =v/2-\sqrt{\Delta}
y.
\end{align*}
Now, for brevity let us write $\eta_{o}=\eta_{1}\left(  t_{2r-1}%
^{n+m-1}\right)  $ and $\eta_{e}=\eta_{1}\left(  t_{2r}^{n+m-1}\right)  $ (`o'
is used for odd, and `e' for even)
\begin{align*}
&  (\eta_{1}\left(  t_{2r-1}^{n+m-1}\right)  \Lambda_{i}\left(  t_{2r-1}%
^{n+m-1}\right)  ^{2}+\eta_{1}\left(  t_{2r}^{n+m-1}\right)  \Lambda
_{i}\left(  t_{2r}^{n+m-1}\right)  ^{2}\\
&  +\eta_{1}\left(  t_{2r-1}^{n+m-1}\right)  \Lambda_{j}\left(  t_{2r-1}%
^{n+m-1}\right)  ^{2}+\eta_{1}\left(  t_{2r}^{n+m-1}\right)  \Lambda
_{j}\left(  t_{2r}^{n+m-1}\right)  ^{2})\\
=& \left(  \eta_{o}\left(  u/2+\sqrt{\Delta} x\right)  ^{2}+\eta_{e}\left(
u/2-\sqrt{\Delta} x\right)  ^{2}+\eta_{o}\left(  v/2+\sqrt{\Delta} y\right)
^{2}+\eta_{e}\left(  v/2-\sqrt{\Delta} y\right)  ^{2}\right) \\
=& \frac{1}{4}u^{2}(\eta_{e}+\eta_{o})+\frac{1}{4}v^{2}(\eta_{e}+\eta_{o})
+u(\eta_{o}-\eta_{e})\sqrt{\Delta}x + v(\eta_{o}-\eta_{e})\sqrt{\Delta}y\\
&+(\eta_{e}+\Delta\eta_{o})\Delta x^{2}+(\eta_{e}+\eta_{o})\Delta y^{2}.
\end{align*}
Likewise, put $\theta_{o}=\theta_{1}\left(  t_{2r-1}^{n+m-1}\right)  $ and
$\theta_{e}=\theta_{1}\left(  t_{2r}^{n+m-1}\right)  $%
\begin{align*}
&  \theta_{1}\left(  t_{2r-1}^{n+m-1}\right)  \Lambda_{i}\left(
t_{2r-1}^{n+m-1}\right)  \Lambda_{j}\left(  t_{2r-1}^{n+m-1}\right)
+\theta_{1}\left(  t_{2r}^{n+m-1}\right)  \Lambda_{i}\left(  t_{2r}%
^{n+m-1}\right)  \Lambda_{j}\left(  t_{2r}^{n+m-1}\right)  \\
=&\theta_{o}\left(  u/2+\sqrt{\Delta}x\right)  \left(  v/2+\sqrt{\Delta
}y\right)  +\theta_{e}\left(  u/2-\sqrt{\Delta}x\right)  \left(
v/2-\sqrt{\Delta}y\right) \\
=&\frac{1}{4}uv(\theta_{e}+\theta_{o})+(\theta_{e}+\theta_{o})\Delta xy
+\frac{1}{2} v(\theta_{o}-\theta_{e})\sqrt{\Delta} x +\frac{1}{2}u(\theta
_{o}-\theta_{e})\sqrt{\Delta} y
\end{align*}
We then collect the terms free of $x$ and $y$ and obtain%
\[
\frac{u^{2}}{4}\left(  \eta_{e}+\eta_{o}\right)  +\frac{v^{2}}{4}\left(
\eta_{e}+\eta_{o}\right)  +\frac{uv}{4}(\theta_{e}+\theta_{o}).
\]
Now the coefficients of $x,y,x^{2},y^{2},$ and $xy$%
\begin{align*}
&  \{u(\eta_{o}-\eta_{e})+\frac{1}{2}v(\theta_{o}-\theta_{e})\}\sqrt{\Delta
}x+\{v(\eta_{o}-\eta_{e})+\frac{1}{2}u(\theta_{o}-\theta_{e})\}\sqrt{\Delta
}y\\
&  +(\eta_{e}+\eta_{o})\Delta x^{2}+(\eta_{e}+\eta_{o})\Delta y^{2}\\
&  +(\theta_{e}+\theta_{o})\Delta xy.
\end{align*}
And finally we can apply Lemma \ref{Lem_Gen_Tilting_Gaussian} to get the
corresponding results.
\end{proof}

\subsubsection{Proofs of results in Section \protect\ref{Sub_Cond_LD}}

\begin{proof}[Proof of Lemma \ref{Lemma_MGF_Clean}]
Recalling expression (\ref{E_n_FIN}), we
establish the bound for $E_{n}\left[\exp\left(  \theta_{0}\{L_{i,j}^{n+1}(k^{\prime
})-L_{i,j}^{n+m}(k)\}\right)\right]$, for $i\neq j$, by controlling the contribution of the term
\begin{equation}
{\displaystyle\prod\limits_{l=2}^{m}} {\displaystyle\prod\limits_{r=1}%
^{2^{n+m-l}}} C\left(  t_{r}^{n+m-l}\right)  . \label{Cont_Constants}%
\end{equation}
and the exponential term
\begin{equation}
\label{Cont_Exp}\exp\left(  \sum_{r=1}^{2^{n}}\theta_{m}(t_{r}^{n})\Lambda
_{i}(t_{r}^{n})\Lambda_{j}(t_{r}^{n}) + \sum_{r=1}^{2^{n}}\eta_{m}(t_{r}%
^{n})\left(  \Lambda_{i}(t_{r}^{n})^{2}+\Lambda_{j}(t_{r}^{n})^{2}\right)
\right)
\end{equation}
separately. \newline

We start by analyzing $\theta_{l}$ and $\eta_{l}$. From Corollary
\ref{Cor_E_nm1}, we have
\[
\theta_{1} =\frac{\theta_{0}}{4\left(  1-\theta_{0}^{2}\Delta_{n+m+1}%
^{2}\right)  } \mbox{ and } \eta_{1} =\frac{\theta_{0}^{2}\Delta_{n+m+1}%
}{8\left(  1-\theta_{0}^{2}\Delta_{n+m+1}^{2}\right)  }.
\]
We notice that $2\eta_{1} \leq\theta_{1}^{2}\Delta_{n+m+1} \leq(5/2)\eta_{1}$.

Let
\[
u=\max\{h: k^{\prime}-k>2^{h}\}.
\]
We also denote
\[
\underline{b}(l):=\min\{r: \theta_{l}(t_{r}^{n+m-l})>0\}
\]
and
\[
\bar{b}(l):=\max\{r: \theta_{l}(t_{r}^{n+m-l})>0\}.
\]

The strategy throughout the rest of the proof proceeds as follows. We have
that the $\theta_{l}(t_{r}^{n+m-l})$'s and $\eta_{l}(t_{r}^{n+m-l})$'s,
$r=1,2,\dots,2^{n+m-l}$, are nonnegative. We also have that for $l\leq u\wedge
m$, the number of positive $\theta_{l}(t_{r}^{n+m-l})$'s and $\eta_{l}%
(t_{r}^{n+m-l})$'s reduces by about a half at each step $l$ and also the
actual value of the positive $\theta_{l}(t_{r}^{n+m-l})$'s and $\eta_{l}%
(t_{r}^{n+m-l})$'s shrinks by at least $1/2$. We will establish that if $m>u$,
for $u<l\leq m$, there are at most two positive $\theta_{l}(t_{r}^{n+m-l})$'s
and two positive $\eta_{l}(t_{r}^{n+m-l})$'s and at each step $l$, their
values shrink by more than $2^{-3/2}$. Using these observations we will
establish some facts and then use them to estimate (\ref{Cont_Constants}) and
finally (\ref{Cont_Exp}). We now proceed to carry out this strategy. \newline

We first verify the following claims.\newline

\textbf{Claim 1}:

For $l\leq u$, we claim that $\theta_{l}(t_{r}^{n+m-l}),\eta_{l}(t_{r}%
^{n+m-l})\geq0$ for all $r=1,2,\dots,2^{n+m-l}$ and $\theta_{l}(t_{r}%
^{n+m-l})$'s are equal for $r\in(\underline{b}(l),\bar{b}(l))$ and we denote their
values as $\theta_{l}$. So, following the recursion in \eqref{Many_Defs} we
have that $\theta_{l}=\Delta_{l-1}\theta_{1}$. If $\theta_{l}(t_{\underline
{b}(l)}^{n+m-l})\neq\theta_{l}(t_{\underline{b}(l)+1}^{n+m-l})$, then $\theta
_{l}(t_{\underline{b}(l)}^{n+m-l})<\theta_{l}(t_{\underline{b}(l)+1}^{n+m-l}%
)=\theta_{l}$, and if $\theta_{l}(t_{\bar{b}(l)}^{n+m-l})\neq\theta_{l}%
(t_{\bar{b}(l)-1}^{n+m-l})$, then $\theta_{l}(t_{\bar{b}(l)}^{n+m-l})<\theta
_{l}(t_{\bar{b}(l)-1}^{n+m-l})=\theta_{l}$.
Likewise, $\eta_{l}(t_{r}^{n+m-l})$'s are equal for $r\in(\underline{b}(l)%
,\bar{b}(l))$; we denote their common values as $\eta_{l}$ and we have from
\eqref{Many_Defs} that $\eta_{l}=\Delta_{l-1}\eta_{1}$. If $\eta
_{l}(t_{\underline{b}(l)}^{n+m-l})\neq\eta_{l}(t_{\underline{b}(l)+1}^{n+m-l})$,
then $\eta_{l}(t_{\underline{b}(l)}^{n+m-l})<\eta_{l}(t_{\underline{b}(l)+1}%
^{n+m-l})$, and if $\eta_{l}(t_{\bar{b}(l)}^{n+m-l})\neq\eta_{l}(t_{\bar{b}(l)%
-1}^{n+m-l})$, then $\eta_{l}(t_{\bar{b}}^{n+m-l})<\eta_{l}(t_{\bar{b}(l)%
-1}^{n+m-l})$. In other words, at each step, $l$ for $l<u$, $\theta_{l}%
(t_{r}^{n+m-l})$ and $\eta_{l}(t_{r}^{n+m-l})$ decay at rate $1/2$ if it is
not at the boundary ($r\in(\underline{b}(l),\bar{b}(l))$), and the boundary ones
($\theta_{l}(t_{\underline{b}(l)}^{n+m-l})$, $\theta_{l}(t_{\bar{b}(l)}^{n+m-l})$
and $\eta_{l}(t_{\underline{b}(l)}^{n+m-l})$, $\eta_{l}(t_{\bar{b}(l)}^{n+m-l})$),
may decay at a faster rate.\newline

We now prove the claim by induction using the recursive relation in
\eqref{Many_Defs}. The claim is immediate for $\theta_{1}$ and $\eta_{1}$. Now
suppose it holds for $\theta_{l}(t_{r}^{n+m-l})$ and $\eta_{l}(t_{r}^{n+m-l}%
)$, $r=1,2,\dots,2^{n+m-l}$. We next show that the claim holds for
$\theta_{l+1}(t_{r}^{n+m-l-1})$, $r=1,2,\dots,2^{n+m-l-1}$, as well. We omit
the proof of $\eta_{l+1}(t_{r}^{n+m-l-1})$ here, as it follows exactly the
same line of analysis as $\theta_{l+1}(t_{r}^{n+m-l-1})$.\newline

We divide the analysis into five cases.\newline

\noindent\textbf{Case 1.}
$\theta_{l}\left(  t_{2r-1}^{m+n-l}\right)  =\theta_{l}\left(  t_{2r}%
^{m+n-l}\right)  $ and $\eta_{l}\left(  t_{2r-1}^{m+n-l}\right)  =\eta
_{l}\left(  t_{2r}^{m+n-l}\right)  $. Then $\theta_{+}^{l+1}\left(
t_{r}^{m+n-l}\right)  =2\theta_{l}\left(  t_{2r-1}^{m+n-l+1}\right)  $ and
$\theta_{-}^{l+1}\left(  t_{r}^{m+n-l}\right)  =0$. Likewise $\eta_{+}%
^{l+1}\left(  t_{r}^{m+n-l}\right)  =2\eta_{l}\left(  t_{2r-1}^{m+n-l+1}%
\right)  $ and $\eta_{-}^{l+1}\left(  t_{r}^{m+n-l}\right)  =0$. From
\eqref{Many_Defs}, we have $\theta_{l}\left(  t_{r}^{m+n-l-1}\right)
=\theta_{l-1}\left(  t_{2r-1}^{m+n-l+1}\right)  /2$ and $\eta_{l}\left(
t_{r}^{m+n-l-1}\right)  =\eta_{l-1}\left(  t_{2r-1}^{m+n-l+1}\right)
/2$.\newline

\noindent\textbf{Case 2.} $\theta_{l}\left(  t_{2r-1}^{m+n-l}\right)  =0$,
$\theta_{l}\left(  t_{2r}^{m+n-l}\right)  >0$ and $\eta_{l}\left(
t_{2r-1}^{m+n-l}\right)  =0$, $\eta_{l}\left(  t_{2r}^{m+n-l}\right)  >0$.
Then we know that $2r=\underline{b}(l)$. We also have $\theta_{+}^{l+1}\left(
t_{r}^{m+n-l-1}\right)  =\theta_{l}\left(  t_{2r}^{m+n-l}\right)  $ and
$\theta_{-}^{l+1}\left(  t_{r}^{m+n-l-1}\right)  =-\theta_{l}\left(
t_{2r}^{m+n-l}\right)  $. Likewise, $\eta_{+}^{l+1}\left(  t_{r}%
^{m+n-l-1}\right)  =\eta_{l}\left(  t_{2r}^{m+n-l}\right)  $ and $\eta
_{-}^{l+1}\left(  t_{r}^{m+n-l-1}\right)  =-\eta_{l}\left(  t_{2r}%
^{m+n-l}\right)  $. We rewrite the expression for $\theta_{l+1}(t_{r}%
^{n+m-l-1})$ in \eqref{Many_Defs} as
\begin{align}
&\theta_{l+1}\left(  t_{r}^{m+n-l-1}\right) \nonumber\\
 =  &  \theta_{+}^{l+1}%
(t_{r}^{m+n-l-1})\frac{1}{4}+|\theta_{-}^{l+1}(t_{r}^{m+n-l-1})|\times\{
h_{l+1}(t_{r}^{m+n-l-1})|\eta_{-}^{l+1}(t_{r}^{m+n-l-1})| \nonumber\\
&  +\frac{1}{4}h_{l+1}(t_{r}^{m+n-l-1})|\theta_{-}^{l+1}(t_{r}^{m+n-l-1}%
)|\rho_{l+1}(t_{r}^{m+n-l-1}) \nonumber\\
&  +h_{l+1}(t_{r}^{m+n-l-1})\eta_{-}^{l+1}(t_{r}^{m+n-l-1})^{2}\frac
{\rho_{l+1}(t_{r}^{m+n-l-1})}{|\theta_{-}^{l+1}(t_{r}^{m+n-l-1})|}\} \nonumber\\
=  &  \theta_{l}(t_{2r}^{m+n-l})\times\{\frac{1}{4}+h_{l+1}(t_{r}^{m+n-l-1})\eta
_{l}(t_{2r}^{m+n-l})\nonumber\\
&  +\frac{1}{4}h_{l+1}(t_{r}^{m+n-l-1})\theta_{l}(t_{2r}^{m+n-l})\rho
_{l+1}(t_{r}^{m+n-l-1})\nonumber\\
&  +h_{l+1}(t_{r}^{m+n-l-1})\eta_{l}(t_{2r}^{m+n-l})^{2}\frac{\rho_{l+1}%
(t_{r}^{m+n-l-1})}{\theta_{l}(t_{2r}^{m+n-l})}\} \label{eq:reorg}
\end{align}
As
\[
\theta_{l}\Delta_{n+m-l}\leq\theta_{1}\Delta_{n+m-1} \leq\frac{1}{4}%
\]
and
\[
\eta_{l}\Delta_{n+m-l} \leq\eta_{1}\Delta_{n+m-1} \leq\frac{1}{48},
\]
then
$$\rho_{l+1}(t_{r}^{m+n-l-1})\leq \frac{\frac{1}{4}\frac{1}{4}}{1-\frac{1}{2}\frac{1}{48}}<\frac{1}{15}$$
and
$$h_{l+1}(t_{r}^{m+n-l-1})\theta_{l}(t_{2r}^{m+n-l})\leq\frac{\frac{1}{4}}{\left(1-\frac{1}{2}\frac{1}{48}\right)\left(1-\frac{1}{15^2}\right)}<\frac{1}{3}$$
Likewise,
$$h_{l+1}(t_{r}^{m+n-l-1})\eta_{l}(t_{2r}^{m+n-l})<1/95$$
and
$$\eta_{l}(t_{2r}^{m+n-l})\frac{\rho_{l+1}(t_{r}^{m+n-l-1})}{\theta_{l}(t_{2r}^{m+n-l})}<1/95$$
Plug these in \eqref{eq:reorg}, we have
\[
\frac{1}{4}\theta_{l}\left(  t_{2r}^{m+n-l}\right)  <\theta_{l+1}\left(
t_{r}^{m+n-l-1}\right)  < \frac{3}{10} \theta_{l} \leq\frac{3}{5} \theta
_{l+1}.
\]
\newline

\noindent\textbf{Case 3.} $\theta_{l}\left(  t_{2r-1}^{m+n-l}\right)  >0$,
$\theta_{l}\left(  t_{2r}^{m+n-l}\right)  =0$ and $\eta_{l}\left(
t_{2r-1}^{m+n-l}\right)  >0$, $\eta_{l}\left(  t_{2r}^{m+n-l}\right)  =0$.
Then we know that $2r-1=\bar{b}(l)$. Following the same line of analysis as in
Case 2, we have
\[
\frac{1}{4}\theta_{l}\left(  t_{2r}^{m+n-l}\right)  <\theta_{l+1}\left(
t_{r}^{m+n-l-1}\right)  <\frac{3}{10} \theta_{l} \leq\frac{3}{5} \theta
_{l+1}.
\]
\newline

\noindent\textbf{Case 4.} $0<\theta_{l}\left(  t_{2r-1}^{m+n-l}\right)
<\theta_{l}\left(  t_{2r}^{m+n-l}\right)  $ and $0<\eta_{l}\left(
t_{2r-1}^{m+n-l}\right)  <\theta_{l}\left(  t_{2r}^{m+n-l}\right)  $. Then we
know that $2r-1=\underline{b}(l)$. There exist $\xi<1$, such that $\theta
_{l}\left(  t_{2r-1}^{m+n-l}\right)  \leq\xi\theta_{l}\left(  t_{2r}%
^{m+n-l}\right)  =\xi\Delta_{l-1}\theta_{1}$ and $\eta_{l}\left(
t_{2r-1}^{m+n-l}\right)  \leq\xi\eta_{l}\left(  t_{2r}^{m+n-l}\right)
=\xi\Delta_{l-1}\eta_{1}$. From \eqref{Many_Defs}, we have
\begin{align*}
&\theta_{l+1}\left(  t_{r}^{m+n-l-1}\right)\\
\leq &  \theta_{+}^{l+1}%
(t_{r}^{m+n-l-1})\left\{  \frac{1}{4}+h_{l+1}(t_{r}^{m+n-l-1})\eta_{-}%
^{l+1}(t_{r}^{m+n-l-1})^{2}\frac{\rho_{l+1}(t_{r}^{m+n-l-1})}{\theta_{+}%
^{l+1}(t_{r}^{m+n-l-1})}\right\} \\
&  +|\theta_{-}^{l+1}(t_{r}^{m+n-l-1})|\times\{h_{l+1}(t_{r}^{m+n-l-1})|\eta
_{-}^{l+1}(t_{r}^{m+n-l-1})|\\
& +\frac{1}{4}h_{l+1}(t_{r}^{m+n-l-1})|\theta_{-}^{l+1}(t_{r}^{m+n-l-1}%
)|\rho_{l+1}(t_{r}^{m+n-l-1})\}.
\end{align*}
As $|\theta_{-}^{l+1}(t_{r}^{m+n-l-1})|\leq\theta_{l}$ and $|\eta_{-}%
^{l+1}(t_{r}^{m+n-l-1})|\leq\eta_{l}$, following the same calculation as in Case 2, it is easy to check that
\[
\theta_{l+1}\left(  t_{r}^{m+n-l-1}\right)  <\theta_{+}^{l+1}(t_{r}%
^{m+n-l-1})\left(  \frac{1}{4}+0.01\right)  +|\theta_{-}^{l+1}(t_{r}%
^{m+n-l-1})|\times0.05.
\]
Since $\theta_{l+1}\left(  t_{r}^{m+n-l-1}\right)  +|\theta_{-}^{l+1}%
(t_{r}^{m+n-l-1})|=\theta_{l}$, we have
\begin{align*}
\theta_{l+1}\left(  t_{r}^{m+n-l-1}\right)  &< \theta_{l}\left(  \left(
\frac{1}{4}+0.01-0.05\right)  (1+\xi)+0.05\right)\\
&=\frac{1}{2}\theta
_{l}\left(  \frac{1}{2}+0.02+0.42\xi\right)  <\frac{\theta_{l}}{2}%
=\theta_{l+1}.
\end{align*}

\noindent\textbf{Case 5.} $\theta_{l}\left(  t_{2r-1}^{m+n-l}\right)
>\theta_{l}\left(  t_{2r}^{m+n-l}\right)  >0$ and $\eta_{l}\left(
t_{2r-1}^{m+n-l}\right)  >\theta_{l}\left(  t_{2r}^{m+n-l}\right)  >0$. Then
we know that $2r=\bar{b}(l)$. Following the same line of analysis as in Case 4,
we have
\[
\theta_{l+1}\left(  t_{r}^{m+n-l-1}\right)  <\theta_{l+1}.
\]
We thus prove that the claim holds for $\theta_{l+1}(t_{r}^{m+n-l-1})$,
$r=1,2,\dots,2^{n+m-l-1}$, as well.\newline

We have established Claim 1. We now continue with a second claim.\\

\textbf{Claim 2: }

For $u<l<m$, we have at most two positive $\theta_{l}(t_{r}^{m+n-l})$'s,
namely $\theta_{l}(t_{\underline{b}(l)}^{m+n-l})$ and $\theta_{l}(t_{\bar{b}(l)%
}^{m+n-l})$. Notice that it is possible that $\underline{b}(l)=\bar{b}(l)$. We then
claim that if $\underline{b}\neq\bar{b}$,$\theta_{l}(t_{\underline{b}(l)}%
^{m+n-l})\leq\Delta_{l-1}\theta_{1}2^{-(l-u-1)/2}$ and $\theta_{l}(t_{\bar{b}(l)%
}^{m+n-l})\leq\Delta_{l-1}\theta_{1}2^{-(l-u-1)/2}$. Similarly $\eta
_{l}(t_{\underline{b}(l)}^{m+n-l})\leq\Delta_{l-1}\eta_{1}2^{-(l-u-1)/2}$ and
$\eta_{l}(t_{\bar{b}(l)}^{m+n-l})\leq\Delta_{l-1}\eta_{1}2^{-(l-u-1)/2}$. If
$\underline{b}(l)=\bar{b}(l)$, $\theta_{l}(t_{\underline{b}(l)}^{m+n-l})\leq
\Delta_{l-1}\theta_{1}2^{-(l-u-2)/2}$, $\theta_{l}(t_{\bar{b}(l)}^{m+n-l}%
)\leq\Delta_{l-1}\theta_{1}2^{-(l-u-2)/2}$ and $\eta_{l}(t_{\underline{b}(l)%
}^{m+n-l})\leq\Delta_{l-1}\eta_{1}2^{-(l-u-2)/2}$, $\eta_{l}(t_{\bar{b}(l)%
}^{m+n-l})\leq\Delta_{l-1}\eta_{1}2^{-(l-u-2)/2}$.\newline

We prove the claim by induction. We shall give the proof of $\theta_{l}%
(t_{r}^{m+n-l})$ only, as the proof of $\eta_{l}(t_{r}^{m+n-l})$ follows
exactly the same line of analysis. For $l=u$, we have the following
cases.\newline

\textbf{i)} $\bar{b}(l)=\underline{b}(l)+2$, $\underline{b}(l)$ is odd. In this case,
$\theta_{l+1}(t_{(\underline{b}(l)+1)/2}^{m+n-l-1})<\Delta_{l}\theta_{1}$, which
follows from the analysis in Case 4 for $l\leq u$. And $\theta_{l+1}%
(t_{(\bar{b}(l)+1)/2}^{m+n-l-1})<(3/5)\Delta_{l}\theta_{1}$, following the
analysis in Case 3 for $l\leq u$.\\

\textbf{ii)} $\bar{b}(l)=\underline{b}(l)+2$, $\underline{b}(l)$ is even. In this case,
$\theta_{l+1}(t_{\underline{b}(l)/2}^{m+n-l-1})<(3/5)\Delta_{l}\theta_{1}$, which
follows from the analysis in Case 2 for $l\leq u$. And $\theta_{l+1}%
(t_{\bar{b}(l)/2}^{m+n-l-1})<\Delta_{l}\theta_{1}$, following the analysis in
Case 5, for $l\leq u$.\\

\textbf{iii)} $\bar{b}(l)=\underline{b}(l)+1$, $\underline{b}(l)$ is odd. In this case,
let $\bar{\theta}_{l}=\max\{\theta_{l}(t_{\underline{b}(l)}^{m+n-l}),\theta
_{l}(t_{\bar{b}(l)}^{m+n-l})\}$, Then following the same analysis as in Case 4 or
Case 5 for $l\leq u$ (depending on which one of $\theta_{l}(t_{\underline{b}(l)%
}^{m+n-l})$ and $\theta_{l}(t_{\bar{b}(l)}^{m+n-l})$ is smaller), we have
$\theta_{l+1}(t_{\bar{b}(l)/2}^{m+n-l-1})<\bar{\theta}_{l}/2\leq\Delta_{l}%
\theta_{1}$.\\

\textbf{iv)} $\bar{b}(l)=\underline{b}(l)+1$, $\underline{b}(l)$ is even. In this case,
$\theta_{l+1}(t_{\underline{b}(l)/2}^{m+n-l-1})<(3/5)\Delta_{l}\theta_{1}$, which
follows from the analysis in Case 2 for $l\leq u$. And $\theta_{l+1}%
(t_{(\bar{b}(l)+1)/2}^{m+n-l-1})<(3/5)\Delta_{l}\theta_{1}$, following the
analysis in Case 3 for $l\leq u$.\newline

Therefore, the claim holds for $u+1$. Suppose the claim holds for $l\geq u+1$.
Then when moving from level $l$ to level $l+1$, one of the following three cases can happen.\\

\textbf{a)} $\bar{b}(l)=\underline{b}(l)+1$ and $\underline{b}(l)$ is even. In this
case, following the analysis in Case 2 and Case 3 for $l\leq u$, we have
\[
\theta_{l+1}(t_{\underline{b}(l)/2}^{m+n-l-1})\leq\frac{3}{10}\theta
_{l}(t_{\underline{b}(l)}^{m+n-l})\leq\Delta_{l}\theta_{1}2^{-(l-u)/2}%
\]
and
\[
\theta_{l+1}(t_{(\bar{b}(l)+1)/2}^{m+n-l-1})\leq\frac{3}{10}\theta_{l}(t_{\bar
{b}(l)}^{m+n-l})\leq\Delta_{l}\theta_{1}2^{-(l-u)/2}.
\]

\textbf{b)} $\bar{b}(l)=\underline{b}(l)$. In this case, following the analysis in
Case 2 or Case 3 for $l\leq u$ (depending on whether $\underline{b}(l)$ is odd or
even), we have
\[
\theta_{l+1}(t_{\lceil\underline{b}(l)/2\rceil}^{m+n-l-1})\leq\frac{3}{10}%
\theta_{l}(t_{\underline{b}(l)}^{m+n-l})\leq\Delta_{l}\theta_{1}2^{-(l-u-1)/2}.
\]

\textbf{c)} $\bar{b}(l)=\underline{b}(l)+1$ and $\underline{b}(l)$ is odd. In this
case, we let $\bar{\theta}_{l}=\max\{\theta_{l}(t_{\underline{b}(l)}%
^{m+n-l}),\theta_{l}(t_{\bar{b}(l)}^{m+n-l})\}$, Then we can use the same
analysis as in Case 4 or Case 5 for $l\leq u$ (depending on which one of
$\theta_{l}(t_{\underline{b}(l)}^{m+n-l})$ and $\theta_{l}(t_{\bar{b}(l)}^{m+n-l})$
is smaller) to conclude that
\[
\theta_{l+1}(t_{\bar{b}(l)/2}^{m+n-l-1})<\frac{1}{2}\bar{\theta}_{l}\leq
\Delta_{l}\theta_{1}2^{-(l-u-1)/2}.
\]
We notice that case c) can happen only once.

We are now ready to control the contribution of the term
\eqref{Cont_Constants}. As $\Delta_{n+m-l+2}\eta_{+}^{l}(t_{r}^{n+m-l}%
)\leq1/30$ and $\rho_{l}(t_{r}^{n+m-l})<1/7$, we have when $m\leq u$
\begin{align*}
&  \prod_{l=2}^{m}\prod_{r=1}^{2^{n+m-l}}C(t_{r}^{n+m-l})\\
\leq &  \prod_{l=2}^{m}\prod_{r=1}^{2^{n+m-l}}\exp\left(  4\Delta
_{n+m-l+2}\eta_{+}^{l}(t_{r}^{n+m-l})+\rho_{l}(t_{r}^{n+m-l})^{2}\right) \\
\leq &  \prod_{l=2}^{m}\exp\left(  \left(  16\Delta_{n+m}\eta_{1}%
+\frac{(4\Delta_{n+m}\theta_{1})^{2}}{(1-8\Delta_{n+m}\eta_{1})^{2}}\right)
((k^{\prime}-k)\Delta_{l}+2)\right) \\
\leq &  \prod_{l=2}^{m}\exp\left(  \left(  \frac{11}{5}\frac{\gamma^{2}%
}{k^{\prime}-k}\Delta_{n}^{1-2\alpha^{\prime}}+\frac{6}{5}\frac{\gamma^{2}%
}{k^{\prime}-k}\Delta_{n}^{2-4\alpha^{\prime}}\right)  ((k^{\prime}%
-k)\Delta_{l}+2)\right).
\end{align*}
The last inequality follows from Corollary 2 that $\theta_{1}=\theta
_{0}/4(1-\theta_{0}^{2}\Delta_{n+m}^{2})$, $\eta_{1}=\theta_{0}^{2}%
\Delta_{n+m}/2(1-\theta_{0}^{2}\Delta_{n+m}^{2})$, and our choice of
$\theta_{0}=\gamma/(k^{\prime1/2}\Delta_{n}^{2\alpha^{\prime}}\Delta_{m})$.
Then, as $(k^{\prime}-k)^{-1}\leq2^{-m}$,
\begin{align*}
&  \prod_{l=2}^{m}\prod_{r=1}^{2^{n+m-l}}C(t_{r}^{n+m-l})\\
\leq &  \exp\left(  \frac{11}{5}\gamma^{2}\left(  \sum_{l=2}^{m}\Delta
_{l}+2(m-1)\Delta_{m}\right)  +\frac{6}{5}\gamma^{2}\left(  \sum_{l=2}%
^{u}\Delta_{l}+2(m-1)\Delta_{m}\right)  \right) \\
\leq &  \exp\left(  \frac{8}{25}\right).
\end{align*}
When $m>u$,
\begin{align*}
&  \prod_{l=2}^{m}\prod_{r=1}^{2^{n+m-l}}C(t_{r}^{n+m-l})\\
\leq &  \prod_{l=2}^{m}\prod_{r=1}^{2^{n+m-l}}\exp\left(  4\Delta
_{n+m-l+2}\eta_{+}^{l}(t_{r}^{n+m-l})+\rho_{l}(t_{r}^{n+m-l})^{2}\right) \\
\leq &  \prod_{l=2}^{u}\exp\left(  \left(  \frac{11}{5}\frac{\gamma^{2}%
}{k^{\prime}-k}\Delta_{n}^{1-2\alpha^{\prime}}+\frac{6}{5}\frac{\gamma^{2}%
}{k^{\prime}-k}\Delta_{n}^{2-4\alpha^{\prime}}\right)  ((k^{\prime}%
-k)\Delta_{l}+2)\right) \\
&  \times\prod_{l=u+1}^{m}\exp\left(  \frac{11}{5}\frac{\gamma^{2}}{k^{\prime
}-k}\Delta_{n}^{1-2\alpha^{\prime}}\Delta_{l-u-2}^{1/2}+\frac{6}{5}%
\frac{\gamma^{2}}{k^{\prime}-k}\Delta_{n}^{2-4\alpha^{\prime}}\Delta
_{l-u-2}\right).
\end{align*}
As $(k^{\prime}-k)^{-1}\leq2^{-u}$,
\begin{align*}
&  \prod_{l=2}^{m}\prod_{r=1}^{2^{n+m-l}}C(t_{r}^{n+m-l})\\
\leq &  \exp\{ \frac{11}{5}\gamma^{2}\left(  \sum_{l=2}^{u}\Delta
_{l}+2(u-1)\Delta_{u}+\sum_{l=u+1}^{m}\Delta_{l-2}^{1/2}\right)\\
& +\frac{6}{5}\gamma^{2}\left(  \sum_{l=2}^{u}\Delta_{l}+2(u-1)\Delta_{u}+\sum
_{l=u+1}^{m}\Delta_{l-2}\right)  \} \\
\leq &  \exp\left(  \frac{1}{2}\right).
\end{align*}

For \eqref{Cont_Exp}, we notice that under condition \eqref{Cond_on_n} and
\eqref{Cond_on_n_2}, we have
\begin{align*}
\left\vert \sum_{r=1}^{2^{n}}\theta_{m}(t_{r}^{n})\Lambda_{i}(t_{r}%
^{n})\Lambda_{j}(t_{r}^{n})\right\vert  &  \leq\theta_{1}\Delta_{m-1}%
\varepsilon_{0}((k^{\prime}-k)\Delta_{m})^{\beta}\Delta_{n}^{2\alpha^{\prime}%
}+2\theta_{1}\Delta_{m-1}\Delta_{n}^{2\alpha^{\prime}}\\
&  \leq\varepsilon_{0}\gamma(k^{\prime}-k)^{\beta-1/2}+2\gamma,
\end{align*}
and
\begin{align*}
\left\vert \sum_{r=1}^{2^{n}}\eta_{m}(t_{r}^{n})\left(  \Lambda_{i}(t_{r}%
^{n})^{2}+\Lambda_{j}(t_{r}^{n})^{2}\right)  \right\vert  &  \leq\left(
(k^{\prime}-k)\Delta_{m}+2\right)  \eta_{1}\Delta_{m-1}2\Delta_{n}%
^{2\alpha^{\prime}}\\
&  \leq2\gamma^{2}.
\end{align*}
Combining the analysis for \eqref{Cont_Constants} and \eqref{Cont_Exp}, we
have
\begin{align*}
&E_{n}\exp(\theta_{0}\{L_{i,j}^{n+m}(k^{\prime})-L_{i,j}^{n+m}(k)\})\\
\leq& \exp\left(  \theta_{0}^{2}\Delta_{n+m}^{2}(k^{\prime}-k)+\frac{1}%
{2}+\varepsilon_{0}\gamma(k^{\prime}-k)^{\beta-1/2}+2\gamma+2\gamma^{2}\right)\\
\leq& 4\exp\left(  \varepsilon_{0}\gamma(k^{\prime}-k)^{\beta-1/2}\right)  .
\end{align*}
\end{proof}


\bibliographystyle{plainnat}
\bibliography{rough_path_ref}

\end{document}